\documentclass[12pt]{amsart}
\usepackage{amssymb,amsfonts,amsmath,amsthm,cite,verbatim}
\usepackage{xcolor}
\usepackage{tikz}
\usepackage{graphicx}
\usepackage[left=2.5cm, right=2.5cm, top=3.5cm, bottom=3.5cm]{geometry}

\usepackage{hyperref}

\usepackage[normalem]{ulem}
\newtheorem{theorem}{Theorem}[section]
\newtheorem{thm}[theorem]{Theorem}
\newtheorem{prop}[theorem]{Proposition}
\newtheorem{lem}[theorem]{Lemma}

\newtheorem{cor}[theorem]{Corollary}
\newtheorem{ex}[theorem]{Example}

\makeatletter \@addtoreset{equation}{section}

\newcommand{\red}[1]{{\color{red}  #1}}

\newcommand{\qbinom}[2]{\genfrac{[}{]}{0pt}{}{#1}{#2}}

\newcommand{\Mid}{\:|\:}  
\DeclareMathOperator*{\CT}{CT}

\newcommand{\CC}{\mathbb{C}}

\begin{document}

\title[Symmetric function generalizations of the $q$-Baker--Forrester ex-conjecture]
{Symmetric function generalizations of the $q$-Baker--Forrester ex-conjecture and Selberg-type integrals}

\author{Guoce Xin}
\address{School of Mathematical Sciences, Capital Normal University,
Beijing 100048, P.R. China}

\email{guoce\_xin@163.com}

\author{Yue Zhou*}

\thanks{*Corresponding author}

\address{School of Mathematics and Statistics, HNP-LAMA, Central South University,
Changsha 410075, P.R. China}

\email{zhouyue@csu.edu.cn}


\date{June 24, 2022}

\begin{abstract}
It is well-known that the famous Selberg integral is equivalent to the Morris constant term identity.
In 1998, Baker and Forrester conjectured a generalization of the $q$-Morris constant term identity.
This conjecture was proved and extended by K\'{a}rolyi, Nagy, Petrov and Volkov in 2015.
In this paper, we obtain two symmetric function generalizations of the $q$-Baker--Forrester
ex-conjecture. These includes: (i) a $q$-Baker--Forrester type constant term identity for a product of
a complete symmetric function and a Macdonald polynomial;
(ii) a complete symmetric function generalization of KNPV's result.

\end{abstract}

\maketitle

\section{Introduction}\label{s-intr}

In 1944, Atle Selberg \cite{Selberg} gave the following remarkable multiple integral:
\begin{multline*}
\int_{0}^1\cdots \int_{0}^1\prod_{i=1}^nx_{i}^{\alpha-1}(1-x_i)^{\beta-1}\prod_{1\leq i<j\leq n}
|x_i-x_j|^{2\gamma}\mathrm{d}x_{1}\cdots \mathrm{d}x_{n} \nonumber \\
=\prod_{j=0}^{n-1}\frac{\Gamma(\alpha+j\gamma)\Gamma(\beta+j\gamma)
\Gamma(1+(j+1)\gamma)}{\Gamma(\alpha+\beta+(n+j-1)\gamma)\Gamma(1+\gamma)},
\end{multline*}
where $\alpha, \beta, \gamma$ are complex parameters such that
\[
\mathrm{Re}(\alpha)>0, \quad \mathrm{Re}(\beta)>0,\quad
\mathrm{Re}(\gamma)>-\min\{1/n,\mathrm{Re}(\alpha)/(n-1),\mathrm{Re}(\beta)/(n-1)\}.
\]
When $n=1$, the above Selberg integral reduces to the Euler beta integral.
For the importance of the Selberg integral, one can refer \cite{FW}.

It is well-known that the Selberg integral is equivalent to the Morris constant term identity \cite{Morris1982}
\begin{equation}\label{Morris}
\CT_x \prod_{i=1}^{n}(1-x_0/x_i)^a(1-x_i/x_0)^b
\prod_{1\leq i\neq j\leq n}(1-x_i/x_j)^c
=\prod_{i=0}^{n-1}\frac{(a+b+ic)!\big((i+1)c\big)!}{(a+ic)!(b+ic)!c!}
\end{equation}
for nonnegative integers $a,b,c$,
where $\CT\limits_x $ denotes taking the constant term with respect to $x:=(x_1,\dots,x_n)$.
Note that we can set $x_0=1$ in \eqref{Morris} without changing the constant term, since we take the constant term in a homogeneous Laurent polynomial.
The same principle applies to all the homogeneous Laurent polynomials in this paper.
In his Ph.D. thesis \cite{Morris1982}, Morris also conjectured the following
$q$-analogue constant term identity
\begin{equation}\label{q-Morris}
\CT_x \prod_{i=1}^{n}(x_0/x_i)_a(qx_i/x_0)_b
\prod_{1\leq i<j\leq n}(x_i/x_j)_c(qx_j/x_i)_c
=\prod_{i=0}^{n-1}\frac{(q)_{a+b+ic}(q)_{(i+1)c}}{(q)_{a+ic}(q)_{b+ic}(q)_{c}},
\end{equation}
where $(z)_c=(z;q)_c:=(1-z)(1-zq)\cdots (1-z q^{c-1})$ is the $q$-factorial for a positive integer $c$ and
$(z)_0:=1$.
In 1988,  Habsieger \cite{Habsieger} and Kadell \cite{Kad} independently proved Askey's conjectured
$q$-analogue of the Selberg integral \cite{Askey}.
Expressing their $q$-analogue integral as a constant term identity
they thus proved Morris' $q$-constant term conjecture \eqref{q-Morris}.
Now the constant term identity \eqref{q-Morris} is called the Habsieger--Kadell $q$-Morris identity.

Recently, Albion, Rains and Warnaar obtained an AFLT type $q$-Selberg integral, which is a $q$-Selberg integral over a pair of two Macdonald polynomials \cite{ARW}. Their integral gives a $q$-analogue of the AFLT Selberg integral \cite{AFLT}, which arising from AGT conjecture.
By the standard transformation between $q$-Selberg type integrals and $q$-Morris type constant term identities \cite{Morris1982}, we obtain an equivalent constant term identity, see \eqref{AFLT} below.
For nonnegative integers $a,b,c$ and partitions $\lambda$ and $\mu$, denote
\begin{multline}\label{Defi-A}
A_n(a,b,c,\lambda,\mu):=\CT_x x_0^{-|\lambda|-|\mu|}P_{\lambda}(x_1,\dots,x_n;q,q^c)
P_{\mu}\Big(\Big[\frac{q^{c-b-1}-q^a}{1-q^c}x_0+\sum_{i=1}^nx_i\Big];q,q^c\Big)\\
\times
\prod_{i=1}^{n}(x_0/x_i)_a(qx_i/x_0)_b
\prod_{1\leq i<j\leq n}(x_i/x_j)_c(qx_j/x_i)_c,
\end{multline}
where $P_{\lambda}(x;q,t)$ is the Macdonald polynomial,
$|\lambda|:=\sum_{i=1}^{\infty}\lambda_i$ is the size of $\lambda$,
and $f[t+z]$ is plethystic notation for the symmetric function $f$.
Then, for $\ell(\lambda)\leq n$
\begin{align}\label{AFLT}
A_n(a,b,c,\lambda,\mu)&=(-1)^{|\lambda|}q^{\sum_{i=1}^n\binom{\lambda_i}{2}-cn(\lambda)}
P_{\lambda}\Big(\Big[\frac{1-q^{nc}}{1-q^c}\Big];q,q^c\Big)
P_{\mu}\Big(\Big[\frac{q^{c-b-1}-q^{a+nc}}{1-q^c}\Big];q,q^c\Big) \\
&\quad \times
\prod_{i=1}^n\prod_{j=1}^{\ell(\mu)}(q^{b+(n-i-j)c+\lambda_i+\mu_{j+1}+1})_{\mu_j-\mu_{j+1}}
\prod_{i=1}^n \frac{(q^{a+(i-1)c-\lambda_i+1})_{b+\lambda_i}(q)_{ic}}{(q)_{b+(n-i)c+\lambda_i+\mu_1}(q)_c},\nonumber
\end{align}
where $n(\lambda):=\sum_{i=1}^n(i-1)\lambda_i$ and $\ell(\lambda)$ is the length of the partition $\lambda$.
If $b=c-a-1$, the constant term identity \eqref{AFLT} reduces to \cite[Theorem 1.7]{Warnaar05}.
Note that if $\ell(\lambda)>n$ then both sides of \eqref{AFLT} reduces to zero.

For nonnegative integers $n,n_0,a,b,m$ and a positive integer $c$, denote
\begin{equation}\label{eq-qM}
F_{n,n_0}(x;a,b,c,m)=\prod_{i=1}^{n}(x_0/x_i)_a
(qx_i/x_0)_{b+\chi(i>n-m)}\prod_{1\leq i<j\leq n}
(x_i/x_j)_{c-\chi(i\leq n_0)}(qx_j/x_i)_{c-\chi(i\leq n_0)},
\end{equation}
where $\chi(true)=1$ and $\chi(false)=0$.
Intimately related to the theory of random matrices, Baker and Forrester \cite{BF} conjectured the following extension of the Habsieger--Kadell $q$-Morris identity:
\begin{multline}\label{q-BF}
\CT_x F_{n,n_0}(x;a,b,c,0)\\
=\prod_{i=1}^{n-n_0-1}(1-q^{(i+1)c})
\prod_{i=0}^{n-1} \frac{(q^{a+ic-n_0-\chi(i\leq n_0)(i-n_0)+1})_b(q)_{(i+1)c-n_0-\chi(i\leq n_0)(i-n_0)-1}}
{(q)_{b+ic-n_0-\chi(i\leq n_0)(i-n_0)}(q)_{c-\chi(i\leq n_0)}}.
\end{multline}
Using Combinatorial Nullstellensatz, K\'{a}rolyi, Nagy, Petrov and Volkov
\cite{KNPV} obtained a constant term identity for $F_{n,n_0}(x;a,b,c,m)$ (see the $l=0$ case of Theorem~\ref{thm-qForrester} below), which extends
the original $q$-Baker--Forrester conjecture \eqref{q-BF} by adding the new parameter $m$.

Let $X=(x_0,x_1,x_2,\dots)$ be an alphabet of countably many variables.
Then for nonnegative integer $l$, the $l$-th complete symmetric function $h_l(X)$
may be defined in terms of its generating function as
\begin{equation}\label{e-gfcomplete}
\sum_{l\geq 0} z^l h_l(X)=\prod_{i\geq 0}
\frac{1}{1-zx_i}.
\end{equation}
For $l$ a nonnegative integer and $\mu$ a partition, let
\begin{multline}\label{eq-qF2}
B_{n,n_0}(a,b,c,l,\mu):=\CT_x x_0^{-l-|\mu|}F_{n,n_0}(x;a,b,c,0)
\times h_{l}\Big[\sum_{i=1}^{n}\frac{1-q^{c-\chi(i\leq n_0)}}{1-q}x_i\Big] \\
\times P_{\mu}\Big(\Big[\frac{q^{c-b-1}-q^a}{1-q^c}x_0+\sum_{i=1}^{n}\frac{1-q^{c-\chi(i\leq n_0)}}{1-q^{c}}x_i\Big];q,q^c\Big).
\end{multline}
Our first result is the next $q$-Baker--Forrester type constant term identity for a product of a complete symmetric function and a Macdonald polynomial.
\begin{thm}\label{thm-qF2}
Let $B_{n,n_0}(a,b,c,l,\mu)$ be defined in \eqref{eq-qF2}.
Then for $0\leq n_0<n$ and $\ell(\mu)<n-n_0$,
\begin{align}\label{main-B}
B_{n,n_0}(a,b,c,l,\mu)
&=(-1)^{l+|\mu|}q^{\binom{l}{2}+\sum_{i=1}^n\binom{\mu_i}{2}-cn(\mu)}
h_{l}\Big[\frac{1-q^{nc-n_0}}{1-q}\Big]\times\frac{(q^{a-l+1})_l}{(q^{n_0(c-1)+b+1})_l} \\
&\quad \times P_{\mu}\Big(\Big[\frac{1-q^{a+b+(n-1)c-n_0+1}}{1-q^c}\Big];q,q^c\Big)\nonumber \\
&\quad \times\prod_{j=0}^{n-1} \frac{(q^{a+jc-n_0-\chi(j\leq n_0)(j-n_0)+1})_b
(q)_{(j+1)c-n_0-\chi(j\leq n_0)(j-n_0)-1}}
{(q)_{b+jc-n_0-\chi(j\leq n_0)(j-n_0)}(q)_{c-\chi(j\leq n_0)}}\nonumber\\
&\quad \times\prod_{j=1}^{n-n_0-1}\frac{(1-q^{(j+1)c})(q^{jc-b-\mu_j})_{\mu_j}(q^{(n-j-1)c-n_0+b+\mu_j+1})_l}
{(q^{(n-j)c+b+1-n_0})_{u_j+l}}.
\nonumber
\end{align}
\end{thm}
Note that if $n_0\geq n$ then this forces $\mu=0$ by the restriction $\ell(\mu)<n-n_0$.
Then it is easy to see that $B_{n,n_0}(a,b,c,l,0)=B_{n,0}(a,b,c-1,l,0)$ for $n_0\geq n$.

For $l$ a nonnegative integer, define
\begin{equation}\label{eq-qForrester}
C_{n,n_0}(a,b,c,l,m):=\CT_xx_0^{-l}h_{l}\Big[\sum_{i=1}^{n}\frac{1-q^{c-\chi(i\leq n_0)}}{1-q}x_i\Big]
F_{n,n_0}(x;a,b,c,m).
\end{equation}
Our second result is the next complete symmetric function generalization of the result of K\'{a}rolyi et al.
\begin{thm}\label{thm-qForrester}
Let $C_{n,n_0}(a,b,c,l,m)$ be defined in \eqref{eq-qForrester}.
For $m\geq n-n_0$,
\begin{multline}\label{main-second}
C_{n,n_0}(a,b,c,l,m)
=(-1)^lq^{\binom{l}{2}}h_l\Big[\frac{1-q^{nc-n_0}}{1-q}\Big]
\frac{(q^{a-l+1})_l}{(q^{(n-1)c-n_0+b+2})_l}\prod_{j=2}^{n-n_0}(1-q^{jc})\\
 \times \prod_{j=0}^{n-1}\frac{(q^{a+j(c-1)+\chi(j>n_0)(j-n_0)+1})_{b+\chi(j\geq n-m)}
(q)_{(j+1)(c-1)+\chi(j>n_0)(j-n_0)}}
{(q)_{b+j(c-1)+\chi(j>n_0)(j-n_0)+\chi(j\geq n-m)}(q)_{c-\chi(j\leq n_0)}}.
\end{multline}
\end{thm}
Note that by the definition of $C_{n,n_0}(a,b,c,l,m)$, for $m\geq n$ we have
$C_{n,n_0}(a,b,c,l,m)=C_{n,n_0}(a,b+1,c,l,0)=B_{n,n_0}(a,b+1,c,l,0)$.
One can also notice that $C_{n,n_0}(a,b,c,l,m)=C_{n,0}(a,b,c-1,l,m)$ for $n_0\geq n$.
The $l=0$ case of Theorem~\ref{thm-qForrester} reduces to the result of K\'{a}rolyi et al
\cite[Theorem 6.2]{KNPV}.

The method employed to prove Theorems \ref{thm-qF2} and \ref{thm-qForrester}
is based on the Gessel--Xin method, which first appeared in \cite{GX} to prove the Zeilberger--Bressoud $q$-Dyson theorem \cite{zeil-bres1985}. We managed to extend the Gessel--Xin method in the proof of the first-layer formulas for the $q$-Dyson product \cite{LXZ} and in dealing with the $q$-Dyson orthogonality problem \cite{Zhou}.
The basic idea of the Gessel--Xin method is the well-known fact that to prove the equality of two polynomials of degree at most $d$, it is sufficient to prove that they agree at $d+1$ distinct points.
We briefly outline the key steps to prove our main results \eqref{main-B} and \eqref{main-second}.
\vskip 0.2cm

\begin{enumerate}
\item \textbf{Polynomiality.}
It is routine to show that the constant terms
$B_{n,n_0}(a,b,c,l,\mu)$ and $C_{n,n_0}(a,b,c,l,m)$ (we refer these two constant terms as $B$ and $C$ for short respectively in the following of this section) are polynomials in $q^{a}$, assuming that all parameters but $a$ are fixed. See Corollary~\ref{cor-poly-BC}.
Then, we can extend the definitions of $B$ and $C$ for negative $a$, especially negative integers.
\vskip 0.2cm

\item \textbf{Rationality.}
By a rationality result, see Corollary~\ref{cor-rationality} below, it suffices to prove \eqref{main-B} and \eqref{main-second} for sufficiently large $c$.
\vskip 0.2cm

\item \textbf{Determination of roots.}
Let
\begin{align}\label{Roots-B}
B_1&=C_1=\{-ic+n_0-\chi(i\leq n_0)(n_0-i)-1,\dots,\\
&\qquad \qquad -ic+n_0-\chi(i\leq n_0)(n_0-i)-b\Mid i=0,1,\dots,n-1\},\nonumber\\
B_2&=C_2=\{l-1,l-2,\dots,0\}, \nonumber\\
B_3&=\{-(n-j)c+n_0-b-1,\dots, -(n-j)c+n_0-b-\mu_j \Mid j=1,2,\dots,\ell(\mu)\}. \nonumber
\end{align}
Suppose $c>b+\mu_1$, then all the elements of $B_1\cup B_2\cup B_3$ are distinct.
Let
\begin{equation}\label{Roots-C}
C_3=\{-ic+n_0-\chi(i\leq n_0)(n_0-i)-b-1\Mid i=n-m,\dots,n-1\}.
\end{equation}
Suppose $c>b+1$, then all the elements of $C_1\cup C_2\cup C_3$ are distinct.

For $B$ and $C$ viewed as polynomials in $q^a$, we will determine all their roots.
Explicitly, $B$ and $C$ vanish only when
$a\in B_1\cup B_2\cup B_3$ and $a\in C_1\cup C_2\cup C_3$ respectively.
Note that $B_1,B_3,C_1,C_3$ are sets of negative integers. For $B$ and $C$ at a negative integer $a$,
we are actually concerned with constant terms of rational functions. Hence, we need the theory of the field of iterated Laurent series, which was developed by Xin \cite{xiniterate,xinresidue}.

\item \textbf{Value at an additional point.}
We obtain explicit expressions for $B$ and $C$ at $a=-n_0(c-1)-b-1$ and $a=-(n-m-1)(c-1)-b-1$ respectively.
\end{enumerate}
\vskip 0.2cm

We can uniquely determine the closed-form expressions for $B$ and $C$ by the above steps.
The steps (1) and (2) are routine but not trivial. The step (3) is quite lengthy but conceptually simple.
In the step (4) we reduce $B$ and $C$ to similar types.
When carried out the details, we made a breakthrough to the Gessel--Xin method by mixing the ideas of the original Gessel--Xin method, plethystic substitutions and the splitting formula for the $q$-Dyson product by Cai \cite{cai}. We will explain this below.

Let the degree of a rational function of $x$ be the degree of the numerator in $x$ minus the degree in $x$ of the denominator. The original Gessel--Xin method \cite{GX} is a constant term method to rational functions with negative degrees. We extended it to rational functions with non-positive degrees
to obtain the first-layer formulas for the $q$-Dyson product in \cite{LXZ}.
We attempted to prove the Forrester conjecture in \cite{Gessel-Lv-Xin-Zhou2008}. However, the method failed in general. The main obstacle is dealing with constant terms of rational functions with positive degrees.
In this paper, we overcome this difficulty.

For the rational functions with positive degrees in this paper, we find that they reduce to Laurent polynomials under certain conditions, see Lemma~\ref{lem-QLaurentpoly}. Then, using Cai's idea, we obtain similar splitting formulas for the $q$-Baker--Forrester type constant terms, see Proposition~\ref{prop-split}.
By the splitting formulas, we can confirm whether the constant terms of those Laurent polynomials mentioned above vanish.
By such process, we extend the original Gessel--Xin method to rational functions with positive degrees.

To complete the step (3), we also need to combine the Gessel--Xin method with plethystic substitutions. We developed this method recently to deal with the $q$-Dyson orthogonality problem \cite{Zhou}.

The structure of this paper is as follows. The first section is this introduction.
In Section~\ref{sec-2}, we introduce the basic notation. In Section~\ref{sec-tools},
we introduce the three main tools in this paper --- plethystic notation, iterated Laurent series and a splitting formula. In Section~\ref{sec-Mac}, we give several essential results for Macdonald polynomials.
In Section~\ref{sec-poly}, we obtain the properties of polynomiality and rationality for $B_{n,n_0}(a,b,c,l,\mu)$ and $C_{n,n_0}(a,b,c,l,m)$. In Section~\ref{sec-preliminaries}, we present some preliminaries for the determination of the roots of $B_{n,n_0}(a,b,c,l,\mu)$ and $C_{n,n_0}(a,b,c,l,m)$.
In Section~\ref{sec-proof1} and Section~\ref{sec-proof2}, we complete the proof of Theorem~\ref{thm-qF2}
and Theorem~\ref{thm-qForrester} respectively.
In Section~\ref{sec-case3}, we complete the proof of Case (3) of Lemma~\ref{lem-Q}.

\section{Basic notation}\label{sec-2}

In this section we introduce some basic notation used throughout this paper.

A partition is a sequence $\lambda=({\lambda_1,\lambda_2,\dots)}$ of nonnegative integers such that
${\lambda_1\geq \lambda_2\geq \cdots}$ and
only finitely-many $\lambda_i$ are positive.
The length of a partition $\lambda$, denoted
$\ell(\lambda)$ is defined to be the number of non-zero $\lambda_i$ (such $\lambda_i$ are known as the parts of $\lambda$).
We adopt the convention of not displaying the
tails of zeros of a partition.
Sometimes, it is convenient to denote a partition by
\[
\lambda=(1^{m_1}2^{m_2}\cdots r^{m_r}\cdots),
\]
which means that exactly $m_i$ of the parts of $\lambda$ are equal to $i$.
As usual, we identify a partition with its Young diagram --- a collection of left-aligned rows of squares such that the $i$th row contains $\lambda_i$ squares. For example, the partition $(6,4,3,1)$ corresponds to
\begin{center}
\includegraphics[height=0.15\textwidth]{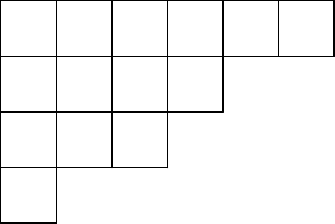}
\end{center}
The conjugate of a partition $\lambda$ is the partition $\lambda'$ whose diagram is the transpose of the diagram $\lambda$. For example, the conjugate of $(6,4,3,1)$
is $(4,3,3,2,1,1)$.
If $\lambda,\mu$ are partitions, we shall write $\mu \subset \lambda$ if $\mu_i\leq \lambda_i$
for all $i\geq 1$. We can construct a skew diagram $\lambda/\mu$ whenever $\mu \subset \lambda$
by removing the squares of $\mu$ from those of $\lambda$.
For example, if $\lambda=(6,4,3,1)$ and $\mu=(5,3,1)$, then
$\mu \subset \lambda$ and the skew diagram $\lambda/\mu$ is the following:
\begin{center}
\includegraphics[height=0.15\textwidth]{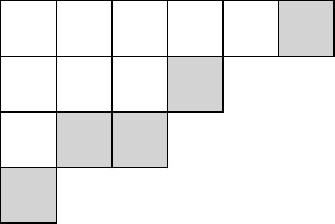}
\end{center}
A skew diagram $\theta$ is said to be a horizontal $r$-strip if $|\theta|=r$ and it contains at most one square in every column.
For example, the above diagram for $(6,4,3,1)/(5,3,1)$ is a horizontal 5-strip.
We say that $|\lambda|=\lambda_1+\lambda_2+\cdots$ is the size of the partition $\lambda$.
We adopt the standard dominance order on
the set of partitions of the same size.
If $\lambda,\mu$ are partitions such that $|\lambda|=|\mu|$ then $\mu\leq \lambda$ if
$\mu_1+\cdots+\mu_i\leq \lambda_1+\cdots+\lambda_i$ for all $i\geq 1$.
As usual, we write $\mu<\lambda$ if $\mu\leq \lambda$ but $\mu\neq \lambda$.

The infinite $q$-shifted factorial is defined as
\[
(z)_{\infty}=(z;q)_{\infty}:=
\prod_{i=0}^{\infty}(1-zq^i)
\]
where, typically, we suppress the base $q$.
Then, for $k$ an integer,
\[
(z)_k=(z;q)_k:=\frac{(z;q)_{\infty}}
{(zq^k;q)_{\infty}}.
\]
Note that
\[
(z)_k=
\begin{cases}
(1-z)(1-zq)\cdots (1-zq^{k-1}) & \text{if $k>0,$}
\\
1 & \text{if $k=0,$}\\
\displaystyle
\frac{1}{(1-zq^k)(1-zq^{k+1})\cdots(1-zq^{-1})}
& \text{if $k<0$.}
\end{cases}
\]
Using the above we can define the
$q$-binomial coefficient as
\[
\qbinom{n}{k}=\frac{(q^{n-k+1})_k}{(q)_k}
\]
for $n$ an arbitrary integer and $k$ a nonnegative integer.

By convention, for $k<j$ we set $\sum_{i=j}^k n_i:=0$ and $\prod_{i=j}^k n_i:=1$.

\section{Main tools}\label{sec-tools}

In this section, we introduce our three main tools --- plethystic notation, iterated Laurent series and a splitting formula in the subsections~\ref{sec-ple}, \ref{sec-iterate} and \ref{sec-splitting} respectively.
In Subsection~\ref{sec-splitting-conseq}, we give several consequences by the splitting formula.

\subsection{Plethystic notation and substitution}\label{sec-ple}

Plethystic or $\lambda$-ring notation is a device to facilitate computations in the ring of symmetric functions.
In this subsection, we briefly introduce plethystic notation and substitution.
The reader can refer \cite{haglund,Lascoux,RW} for more details.

Let $\mathbb{F}$ be a field and $\Lambda_{\mathbb{F}}$ be the ring of symmetric functions in countably many variables with coefficients in $\mathbb{F}$.
For a set of countably many variables $X=\{x_1,x_2,\dots\}$ (we refer $X$ as an alphabet),
we additively write $X:=x_1+x_2+\cdots$, and
use plethystic brackets to indicate this additive notation:
\[
f(X)=f(x_1,x_2,\dots)=f[x_1+x_2+\cdots]=f[X], \quad
\text{for $f\in \Lambda_{\mathbb{F}}$.}
\]

Let $p_r$ be the power sum
symmetric function in the alphabet $X$, defined by
\[
p_r=\sum_{i\geq 1}x_i^r
\]
for $r$ a positive integer and $p_0=1$.
For a partition $\lambda=(\lambda_1,\lambda_2,\dots)$, define
\[
p_{\lambda}=p_{\lambda_1}p_{\lambda_2}\cdots.
\]
The $p_r$ are algebraically independent over
$\mathbb{Q}$, and the $p_{\lambda}$ form a basis of $\Lambda_{\mathbb{Q}}$ \cite{Mac95}.
Hence, we can write
\[
\Lambda_{\mathbb{Q}}=\mathbb{Q}[p_1,p_2,\dots].
\]

A power sum whose argument is the sum, difference or Cartesian product
of two alphabets $X$ and $Y$ is defined as
\begin{subequations}\label{asm}
\begin{align}
\label{asm1}
p_r[X+Y]&=p_r[X]+p_r[Y], \\
p_r[X-Y]&=p_r[X]-p_r[Y], \\
p_r[XY]&=p_r[X]p_r[Y].\label{asm3}
\end{align}
\end{subequations}
Occasionally we need to use an ordinary minus sign in plethystic notation.
To distinguish this from a plethystic minus sign, we denote by $\epsilon$ the alphabet consisting of the single letter $-1$, so that for $f\in \Lambda_{\mathbb{F}}$
\[
f(-x)=f(-x_1,-x_2,\dots)=f[\epsilon x_1+\epsilon x_2+\cdots]=f[\epsilon X].
\]
Hence
\[
p_r[\epsilon X]=(-1)^rp_r[X], \quad p_r[-\epsilon X]=(-1)^{r-1}p_r[X].
\]
If $f$ is a homogeneous symmetric function of degree $k$ then
\begin{equation}\label{e-homo-sym}
f[aX]=a^kf[X]
\end{equation}
for a single-letter alphabet $a$.
In particular,
\begin{equation}\label{e-Mac3}
f[\epsilon X]=(-1)^kf[X].
\end{equation}
From \eqref{asm1}, we also have
\[
p_r[2X]:=p_r[X+X]=2p_r[X].
\]
We extend the above to any $k\in \mathbb{F}$ via
\begin{equation}
p_r[kX]=kp_r[X].
\end{equation}
Note that this leads to some notational ambiguities,
and whenever not clear from the context we will indicate if a symbol such as $a$ or $k$
represents a letter or a binomial element\footnote{In \cite[p. 32]{Lascoux} Lascoux refers to $k\in \mathbb{F}$ as a binomial element.}.
Suppose $t$ is a single-letter alphabet such that $|t|<1$. By \eqref{asm1} and \eqref{asm3}
\[
p_r\Big[\frac{X}{1-t}\Big]=p_r[X(1+t+t^2+\cdots)]=p_r[X]\sum_{k=0}^{\infty} p_r[t^k]=
p_r[X]\sum_{k=0}^{\infty} t^{kr}=\frac{p_r[X]}{1-t^r}.
\]
Hence, we have the next division rule:
\begin{equation}\label{division}
p_r\Big[\frac{X}{1-t}\Big]=\frac{p_r[X]}{1-t^r}.
\end{equation}
Note that we can not define the plethystic division by an arbitrary
alphabet.

The plethystic substitution greatly simplifies our computation of constant terms because of the following rules:

1. We can simplify and compute the alphabets in the bracket.
For instance $f[X-X]=f[0]=f(0,0,\dots)$ and $f[(1-q^k)/(1-q)] =f[1+q+\cdots +q^{k-1}]$ for
$k$ a nonnegative integer and any $f\in \Lambda_{\mathbb{F}}$, as this is easily checked for the $f=p_r$ case.

2. We can make the substitution inside the bracket if we substitute
an $x$(or $y$)-variable by a one-letter alphabet (or a product of finitely many one-letter alphabets). For instance, if we want to substitute $Y$ by $M$
in $f[X-Y]$, then we can check that
$p_r[X-Y] \big|_{Y\to M}= p_r[X]-p_r[Y]\big|_{Y\to M}=p_r[X]-p_r[M]=p_r[X-M]$.
In fact, we can summarize this rule simply by:
a variable can only be substituted by a one-letter alphabet
(or a product of finitely many one-letter alphabets), not a signed one-letter alphabet.

Let the elementary symmetric
function be defined as
\[
e_r=\sum_{1\leq i_1<\dots<i_r}x_{i_1}x_{i_2}\cdots x_{i_r}
\]
for $r$ a positive integer and $e_0=1$.
The next simple fact plays an important role in proving vanishing
properties of the constant terms $B_{n,n_0}(a,b,c,l,\mu)$ and $C_{n,n_0}(a,b,c,l,m)$.
By the definition of the elementary symmetric function,
\begin{equation}\label{e-vanish}
e_r[X]=0 \qquad \text{if $|X|<r$,}
\end{equation}
for $r$ a positive integer and $X$ an alphabet of finitely many variables.
Here $|X|$ denotes the cardinality of $X$.

Finally, we need the following two basic plethystic identities.
One can find proofs in \cite[Theorem 1.27]{haglund}.
\begin{prop}\label{Ple-basic}
Let $X$ and $Y$ be two alphabets. For $r$ a nonnegative integer,
\begin{align}\label{e-xy}
h_{r}[X+Y]&=\sum_{i=0}^rh_i[X]h_{r-i}[Y], \\
h_{r}[-X]&=(-1)^re_r[X].\label{e-he}
\end{align}
\end{prop}

\subsection{Constant term evaluations using iterated Laurent series}\label{sec-iterate}

In this subsection we present a basic lemma for extracting
constant terms from rational functions in the framework of the field of iterated Laurent series.

Throughout this paper we let $K=\CC(q)$ and work in the field of iterated
Laurent series $K\langle\!\langle x_n, x_{n-1},\dots,x_0\rangle\!\rangle
=K(\!(x_n)\!)(\!(x_{n-1})\!)\cdots (\!(x_0)\!)$.
Elements of $K\langle\!\langle x_n,x_{n-1},\dots,x_0\rangle\!\rangle$
are regarded first as Laurent series in $x_0$, then as
Laurent series in $x_1$, and so on.
In~\cite{GX}, the authors showed that the field $K\langle\!\langle x_n, x_{n-1},\dots,x_0\rangle\!\rangle$ is highly suitable for proving the $q$-Dyson style constant term identities.
The reader may refer \cite{xinresidue} and~\cite{xiniterate} for more detailed
account of the properties of this field.
The most applicable fact in what is to follow is that the field $K(x_0,\dots,x_n)$ of
rational functions in the variables $x_0,\dots,x_n$ with coefficients in $K$
forms a subfield of $K\langle\!\langle x_n, x_{n-1},\dots,x_0\rangle\!\rangle$, so that every rational function can be identified with its unique Laurent
series expansion.

The following series expansion of $1/(1-cx_i/x_j)$
for $c\in K\setminus \{0\}$ forms a key ingredient in our approach:
\[
\frac{1}{1-c x_i/x_j}=
\begin{cases} \displaystyle \sum_{l\geq 0} c^l (x_i/x_j)^l
& \text{if $i<j$}, \\[5mm]
\displaystyle -\sum_{l<0} c^l (x_i/x_j)^l
& \text{if $i>j$}.
\end{cases}
\]
Thus,
\begin{equation}
\label{e-ct}
\CT_{x_i} \frac{1}{1-c x_i/x_j} =
\begin{cases}
    1 & \text{if $i<j$}, \\
    0 & \text{if $i>j$}, \\
\end{cases}
\end{equation}
where, for $f\in K\langle\!\langle x_n, x_{n-1},\dots,x_0\rangle\!\rangle$,
we also use the notation $\displaystyle\CT_{x_i} f$ to denote taking the constant term
of $f$ with respect to $x_i$.
The constant term operators defined in $K\langle\!\langle x_n,\dots,x_0\rangle\!\rangle$ have the property of commutativity:
\[
\CT_{x_i} \CT _{x_j} f = \CT_{x_j} \CT_{x_i} f
\]
for any $i$ and $j$.
This means that we are taking constant term in a set of variables.
Then $\CT\limits_x$ is well-defined in this field.

The following lemma has appeared previously in~\cite{GX}.
It is a basic tool for
extracting constant terms from rational functions.
\begin{lem}\label{lem-almostprop}
For a positive integer $m$, let $p(x_k)$ be a Laurent polynomial
in $x_k$ of degree at most $m-1$ with coefficients in
$K\langle\!\langle x_n,\dots,x_{k-1},x_{k+1},\dots,x_0\rangle\!\rangle$.
Let $0\leq i_1\leq\dots\leq i_m\leq n$ such that all $i_r\neq k$,
and define
\begin{equation}\label{e-defF}
f=\frac{p(x_k)}{\prod_{r=1}^m (1-c_r x_k/x_{i_r})},
\end{equation}
where $c_1,\dots,c_m\in K\setminus \{0\}$ such that $c_r\neq c_s$ if $x_{i_r}=x_{i_s}$.
Then
\begin{equation}\label{e-almostprop}
\CT_{x_k} f=\sum_{\substack{r=1 \\[1pt] i_r>k}}^m
\big(f\,(1-c_rx_k/x_{i_r})\big)\Big|_{x_k=c_r^{-1}x_{i_r}}.
\end{equation}
\end{lem}
Note that we extended Lemma~\ref{lem-almostprop} to \cite[Lemma~3.1]{LXZ},
in which $p(x_k)$ is a Laurent polynomial in $x_k$ of degree at most $m$.
In this paper, we find a way to deal with the cases when the degree in $x_k$ of $p(x_k)$ is larger than $m$.
In our application of plethystic substitution, a crucial fact is that every $c_r$ is of the form $q^u$ for
$u\in \mathbb{Z}$. In other words, $q$ and $q^{-1}$ are treated as single letter alphabets in plethystic notation,
and each $c_r$ is a product of finite many $q$ or $q^{-1}$.

The next proposition presents an equivalence between two kinds of constant terms.
\begin{prop}\label{prop-equiv}
For a positive integer $n$, a nonnegative integer $c$ and $x:=(x_1,\dots,x_n)$, let $f(x)$ be a rational function invariant under any permutation of $x$. Then
\begin{equation}\label{e-equiv}
\CT_{x} f(x)\prod_{1\leq i<j\leq n}(x_i/x_j)_c(qx_j/x_i)_c
=\frac{1}{n!}\prod_{i=1}^{n-1}\frac{1-q^{(i+1)c}}{1-q^c}\times\CT_{x} f(x)\prod_{1\leq i\neq j\leq n} (x_i/x_j)_c.
\end{equation}
\end{prop}
\begin{proof}
For $w:=(w_1,\dots,w_n)\in \mathfrak{S}_n$ and any function $g(x)$, define
\[
w\circ g(x):=g(x_{w_1},\dots,x_{w_n}).
\]
Since the action of $w$ on any Laurent polynomial (or Laurent series) does not change the constant term, we have
\begin{multline}\label{e-equiv1}
\CT_{x} f(x)\prod_{1\leq i<j\leq n}(x_i/x_j)_c(qx_j/x_i)_c \\
\quad =\frac{1}{n!}\sum_{w\in \mathfrak{S}_n}\CT_{x} w\circ \bigg(f(x)\prod_{1\leq i\neq j\leq n} (x_i/x_j)_c
\prod_{1\leq i<j\leq n}\frac{1-q^cx_j/x_i}{1-x_j/x_i}\bigg).
\end{multline}
Because $f(x)\prod_{1\leq i\neq j\leq n} (x_i/x_j)_c$ is invariant under the action of $w$, the right-hand side of \eqref{e-equiv1} becomes
\[
\frac{1}{n!}\CT_{x} f(x)\prod_{1\leq i\neq j\leq n} (x_i/x_j)_c \cdot \sum_{w\in \mathfrak{S}_n}w\circ
\bigg(\prod_{1\leq i<j\leq n}\frac{1-q^cz_j/x_i}{1-x_j/x_i}\bigg),
\]
which is the right-hand side of \eqref{e-equiv} by the next identity \cite[III, (1.4)]{Mac95}.
\begin{equation}\label{e-equiv2}
\sum_{w\in \mathfrak{S}_n}w\circ
\bigg(\prod_{1\leq i<j\leq n}\frac{1-q^cx_j/x_i}{1-x_j/x_i}\bigg)
=\prod_{i=1}^{n-1}\frac{1-q^{(i+1)c}}{1-q^c}. \qedhere
\end{equation}
\end{proof}

\subsection{A splitting formula}\label{sec-splitting}

In this subsection, we give a splitting formula for $S_{n,n_0}(x)$ defined in \eqref{FD} below.
The idea first appeared in Cai's proof of Kadell's $q$-Dyson orthogonality conjecture \cite{cai}.

For nonnegative integers $n,n_0$ and a positive integer $c$, let
\begin{equation}\label{D}
D_{n,n_0}(x)=D_{n,n_0}(x,c):=\prod_{1\leq i<j\leq n}\big(x_i/x_j\big)_{c-\chi(i\leq n_0)}\big(qx_j/x_i\big)_{c-\chi(i\leq n_0)},
\end{equation}
where $x:=(x_1,x_2,\dots,x_n)$.
Recall the definition of $F_{n,n_0}(x;a,b,c,m)$ in \eqref{eq-qM}. It is clear that
\[
D_{n,n_0}(x,c)=F_{n,n_0}(x;0,0,c,0).
\]
Denote
\begin{equation}\label{FD}
S_{n,n_0}(x)=S_{n,n_0}(x,c):=D_{n,n_0}(x,c)\times
\prod_{i=1}^{n_0}(x_i/w)^{-1}_{c-1}
\prod_{i=n_0+1}^{n}(x_i/w)^{-1}_{c},
\end{equation}
where $w$ is a parameter such that
all terms of the form $q^ux_i/w$ in \eqref{FD}
satisfy $|q^ux_i/w|<1$. Hence,
\[
\frac{1}{1-q^ux_i/w}=\sum_{j\geq 0} (q^ux_i/w)^j.
\]
Then, it is easy to see that
\begin{equation}\label{e-relation}
\CT_x x^vD_{n,n_0}(x)=\CT_{x,w}x^vS_{n,n_0}(x),
\end{equation}
where $v:=(v_1,v_2,\dots,v_n)\in \mathbb{Z}^n$ and
$x^v$ denotes the monomial $x_1^{v_1}\cdots x_n^{v_n}$.
Note that if $|v|:=v_1+\cdots+v_n\neq 0$ then
$\CT\limits_x x^vD_{n,n_0}(x)=0$ by the homogeneity of $D_{n,n_0}(x)$.

We need the following simple result in the proof of the splitting formula mentioned above.
\begin{lem}
Let $i,j$ be nonnegative integers and $t$ be an integer. Then, for $0\leq t\leq j$
\begin{subequations}\label{e-ab}
\begin{equation}\label{prop-b1}
\frac{(1/y)_{i}(qy)_j}{(q^{-t}/y)_i}=q^{it}(q^{1-i}y)_t(q^{t+1}y)_{j-t};
\end{equation}
for $-1\leq t\leq j-1$
\begin{equation}\label{prop-b2}
\frac{(y)_j(q/y)_i}{(q^{-t}/y)_i}=q^{i(t+1)}(q^{-i}y)_{t+1}(q^{t+1}y)_{j-t-1};
\end{equation}
and for $0\leq t\leq j-1$
\begin{equation}\label{prop-c}
\frac{(y)_j(q/y)_i}{(q^{-t}/y)_{i+1}}=-yq^{(i+1)t}(q^{-i}y)_t(q^{t+1}y)_{j-t-1}.
\end{equation}
\end{subequations}
\end{lem}
Note that the $t=j$ case of \eqref{prop-b1} (taking $y\mapsto y/q$) is the standard fact in \cite[Equation~(I.13)]{GR}.
The condition $0\leq t\leq j-1$ in \eqref{prop-c} forces $j$ to be a positive integer.
\begin{proof}
For $0\leq t\leq j$,
\[
\frac{(1/y)_{i}(qy)_j}{(q^{-t}/y)_i}=\frac{(q^{i-t}/y)_t(qy)_j}{(q^{-t}/y)_t}
=\frac{(-1/y)^tq^{it-\binom{t+1}{2}}(q^{1-i}y)_t(qy)_j}{(-1/y)^{t}q^{-\binom{t+1}{2}}(qy)_t}
=q^{it}(q^{1-i}y)_t(q^{t+1}y)_{j-t}.
\]

Taking $y\mapsto y/q$ and $t\mapsto t+1$ in \eqref{prop-b1} yields
\eqref{prop-b2} for $-1\leq t\leq j-1$.

For $0\leq t\leq j-1$,
\[
\frac{(y)_j(q/y)_i}{(q^{-t}/y)_{i+1}}=
\frac{(y)_j(q^{i-t+1}/y)_t}{(q^{-t}/y)_{t+1}}=
\frac{(y)_j(-1/y)^tq^{it-\binom{t}{2}}(q^{-i}y)_t}{(-1/y)^{t+1}q^{-\binom{t+1}{2}}(y)_{t+1}}=
-yq^{(i+1)t}(q^{-i}y)_t(q^{t+1}y)_{j-t-1}. \qedhere
\]
\end{proof}

The rational function $S_{n,n_0}(x)$ admits the following partial fraction expansion.
We refer it as the splitting formula for $S_{n,n_0}(x)$.
\begin{prop}\label{prop-split}
Let $S_{n,n_0}(x)$ be defined in \eqref{FD}. Then
\begin{equation}\label{e-Fsplit}
S_{n,n_0}(x)=\sum_{i=1}^{n_0}\sum_{j=0}^{c-2}\frac{A_{ij}}{1-q^jx_i/w}
+\sum_{i=n_0+1}^{n}\sum_{j=0}^{c-1}\frac{B_{ij}}{1-q^jx_i/w},
\end{equation}
where
\begin{multline}\label{A}
A_{ij}=\frac{q^{(c-1)\big(j(n-1)+n_0-i\big)+j(n-n_0)}}{(q^{-j})_j(q)_{c-j-2}}
\prod_{l=1}^{i-1}\big(q^{2-c}x_i/x_l\big)_j\big(q^{j+1}x_i/x_l\big)_{c-j-1}
\times D_{n-1,n_0-1}(x^{(i)})
\\
\times\prod_{l=i+1}^{n_0}\big(q^{1-c}x_i/x_l\big)_{j+1}\big(q^{j+1}x_i/x_l\big)_{c-j-2}
\prod_{l=n_0+1}^n(-x_i/x_l)\big(q^{1-c}x_i/x_l\big)_j\big(q^{j+1}x_i/x_l\big)_{c-j-2},
\end{multline}
and
\begin{multline}\label{B}
B_{ij}=\frac{q^{(n-1)jc+(n-i)c-n_0j}}{(q^{-j})_j(q)_{c-j-1}}
\prod_{l=1}^{n_0}\big(q^{2-c}x_i/x_l\big)_j\big(q^{j+1}x_i/x_l\big)_{c-j-1}
\times D_{n-1,n_0}(x^{(i)})
\\
\times\prod_{l=n_0+1}^{i-1}\big(q^{1-c}x_i/x_l\big)_{j}\big(q^{j+1}x_i/x_l\big)_{c-j}
\prod_{l=i+1}^n\big(q^{-c}x_i/x_l\big)_{j+1}\big(q^{j+1}x_i/x_l\big)_{c-j-1},
\end{multline}
where $x^{(i)}:=(x_1,\dots,x_{i-1},x_{i+1},\dots,x_n)$.
\end{prop}
Note that the $A_{ij}$ and the $B_{ij}$ are polynomials in $x_i$. Furthermore, the degrees in $x_i$ of the $A_{ij}$ are at least $n-n_0$. If $c=1$ then the first double sum in \eqref{e-Fsplit} vanishes.
\begin{proof}
By partial fraction decomposition of $S_{n,n_0}(x)$ with respect to $w^{-1}$,
we can rewrite $S_{n,n_0}(x)$ as \eqref{e-Fsplit} and
\begin{equation}\label{Aij}
A_{ij}=S_{n,n_0}(x)(1-q^jx_i/w)|_{w=q^jx_i} \quad \text{for $i=1,\dots,n_0$ and $j=0,\dots,c-2$,}
\end{equation}
and
\begin{equation}\label{Bij}
B_{ij}=S_{n,n_0}(x)(1-q^jx_i/w)|_{w=q^jx_i} \quad \text{for $i=n_0+1,\dots,n$ and $j=0,\dots,c-1$.}
\end{equation}

Carrying out the substitution $w=q^jx_i$ in $S_{n,n_0}(x)(1-q^jx_i/w)$ for $i=1,\dots,n_0$ yields
\begin{multline}\label{A1}
A_{ij}=
\frac{1}{(q^{-j})_j(q)_{c-j-2}}
\prod_{l=1}^{i-1}\frac{(x_l/x_i)_{c-1}(qx_i/x_l)_{c-1}}{(q^{-j}x_l/x_i)_{c-1}}
\prod_{l=i+1}^{n_0}\frac{(x_i/x_l)_{c-1}(qx_l/x_i)_{c-1}}{(q^{-j}x_l/x_i)_{c-1}}
\\
\times
\prod_{l=n_0+1}^n\frac{(x_i/x_l)_{c-1}(qx_l/x_i)_{c-1}}{(q^{-j}x_l/x_i)_{c}}
\prod_{\substack{1\leq v<u\leq n\\v,u\neq i}}
(x_v/x_u)_{c-\chi(v\leq n_0)}(qx_u/x_v)_{c-\chi(v\leq n_0)}.
\end{multline}
Using \eqref{e-ab} with $(i,j,t,y)\mapsto (c-1,c-1,j,x_i/x_l)$,
we have
\begin{subequations}\label{e-B}
\begin{equation}\label{B1}
\frac{(x_l/x_i)_{c-1}(qx_i/x_l)_{c-1}}{(q^{-j}x_l/x_i)_{c-1}}
=q^{j(c-1)}\big(q^{2-c}x_i/x_l\big)_j\big(q^{j+1}x_i/x_l\big)_{c-j-1} \quad \text{for $j=0,\dots,c-1$},
\end{equation}
\begin{equation}\label{B2}
\frac{(x_i/x_l)_{c-1}(qx_l/x_i)_{c-1}}{(q^{-j}x_l/x_i)_{c-1}}
=q^{(j+1)(c-1)}\big(q^{1-c}x_i/x_l\big)_{j+1}\big(q^{j+1}x_i/x_l\big)_{c-j-2}
\quad \text{for $j=-1,\dots,c-2$},
\end{equation}
and
\begin{equation}
\frac{(x_i/x_l)_{c-1}(qx_l/x_i)_{c-1}}{(q^{-j}x_l/x_i)_{c}}=
-q^{cj}x_i/x_l(q^{1-c}x_i/x_l)_j(q^{j+1}x_i/x_l)_{c-j-2} \quad \text{for $j=0,\dots,c-2$},
\end{equation}
respectively.
\end{subequations}
Substituting \eqref{e-B} into \eqref{A1}, and using
\[
\prod_{\substack{1\leq v<u\leq n\\v,u\neq i}}
(x_v/x_u)_{c-\chi(v\leq n_0)}(qx_u/x_v)_{c-\chi(v\leq n_0)}=D_{n-1,n_0-1}(x^{(i)}) \quad \text{for $i=1,\dots,n_0$},
\]
we obtain \eqref{A}.

Carrying out the substitution $w=q^jx_i$ in $FD_{n,n_0}(x)(1-q^jx_i/w)$ for $i=n_0+1,\dots,n$ yields
\begin{multline}\label{BB}
B_{ij}=
\prod_{l=1}^{n_0}\frac{(x_l/x_i)_{c-1}(qx_i/x_l)_{c-1}}{(q^{-j}x_l/x_i)_{c-1}}
\prod_{l=n_0+1}^{i-1}\frac{(x_l/x_i)_{c}(qx_i/x_l)_{c}}{(q^{-j}x_l/x_i)_{c}}
\prod_{l=i+1}^n\frac{(x_i/x_l)_{c}(qx_l/x_i)_{c}}{(q^{-j}x_l/x_i)_{c}}
\\
\times
\frac{1}{(q^{-j})_j(q)_{c-j-1}}
\prod_{\substack{1\leq v<u\leq n\\v,u\neq i}}
(x_v/x_u)_{c-\chi(v\leq n_0)}(qx_u/x_v)_{c-\chi(v\leq n_0)}.
\end{multline}
Taking $c\mapsto c+1$ in \eqref{B1} and \eqref{B2} respectively, we have
\begin{equation}\label{B3}
\frac{(x_l/x_i)_{c}(qx_i/x_l)_{c}}{(q^{-j}x_l/x_i)_{c}}
=q^{jc}\big(q^{1-c}x_i/x_l\big)_j\big(q^{j+1}x_i/x_l\big)_{c-j}
\quad \text{for $j=0,\dots,c$},
\end{equation}
and
\begin{equation}\label{B4}
\frac{(x_i/x_l)_{c}(qx_l/x_i)_{c}}{(q^{-j}x_l/x_i)_{c}}
=q^{(j+1)c}\big(q^{-c}x_i/x_l\big)_{j+1}\big(q^{j+1}x_i/x_l\big)_{c-j-1}
\quad \text{for $j=-1,\dots,c-1$}.
\end{equation}
Substituting \eqref{B1}, \eqref{B3} and \eqref{B4} into \eqref{BB}, and using
\[
\prod_{\substack{1\leq v<u\leq n\\v,u\neq i}}
(x_v/x_u)_{c-\chi(v\leq n_0)}(qx_u/x_v)_{c-\chi(v\leq n_0)}=D_{n-1,n_0}(x^{(i)}) \quad \text{for $i=n_0+1,\dots,n$},
\]
we obtain \eqref{B}.
\end{proof}

\subsection{Consequences of the splitting formula}\label{sec-splitting-conseq}

In this subsection, we obtain several vanishing constant terms by the splitting formula for $S_{n,n_0}(x)$.

We begin with the next vanishing lemma.
\begin{lem}\label{lem-positiveroots}
Let $D_{n,n_0}(x;c)$ be defined in \eqref{D},
$v=(v_1,\dots,v_n)\in \mathbb{Z}^n$, and $\lambda$ be a partition such that $|v|=|\lambda|$.
If $\lambda_1>\max\{v_i\mid i=1,\dots,n\}$, then
\begin{equation}
\CT_x x^{-v}h_{\lambda}\Big[\sum_{t=1}^{n}\frac{1-q^{c-\chi(t\leq n_0)}}{1-q}x_t\Big]
D_{n,n_0}(x;c)=0.
\end{equation}
\end{lem}
\begin{proof}
By the generating function for complete symmetric functions \eqref{e-gfcomplete},
\[
h_{l}\Big[\sum_{t=1}^{n}\frac{1-q^{c-\chi(t\leq n_0)}}{1-q}x_t\Big]
=\CT_w w^l\prod_{i=1}^n(x_i/w)^{-1}_{c-\chi(i\leq n_0)}
\]
for a nonnegative integer $l$.
Thus
\begin{multline}\label{pr-1}
\CT_x x^{-v}h_{\lambda}\Big[\sum_{t=1}^{n}\frac{1-q^{c-\chi(t\leq n_0)}}{1-q}x_t\Big]
D_{n,n_0}(x;c)\\
=\CT_{x,w}\frac{w^{\lambda_1}h_{\lambda^{(1)}}\big[\sum_{t=1}^{n}
\frac{1-q^{c-\chi(t\leq n_0)}}{1-q}x_t\big]D_{n,n_0}(x;c)}{x^v\prod_{i=1}^n(x_i/w)_{c-\chi(i\leq n_0)}},
\end{multline}
where $\lambda^{(1)}=(\lambda_2,\lambda_3,\dots)$.
By the splitting formula \eqref{e-Fsplit} for $S_{n,n_0}(x)$,
\begin{multline}\label{pr-2}
\CT_{x,w}\frac{w^{\lambda_1}
h_{\lambda^{(1)}}\big[\sum_{t=1}^{n}\frac{1-q^{c-\chi(t\leq n_0)}}{1-q}x_t\big]
D_{n,n_0}(x;c)}{x^v\prod_{i=1}^n(x_i/w)_{c-\chi(i\leq n_0)}}
=\CT_{x,w}\sum_{i=1}^{n_0}\sum_{j=0}^{c-2}\frac{w^{\lambda_1}
h_{\lambda^{(1)}}\big[\sum_{t=1}^{n}\frac{1-q^{c-\chi(t\leq n_0)}}{1-q}x_t\big]
A_{ij}}{x^v(1-q^jx_i/w)}\\
+\CT_{x,w}\sum_{i=n_0+1}^{n}\sum_{j=0}^{c-1}\frac{w^{\lambda_1}
h_{\lambda^{(1)}}\big[\sum_{t=1}^{n}\frac{1-q^{c-\chi(t\leq n_0)}}{1-q}x_t\big]
B_{ij}}
{x^v(1-q^jx_i/w)}.
\end{multline}
Here the $A_{ij}$ and the $B_{ij}$ are polynomials in $x_i$.
Taking the constant term with respect to $w$ in the right-hand side of \eqref{pr-2} yields
\begin{align}
&\CT_{x,w}\frac{w^{\lambda_1}
h_{\lambda^{(1)}}\big[\sum_{t=1}^{n}\frac{1-q^{c-\chi(t\leq n_0)}}{1-q}x_t\big]D_{n,n_0}(x;c)}
{x^v\prod_{i=1}^n(x_i/w)_{c-\chi(i\leq n_0)}}\label{pr-3}\\
&\quad=\CT_{x}\sum_{i=1}^{n_0}\sum_{j=0}^{c-2}q^{j\lambda_1}x^{-v}x_i^{\lambda_1}
h_{\lambda^{(1)}}\Big[\sum_{t=1}^{n}\frac{1-q^{c-\chi(t\leq n_0)}}{1-q}x_t\Big]A_{ij}\nonumber\\
&\quad+\CT_{x}\sum_{i=n_0+1}^{n}\sum_{j=0}^{c-1}q^{j\lambda_1}x^{-v}x_i^{\lambda_1}
h_{\lambda^{(1)}}\Big[\sum_{t=1}^{n}\frac{1-q^{c-\chi(t\leq n_0)}}{1-q}x_t\Big]
B_{ij}.\nonumber
\end{align}
Since $\max\{v\}<\lambda_1$, and $A_{ij}$ and $B_{ij}$ are polynomials in $x_i$, every term in the sums of the right-hand side of \eqref{pr-3} is a polynomial in $x_i$ with no constant term.
Thus
\[
\CT_{x,w}\frac{w^{\lambda_1}
h_{\lambda^{(1)}}\big[\sum_{t=1}^{n}\frac{1-q^{c-\chi(t\leq n_0)}}{1-q}x_t\big]D_{n,n_0}(x;c)}
{x^v\prod_{i=1}^n(x_i/w)_{c-\chi(i\leq n_0)}}=0.
\]
Together with \eqref{pr-1}, the lemma follows.
\end{proof}

A consequence of Lemma~\ref{lem-positiveroots} is the next result.
\begin{lem}\label{lem-Roots-2}
For nonnegative integers $a,m,c,l$ such that $c,l\neq 0$,
let
\begin{equation}
P_{n,n_0}(a,c,l,m;H):=x_0^{-l-m}h_{l}\Big[\sum_{i=1}^{n}\frac{1-q^{c-\chi(i\leq n_0)}}{1-q}x_i\Big] H\times
\prod_{i=1}^n(x_0/x_i)_a \times D_{n,n_0}(x;c),
\end{equation}
where $H$ is a homogenous polynomial in $x_0,x_1,\dots,x_n$ of degree $m$, and $D_{n,n_0}(x;c)$ is defined in \eqref{D}. Then $\CT\limits_{x} P_{n,n_0}(a,c,l,m;H)=0$ for $a=0,1,\dots,l-1$.
\end{lem}
Note that it is easy to see that $B_{n,n_0}(a,b,c,l,\mu)=C_{n,n_0}(a,b,c,l,m)=0$ for $a\in \{0,\dots,l-1\}$
by Lemma~\ref{lem-Roots-2}.
\begin{proof}
We first show the $a=0$ case. If $a=0$ then we can write
\begin{equation*}
P_{n,n_0}(0,c,l,m;H)=x_0^{-l-m}h_{l}\Big[\sum_{i=1}^{n}\frac{1-q^{c-\chi(i\leq n_0)}}{1-q}x_i\Big] H\times D_{n,n_0}(x;c).
\end{equation*}
It is clear that $P_{n,n_0}(0,c,l,m;H)$ is a Laurent polynomial in $x_0$ of degree at most $-l<0$.
Thus
\[
\CT_{x} P_{n,n_0}(0,c,l,m;H)=\CT_{x_0} P_{n,n_0}(0,c,l,m;H)=0.
\]

In the following of the proof, we assume that $l>1$. Otherwise, if $l=1$ then we only need to show
that $\CT\limits_{x} P_{n,n_0}(0,c,l,m;H)=0$ and we have completed it in the above.
For $a>0$, expanding the factor $H\prod_{i=1}^n(x_0/x_i)_a$ of
$P_{n,n_0}(a,c,l,m;H)$ and taking the constant term with respect to $x_0$ gives
\begin{equation}\label{expand-factor}
\CT_{x} P_{n,n_0}(a,c,l,m;H)=\sum_v\CT_x x^{-v}
h_{l}\Big[\sum_{i=1}^{n}\frac{1-q^{c-\chi(i\leq n_0)}}{1-q}x_i\Big]
D_{n,n_0}(x;c),
\end{equation}
where $v:=(v_1,\dots,v_n)\in \mathbb{Z}^n$, $x^v=x_1^{v_1}\cdots x_n^{v_n}$, and the sum
is over a finite set of $v$ such that $|v|=l$ and $\max\{v\}\leq a$.
We omit the explicit range of $v$ in \eqref{expand-factor} since we do not need it here.
Then, by Lemma~\ref{lem-positiveroots} with $\ell(\lambda)=1$ every constant term in the sum of \eqref{expand-factor} vanishes for $a=1,\dots,l-1$.
It follows that $\CT\limits_{x} P_{n,n_0}(a,c,l,m;H)=0$ for $a=1,\dots,l-1$.
\end{proof}

In the following of this subsection, we find that some monomials do not appear in the expansion of $D_{n,n_0}(x,c)$.

Denote
\begin{equation}\label{T}
T_{n,n_0}(x)=T_{n,n_0}(x,c):=D_{n,n_0}(x,c)\times
\prod_{i=1}^{n_0}(y/x_i)^{-1}_{c-1}
\prod_{i=n_0+1}^{n}(q^{-1}y/x_i)^{-1}_{c},
\end{equation}
where $y$ is a parameter such that
all terms of the form $q^uy/x_i$ in \eqref{T}
satisfy $|q^uy/x_i|<1$. Hence,
\[
\frac{1}{1-q^uy/x_i}=\sum_{j\geq 0} (q^uy/x_i)^j.
\]
Then, it is easy to see that
\begin{equation}\label{e-relation-2}
\CT_x x^vD_{n,n_0}(x)=\CT_{y,x}x^vT_{n,n_0}(x),
\end{equation}
where $v:=(v_1,v_2,\dots,v_n)\in \mathbb{Z}^n$ and
$x^v$ denotes the monomial $x_1^{v_1}\cdots x_n^{v_n}$.
The rational function $T_{n,n_0}(x)$ admits the following partial fraction expansion.
\begin{prop}\label{prop-split2}
Let $T_{n,n_0}(x)$ be defined in \eqref{T}. Then
\begin{equation}\label{e-Fsplit2}
T_{n,n_0}(x)=\sum_{i=1}^{n_0}\sum_{j=0}^{c-2}\frac{A_{ij}}{1-q^jy/x_i}
+\sum_{i=n_0+1}^{n}\sum_{j=-1}^{c-2}\frac{B_{ij}}{1-q^jy/x_i},
\end{equation}
where
\begin{multline}\label{A-2}
A_{ij}=\frac{1}{(q^{-j})_j(q)_{c-j-2}}
\prod_{l=1}^{i-1}q^{(c-1)(j+1)}\big(q^{1-c}x_l/x_i\big)_{j+1}\big(q^{j+1}x_l/x_i\big)_{c-j-2}
\times D_{n-1,n_0-1}(x^{(i)})
\\
\times\prod_{l=i+1}^{n_0}q^{(c-1)j}\big(q^{2-c}x_l/x_i\big)_j\big(q^{j+1}x_l/x_i\big)_{c-j-1}
\prod_{l=n_0+1}^n(-x_l/x_i)q^{cj+1}(q^{2-c}x_l/x_i)_j(q^{j+2}x_l/x_i)_{c-j-2},
\end{multline}
and
\begin{multline}\label{B-2}
B_{ij}=\frac{1}{(q^{-j-1})_{j+1}(q)_{c-j-2}}
\prod_{l=1}^{n_0}q^{(c-1)(j+1)}\big(q^{1-c}x_l/x_i\big)_{j+1}\big(q^{j+1}x_l/x_i\big)_{c-j-2}
\times D_{n-1,n_0}(x^{(i)})
\\
\times\prod_{l=n_0+1}^{i-1}q^{c(j+2)}\big(q^{-c}x_l/x_i\big)_{j+2}\big(q^{j+2}x_l/x_i\big)_{c-j-2}
\prod_{l=i+1}^nq^{c(j+1)}\big(q^{1-c}x_l/x_i\big)_{j+1}\big(q^{j+2}x_l/x_i\big)_{c-j-1},
\end{multline}
where $x^{(i)}:=(x_1,\dots,x_{i-1},x_{i+1},\dots,x_n)$.
\end{prop}
Note that the $A_{ij}$ and the $B_{ij}$ are polynomials in $x_i^{-1}$. Furthermore, the degrees in $x_i^{-1}$ of the $A_{ij}$ are at least $n-n_0$. If $c=1$ then the first double sum in \eqref{e-Fsplit2} vanishes.
The proof of the proposition is similar to Proposition~\ref{prop-split}. Hence, we omit the proof.

Using the splitting formula~\eqref{e-Fsplit2}, we can prove the next simple example.
\begin{ex}\label{lem-vanfund-2}
For a positive integer $c$,
\begin{equation}
\CT_x\frac{x_1x_2}{x_3x_4}
\times D_{4,2}(x,c)=0.
\end{equation}
\end{ex}
\begin{proof}
Recall the definition of $T_{n,n_0}(x)$ in \eqref{T}.
By \eqref{e-relation-2},
\[
\CT_x\frac{x_1x_2}{x_3x_4}
\times D_{4,2}(x)=\CT_{y,x}\frac{x_1x_2}{x_3x_4}\times T_{4,2}(x).
\]
Using the splitting formula \eqref{e-Fsplit2}
for $T_{n,n_0}(x)$ with $n=4,n_0=2$, we have
\begin{equation}\label{fund-1-2}
\CT_x\frac{x_1x_2}{x_3x_4}\times D_{4,2}(x)
=\CT_{y,x}\frac{x_1x_2}{x_3x_4}
\times\bigg(
\sum_{i=1}^{2}\sum_{j=0}^{c-2}\frac{A_{ij}}{1-q^jy/x_i}
+\sum_{i=3}^{4}\sum_{j=-1}^{c-2}\frac{B_{ij}}{1-q^jy/x_i}\bigg).
\end{equation}
Extracting out the constant term with respect to $y$ in the right-hand side of \eqref{fund-1-2} gives
\[
\CT_x\frac{x_1x_2}{x_3x_4}\times D_{4,2}(x)
=\CT_{x}\frac{x_1x_2}{x_3x_4}
\times\bigg(
\sum_{i=1}^{2}\sum_{j=0}^{c-2}A_{ij}
+\sum_{i=3}^{4}\sum_{j=-1}^{c-2}B_{ij}\bigg).
\]
Since $A_{ij}$ and $B_{ij}$ are polynomials in $x_i^{-1}$ and the degree in $x_i^{-1}$ of $A_{ij}$
is at least 2,
\begin{align*}
\CT_{x}\frac{x_1x_2}{x_3x_4}A_{ij}&=0,\quad \text{for $i=1,2$ and $j=0,\dots,c-2$}, \\
\CT_{x}\frac{x_1x_2}{x_3x_4}B_{ij}&=0, \quad \text{for $i=3,4$ and $j=-1,\dots,c-2$}.
\end{align*}
Then
\[
\CT_x\frac{x_1x_2}{x_3x_4}\times D_{4,2}(x)=0. \qedhere
\]
\end{proof}
By the splitting formula \eqref{e-Fsplit2} for $T_{n,n_0}(x)$, we find that a family of coefficients
in $D_{n,n_0}(x,c)$ vanish.
\begin{lem}\label{lem-vaniconst2}
For $2\leq n_0\leq n-1$ and an integer $h$ such that $1\leq h\leq n_0-1$,
\begin{equation}\label{V0-2}
\CT_x\frac{x_1x_2\cdots x_{n_0}}{x_{n_0+1}^hx_{n_0+2}^h\cdots x_n^h}\prod_{l=1}^nx_l^{t_l}
\times D_{n,n_0}(x,c)=0,
\end{equation}
where $t_1,t_2,\dots,t_n$ are nonnegative integers such that
$\sum_{l=1}^nt_l=h(n-n_0)-n_0$.
\end{lem}
\begin{proof}
Denote by $L$ the left-hand side of \eqref{V0-2}.
If $c=1$ then $L$ reduces to
\[
\CT_x\frac{x_1x_2\cdots x_{n_0}}{x_{n_0+1}^hx_{n_0+2}^h\cdots x_n^h}
\prod_{l=1}^nx_l^{t_l}\times
\prod_{n_0+1\leq i<j\leq n}\big(1-x_i/x_j\big)\big(1-qz_j/x_i\big),
\]
which is clearly 0.
Then we prove $L=0$ for $c\geq 2$ by the induction on $n$.
The induction basis holds by Example~\ref{lem-vanfund-2}.
We assume that the lemma holds for $n\mapsto n-1$.

By \eqref{e-relation-2}, we have
\[
L=\CT_{x,w}\frac{x_1x_2\cdots x_{n_0}}{x_{n_0+1}^hx_{n_0+2}^h\cdots x_n^h}
\prod_{l=1}^nx_l^{t_l}\times
T_{n,n_0}(x).
\]
Using the splitting formula \eqref{e-Fsplit2} for $T_{n,n_0}(x)$ yields
\begin{equation*}
L=\CT_{y,x}\frac{x_1x_2\cdots x_{n_0}}{x_{n_0+1}^hx_{n_0+2}^h\cdots x_n^h}
\prod_{l=1}^nx_l^{t_l}\times\bigg(
\sum_{i=1}^{n_0}\sum_{j=0}^{c-2}\frac{A_{ij}}{1-q^jy/x_i}
+\sum_{i=n_0+1}^{n}\sum_{j=-1}^{c-2}\frac{B_{ij}}{1-q^jy/x_i}\bigg).
\end{equation*}
Extracting the constant term with respect to $y$ gives
\begin{equation}\label{V1}
L=\CT_{x}\frac{x_1x_2\cdots x_{n_0}}{x_{n_0+1}^hx_{n_0+2}^h\cdots x_n^h}
\prod_{l=1}^nx_l^{t_l}\times\bigg(
\sum_{i=1}^{n_0}\sum_{j=0}^{c-2}A_{ij}
+\sum_{i=n_0+1}^{n}\sum_{j=-1}^{c-2}B_{ij}\bigg).
\end{equation}
Let
\begin{equation}\label{C1}
C_1:=\CT_x\frac{x_1\cdots \hat{x_i}\cdots x_{n_0}}{x_{n_0+1}^h\cdots x_n^h}
\times x_i^{t_i+1}\prod_{\substack{l=1\\l\neq i}}^nx_l^{t_l}\times A_{ij}
\end{equation}
for $i=1,\dots,n_0$ and $j=0,\dots,c-2$,
and
\begin{equation}\label{C2}
C_2:=\CT_x\frac{x_1\cdots x_{n_0}}{x_{n_0+1}^h\cdots \hat{x_i}\cdots x_n^h}
\times x_i^{t_i-h}\prod_{\substack{l=1\\l\neq i}}^nx_l^{t_l}\times B_{ij}
\end{equation}
for $i=n_0+1,\dots,n$ and $j=-1,\dots,c-2$, where $\hat{x}_i$ means the omission of $x_i$. Note that we suppress the parameters in $C_1$ and
$C_2$. By \eqref{V1}, it is clear that $L$ is a finite sum of
constant terms of the forms of $C_1$ and $C_2$.
We show that $C_1=C_2=0$.

Notice that $A_{ij}$ is a polynomial in $x_i^{-1}$ of degree at least $n-n_0$.
If $t_i+1<n-n_0$ then $C_1=0$ since the Laurent polynomial in \eqref{C1} is actually a polynomial in $x_i^{-1}$ with no constant term.
If $h=1$ then
\[
t_i+1\leq \sum_{l=1}^nt_l+1=n-2n_0+1<n-n_0
\]
for $n_0>1$. Then $C_1=0$ by the same reason as above. Note that if $n_0=2$ then this forces $h=1$ and we also have  $C_1=0$.
The remainder cases for $C_1$ satisfy $n_0>2$, $t_i+1\geq n-n_0$ and $1<h\leq n_0-1$.
By the expression for $A_{ij}$ in \eqref{A-2} and taking the constant term with respect to $x_i$, $C_1$ can be written as a finite sum of the form
\begin{equation}\label{C1-1}
c\times\CT_{x^{(i)}}\frac{x_1\cdots \hat{x_i}\cdots x_{n_0}}{x_{n_0+1}^{h-1}\cdots x_n^{h-1}}
\times \prod_{\substack{l=1\\l\neq i}}^nx_l^{t'_l}\times D_{n-1,n_0-1}(x^{(i)}).
\end{equation}
Here $c\in K$ and the $t'_l$ are nonnegative integers such that
$\sum_{\substack{l=1\\l\neq i}}^nt'_l=(n-n_0)(h-1)-n_0+1$.
Since $1\leq h-1\leq n_0-2$ for $1<h\leq n_0-1$, the constant term \eqref{C1-1} equals 0 by the induction hypothesis.

Since $B_{ij}$ is a polynomial in $x_i^{-1}$,
if $t_i-h<0$ then $C_2=0$ since the Laurent polynomial in \eqref{C2} is in fact a polynomial in $x_i^{-1}$ with no constant term.
If $t_i-h\geq 0$,  by the expression for $B_{ij}$ in \eqref{B-2} and taking the constant term with respect to $x_i$, $C_2$ can be written as a finite sum of the form
\begin{equation}\label{C2-1}
c\times\CT_{x^{(i)}}\frac{x_1\cdots x_{n_0}}{x_{n_0+1}^h\cdots \hat{x_i}\cdots x_n^h}
\times \prod_{\substack{l=1\\l\neq i}}^nx_l^{t'_l}\times D_{n-1,n_0}(x^{(i)}).
\end{equation}
Here $c\in K$ and the $t'_j$ are nonnegative integers such that
$\sum_{\substack{l=1\\l\neq i}}^nt'_l=(n-n_0-1)h-n_0$.
The constant term \eqref{C2-1} equals 0 by the induction hypothesis for $n_0<n-1$.
For $n_0=n-1$, that the homogeneity of the Laurent polynomial in \eqref{C2} implies $h>t_i$.
This contradicts the assumption $t_i-h\geq 0$.

In conclusion, we obtain that all the forms of $C_1$ and $C_2$ defined in \eqref{C1} and \eqref{C2} respectively are zeros. Then $L=0$, completing the proof.
\end{proof}

\section{Macdonald polynomials}\label{sec-Mac}

In this section, we give several essential results for Macdonald polynomials \cite{MacSMC,Mac95}.

Given a sequence $\alpha=(\alpha_1,\alpha_2,\dots)$ of nonnegative integers such that $|\alpha|=\alpha_1+\alpha_2+\cdots$ is finite.
Let $X=\{x_1,x_2,\dots\}$ be an alphabet of countably many variables. Then the monomial symmetric function indexed by a partition $\lambda$ is defined as
\[
m_{\lambda}=m_{\lambda}(X):=\sum x_1^{\alpha_1}x_2^{\alpha_2}\cdots,
\]
where the sum is over all distinct permutations $\alpha$ of $\lambda$.
Denote $\mathbb{F}=\mathbb{Q}(q,t)$.
Let $\langle\cdot,\cdot\rangle:\Lambda_{\mathbb{F}}\times \Lambda_{\mathbb{F}} \rightarrow \mathbb{F}$
be the $q,t$-Hall scalar product on $\Lambda_{\mathbb{F}}$ given by \cite[page 306]{Mac95}
\[
\langle p_{\lambda}, p_{\mu}\rangle:=\delta_{\lambda,\mu}z_{\lambda}\prod_{i\geq 1}\frac{1-q^{\lambda_i}}{1-t^{\lambda_i}},
\]
where $z_{\lambda}:=\prod_{i\geq 1}i^{m_i}m_i!$ if $m_i$ is the number of parts $i$ in $\lambda$,
and the Kronecker symbol $\delta_{\lambda,\mu}=1$ if $\lambda=\mu$ and 0 otherwise.
The Macdonald polynomials
$P_{\lambda}=P_{\lambda}(q,t)=P_{\lambda}(X;q,t)$ are the unique symmetric functions \cite[VI, (4.7)]{Mac95} such that
\[
\langle P_{\lambda},P_{\mu}\rangle=0 \quad \text{if} \quad \lambda\neq \mu
\]
and
\begin{equation}\label{Mac-m}
P_{\lambda}=m_{\lambda}+\sum_{\mu<\lambda}u_{\lambda\mu}m_{\mu}, \quad u_{\lambda\mu}\in \mathbb{F}.
\end{equation}
Let
\[
b_{\lambda}=b_{\lambda}(q,t)=\langle P_{\lambda},P_{\lambda}\rangle^{-1},
\]
and
\begin{equation}\label{PQ}
Q_{\lambda}=b_{\lambda}P_{\lambda},
\end{equation}
so that
\[
\langle P_{\lambda}, Q_{\mu}\rangle=\delta_{\lambda \mu}
\]
for all partitions $\lambda$ and $\mu$. See \cite[VI, (4.11), (4.12)]{Mac95}.
For any three partitions $\lambda,\mu,\nu$ let
\[
f_{\mu \nu}^{\lambda}=f_{\mu \nu}^{\lambda}(q,t)=\langle Q_{\lambda},P_{\mu}P_{\nu}\rangle\in \mathbb{F}.
\]
The coefficient $f_{\mu \nu}^{\lambda}$ is called the $q,t$-Littlewood--Richardson coefficient.
Let $\lambda,\mu$ be partitions and define skew functions
$Q_{\lambda/\mu}\in \Lambda_{\mathbb{F}}$
by
\begin{equation}\label{skewQ1}
Q_{\lambda/\mu}=\sum_{\nu}f_{\mu \nu}^{\lambda}Q_{\nu}
\end{equation}
so that
\begin{equation}\label{skewQ}
\langle Q_{\lambda/\mu}, P_{\nu} \rangle=\langle Q_{\lambda}, P_{\mu}P_{\nu} \rangle.
\end{equation}
Likewise we define the skew Macdonald polynomial $P_{\lambda/\mu}$ by interchanging the $P$'s and $Q$'s in \eqref{skewQ}
\begin{equation}\label{skewP}
\langle P_{\lambda/\mu}, Q_{\nu} \rangle=\langle P_{\lambda}, Q_{\mu}Q_{\nu} \rangle.
\end{equation}
Since $Q_{\lambda}=b_{\lambda}P_{\lambda}$, it follows that
\begin{equation}\label{skewQP}
Q_{\lambda/\mu}=b_{\lambda}b_{\mu}^{-1}P_{\lambda/\mu}.
\end{equation}
See \cite[VI, (7.8)]{Mac95}.
For $u,v\in \mathbb{F}$ such that $v\neq \pm 1$, let $\omega_{u,v}$ denote the $\mathbb{F}$-algebra
endomorphism of $\Lambda_{\mathbb{F}}$ defined by
\[
\omega_{u,v}(p_{r})=(-1)^{r-1}\frac{1-u^r}{1-v^r}p_r
\]
for all $r\geq 1$, so that
\[
\omega_{u,v}(p_{\lambda})=(-1)^{|\lambda|-\ell(\lambda)}p_{\lambda}\prod_{i=1}^{\ell(\lambda)}\frac{1-u^{\lambda_i}}{1-v^{\lambda_i}}.
\]
Then we have the duality \cite[VI, (7.16)]{Mac95}
\begin{equation}\label{e-Mac1}
\omega_{q,t} P_{\lambda/\mu}(q,t)=Q_{\lambda'/\mu'}(t,q).
\end{equation}
For any alphabet $X$, the endomorphism $\omega_{q,t}$ can be interpret plethystically.
That is
\begin{equation}\label{e-Mac2}
X\mapsto \epsilon X \frac{1-q}{t-1}.
\end{equation}
For two alphabets $X$ and $Y$,
\begin{equation}\label{e-Mac4}
P_{\lambda}[X+Y]
=\sum_{\mu\subset \lambda}P_{\lambda/\mu}[X]P_{\mu}[Y].
\end{equation}
See \cite[VI, (7.9')]{Mac95}.
If $n<\ell(\lambda)$ then
\begin{equation}\label{e-Mac5}
P_{\lambda}(x_1,\dots,x_n;q,t)=0.
\end{equation}
See \cite[VI, (4.10)]{Mac95}.
If $\lambda$ is a partition of length $n$, then
\begin{equation}\label{e-Mac7}
P_{\lambda}(x_1,\dots,x_n;q,t)=x_1\dots x_n P_{\mu}(x_1,\dots,x_n;q,t),
\end{equation}
where $\mu=(\lambda_1-1,\dots,\lambda_n-1)$.
See \cite[VI, (4.17)]{Mac95}.
For the skew functions
\begin{equation}\label{e-Mac6}
Q_{\lambda/\mu}(x_1,\dots,x_n;q,t)=0
\end{equation}
unless $0\leq \lambda'_i-\mu'_i\leq n$ for each $i\geq 1$.
See \cite[VII, (7.15)]{Mac95}.
\begin{lem}\label{lem-Mac-1}
Let $\lambda$ and $\mu$ be partitions such that $\mu\subset \lambda$.
Then for $l$ a positive integer
\begin{equation}\label{lem-Mac-e2}
P_{\lambda/\mu}\Big(\Big[\frac{1-q}{t-1}\big(m_1+m_2+\cdots+m_l\big)\Big];q,t\Big)
=(-1)^{|\lambda|-|\mu|}Q_{\lambda'/\mu'}
\Big(\big[m_1+m_2+\cdots+m_l\big];t,q\Big),
\end{equation}
where the $m_i$ are single-letter alphabets (or monic monomials).
\end{lem}
\begin{proof}
We prove by straightforward calculation.
\begin{align*}\label{lem-Mac-e1}
&Q_{\lambda'/\mu'}\Big(\big[m_1+m_2+\cdots+m_l\big];t,q\Big) \\
&\quad=\omega_{q,t}P_{\lambda/\mu}\Big(\big[m_1+m_2+\cdots+m_l\big];q,t\Big)
\quad &\text{by \eqref{e-Mac1}}\ \\
&\quad =P_{\lambda/\mu}\Big(\Big[\big(\epsilon \frac{1-q}{t-1}\big)\big(m_1+m_2+\cdots+m_l\big)\Big];q,t\Big) \quad &\text{by \eqref{e-Mac2}}\ \\
&\quad =(-1)^{|\lambda|-|\mu|}P_{\lambda/\mu}\Big(\Big[\frac{1-q}{t-1}\big(m_1+m_2+\cdots+m_l\big)\Big];q,t\Big) \quad & \text{by \eqref{e-Mac3}}.
\end{align*}
Then we obtain \eqref{lem-Mac-e2}.
\end{proof}

By Lemma~\ref{lem-Mac-1} we obtain the next vanishing property for Macdonald polynomials.
\begin{prop}\label{prop-Mac-vanish}
Let $\lambda$ be a partition and $P_{\lambda}(x;q,t)$ be a Macdonald polynomial.
For any nonzero part $\lambda_i$,
\begin{equation}\label{prop-Mac-e1}
P_{\lambda}\Big(\Big[\frac{1-q}{t-1}\big(m_1+m_2+\cdots+m_{\lambda_i-1}\big)+n_1+\cdots+n_{i-1}\Big];q,t\Big)=0,
\end{equation}
where the $m_i$ and the $n_i$ are single-letter alphabets (or monic monomials).
\end{prop}
\begin{proof}
By \eqref{e-Mac4}
\begin{multline}\label{prop-Mac-e2}
P_{\lambda}\Big(\Big[\frac{1-q}{t-1}\big(m_1+m_2+\cdots+m_{\lambda_i-1}\big)+n_1+\cdots+n_{i-1}\Big];q,t\Big) \\
=\sum_{\mu\subset \lambda}P_{\lambda/\mu}\Big(\Big[\frac{1-q}{t-1}\big(m_1+m_2+\cdots+m_{\lambda_i-1}\big)\Big];q,t\Big)P_{\mu}\big([n_1+\cdots+n_{i-1}];q,t\big).
\end{multline}
By Lemma~\ref{lem-Mac-1}, the summand in the sum of \eqref{prop-Mac-e2} can be written as
\[
(-1)^{|\lambda|-|\mu|}Q_{\lambda'/\mu'}\Big(\big[m_1+m_2+\cdots+m_{\lambda_i-1}\big];t,q\Big)
P_{\mu}\big([n_1+\cdots+n_{i-1}];q,t\big).
\]
If $\ell(\mu)\geq i$ then $P_{\mu}\big([n_1+\cdots+n_{i-1}];q,t\big)=0$ by \eqref{e-Mac5}.
On the other hand, if $\ell(\mu)<i$ then
\[
\lambda_i-\mu_i=\lambda_i>\mathrm{Car}(m_1+\cdots+m_{\lambda_i-1})=\lambda_i-1,
\]
where $\mathrm {Car}(X)$ means the cardinality of the alphabet $X$.
It follows that
\[
Q_{\lambda'/\mu'}\Big(\big[m_1+m_2+\cdots+m_{\lambda_i-1}\big];t,q\Big)=0
\]
by \eqref{e-Mac6} with $(q,t,\lambda,\mu)\mapsto (t,q,\lambda',\mu')$.
\end{proof}

Define the modified complete symmetric function $g_n(X;q,t)$ by its generating function
\[
\prod_{i\geq 1}\frac{(tx_iy;q)_{\infty}}{(x_iy;q)_{\infty}}=\sum_{n\geq 0}g_n(X;q,t)y^n,
\]
and for any partition $\lambda=(\lambda_1,\lambda_2,\dots)$ define
\[
g_{\lambda}=g_{\lambda}(X;q,t)=\prod_{i\geq 1}g_{\lambda_i}(X;q,t).
\]
Note that $g_n$ can be written in plethystic notation
\begin{equation}\label{modi-complete}
g_n(X;q,t)=h_n\Big[\frac{1-t}{1-q}X\Big].
\end{equation}
For $\lambda=(r)$,
\begin{equation}\label{relation-P-g}
P_{(r)}=\frac{(q)_r}{(t)_r}g_r.
\end{equation}
See \cite[VI, (4.9)]{Mac95}.

The next result is called the Pieri formula for Macdonald polynomials \cite[VI, (6.24)]{Mac95}.
\begin{prop}\label{prop-Pieri}
Let $\mu$ be a partition. For $r$ a positive integer,
\begin{equation}
P_{\mu}g_r=\sum_{\lambda}\varphi_{\lambda/\mu}P_{\lambda},
\end{equation}
where the sum is over partitions $\lambda$ such that $\lambda/\mu$ is a horizontal $r$-strip.
\end{prop}
The explicit expression for $\varphi_{\lambda/\mu}$ is given in \cite[VI, (6.24)]{Mac95}.
We do not need the explicit form of $\varphi_{\lambda/\mu}$ in this paper.

The symmetric functions $g_{\lambda}$ form a basis of
$\Lambda_{\mathbb{F}}$ \cite[VI, (2.19)]{Mac95}.
Then the Macdonald polynomials can be written as a linear combination of
modified complete symmetric functions.
The formula is referred as the generalized Jacobi-Trudi expansions for Macdonald polynomials.
We state that in the next theorem.
\begin{thm}\cite{Lassalle}\label{thm-Lassalle}
Let $\lambda$ and $\mu$ be partitions with length at most $n$. Then
\begin{equation}\label{expand-P}
P_{\lambda}=\sum_{\mu\geq \lambda}c_{\mu}g_{\mu}.
\end{equation}
\end{thm}
Note that the explicit expression for $c_{\mu}$ is quite complicated, we refer the reader to
\cite[Theorem 3]{Lassalle}. When $q=t$, \eqref{expand-P} reduces to the Jacobi-Trudi identity
\[
s_{\lambda}=\det(h_{\lambda_i-i+j})_{1\leq i,j\leq n},
\]
where $s_{\lambda}$ is the Schur function.

We can express the skew Macdonald polynomials as (usual) Madonald polynomials.
\begin{prop}\label{prop-Mac-vanish-3}
For  partitions $\lambda$ and $\mu$ such that $\mu\subset \lambda$ and $\ell(\mu)<\ell(\lambda)$,
let
\begin{equation}\label{coeff-r}
P_{\lambda/\mu}=\sum_{\nu}r_{\mu \nu}^{\lambda}\cdot P_{\nu},
\end{equation}
where $\nu$ is over all partitions such that $|\nu|=|\lambda|-|\mu|$.
Then $r_{\mu \nu}^{\lambda}=0$ for $\nu_1<\lambda_{\ell(\mu)+1}$, or equivalently,
$\nu_1\geq \lambda_{\ell(\mu)+1}$.
\end{prop}
\begin{proof}
Substituting \eqref{PQ} into \eqref{skewQ1} yields
\[
Q_{\lambda/\mu}=\sum_{\nu}f_{\mu \nu}^{\lambda}Q_{\nu}
=\sum_{\nu}f_{\mu \nu}^{\lambda}b_{\nu}P_{\nu}.
\]
Together with $Q_{\lambda/\mu}=b_{\lambda}b_{\mu}^{-1}P_{\lambda/\mu}$ in \eqref{skewQP}, we have
\[
P_{\lambda/\mu}
=\sum_{\nu}f_{\mu \nu}^{\lambda}b_{\mu}b_{\lambda}^{-1}b_{\nu}P_{\nu}.
\]
Since the $P_{\lambda}$ form a basis of $\Lambda_{\mathbb{F}}$ \cite{Mac95},
comparing this with \eqref{coeff-r}, we have
\[
r_{\mu \nu}^{\lambda}=b_{\mu}b_{\nu}b_{\lambda}^{-1}f_{\mu \nu}^{\lambda}.
\]
Hence, to prove the proposition it suffices to prove that a family of $q,t$-Littlewood--Richardson coefficients are zeros. That is
\begin{equation}\label{Mac-vanish1}
f_{\mu \nu}^{\lambda}=\langle Q_{\lambda},P_{\mu}P_{\nu}\rangle=0
\end{equation}
for $\nu_1<\lambda_{\ell(\mu)+1}$, $\mu\subset \lambda$ and $\ell(\mu)<\ell(\lambda)$.

By \eqref{expand-P}, we can write
\[
P_{\mu}=\sum_{\omega\geq \mu}c_{\omega}g_{\omega}.
\]
Using this and by the linearity of the $q,t$-Hall scalar product, we have
\begin{equation}\label{Mac-vanish2}
f_{\mu \nu}^{\lambda}=\langle Q_{\lambda},P_{\mu}P_{\nu}\rangle=\sum_{\omega\geq \mu}c_{\omega}
\langle Q_{\lambda},P_{\nu}g_{\omega}\rangle.
\end{equation}
We can get $\ell(\omega)\leq \ell(\mu)$ by $\omega\geq \mu$.
Consequently, $\ell(\omega)<\ell(\lambda)$.
Then we prove \eqref{Mac-vanish1} by showing that each term in the sum of \eqref{Mac-vanish2}
\begin{equation}\label{Mac-vanish3}
\langle Q_{\lambda},P_{\nu}g_{\omega}\rangle=0
\end{equation}
for $\nu_1<\lambda_{\ell(\mu)+1}$, $\ell(\omega)<\ell(\lambda)$.
Repeatedly using of the Pieri formulas for Macdonald polynomials (Proposition~\ref{prop-Pieri}) for
$\ell(\omega)$ times, we have
\[
P_{\nu}g_{\omega}=P_{\nu}\cdot g_{\omega_1}\cdot g_{\omega_2}\cdots
=\sum_{\lambda^*}\overline{c}_{\lambda^*}P_{\lambda^*},
\]
where $\overline{c}_{\lambda^*}$ is the coefficient for $P_{\lambda^*}$.
Since $\nu_1<\lambda_{\ell(\mu)+1}$ and $\ell(\omega)<\ell(\lambda)$,
none of the partitions $\lambda^*$ can be $\lambda$.
(It needs at least $(\ell(\mu)+1)$-th use of the Pieri formula.)
It follows that
\[
\langle Q_{\lambda},P_{\nu}g_{\omega}\rangle
=\sum_{\lambda^*}\overline{c}_{\lambda^*}\langle Q_{\lambda},P_{\lambda^*}\rangle=0. \qedhere
\]
\end{proof}

The next proposition is concerned about the degrees of the Macdonald polynomials.
\begin{prop}\label{prop-Macdeg}
For positive integers $s$ and $n$, let $\lambda$ be a partition and $u:=(u_1,u_2,\dots,u_s)$ be a vector of integers such that $1\leq u_1<u_2<\cdots<u_s\leq n$. For a polynomial $f(x)=f(x_1,\dots,x_n)$ and $c_i\in \mathbb{Q}(q)$, we define a substitution $T_u\big(f(x)\big)$ by replacing $x_{u_i}=c_ix_{u_s}$ for $i=1,\dots,s-1$ in $f(x)$.
Then the degree in $x_{u_s}$ of $T_u\big(P_{\lambda}(x;q,t)\big)$ is at most $\lambda_1+\cdots+\lambda_s$.
\end{prop}
\begin{proof}
By \eqref{Mac-m}, the Macdonald polynomial $P_{\lambda}$ is a linear combination of those monomial symmetric functions $m_{\mu}$ for $\mu\leq \lambda$.
Hence, to prove the proposition, it suffices to prove that for $\mu\leq \lambda$,
the degree in $x_{u_s}$ of $T_u\big(m_{\mu}(x)\big)$ is no more than $\lambda_1+\cdots+\lambda_s$.
By the definition of $m_{\mu}$, it is clear that
this degree is no more than $\mu_1+\cdots+\mu_s$,
which is no more than $\lambda_1+\cdots+\lambda_s$ by
$\mu\leq \lambda$.
\end{proof}

\section{Polynomiality and rationality}\label{sec-poly}

The aim of this section is to show that $B_{n,n_0}(a,b,c,l,\mu)$ and
$C_{n,n_0}(a,b,c,l,m)$ satisfy the following two properties: i) they are polynomials in $q^a$;
ii) if the expressions for them hold for sufficiently many integers $c$,
then they hold for all positive integers $c$.

We begin with the next polynomiality lemma.
\begin{lem}\label{lem1}
Let $L(x_1,\dots,x_n)$ be an arbitrary Laurent polynomial independent of $a$ and $x_0$.
Then, for fixed nonnegative integers $b$ and $t$ such that $t\leq nb$,
\begin{equation}\label{p1}
\CT_x x_0^{t} L(x_1,\dots,x_n) \prod_{i=1}^n (x_0/x_i)_a(qx_i/x_0)_b
\end{equation}
is a polynomial in $q^a$ of degree at most $nb-t$.
Moreover, if $t>nb$, then the constant term \eqref{p1} vanishes.
\end{lem}
\begin{proof}
The proof is almost the same as that of \cite[Lemma 2.2]{XZ}.
\end{proof}
By Lemma~\ref{lem1}, it is not hard to conclude that $B_{n,n_0}(a,b,c,l,\mu)$ and $C_{n,n_0}(a,b,c,l,m)$
are polynomials in $q^a$.
\begin{cor}\label{cor-poly-BC}
The constant terms $B_{n,n_0}(a,b,c,l,\mu)$ and
$C_{n,n_0}(a,b,c,l,m)$ are polynomials in $q^a$ of degree at most $nb+l+|\mu|$ and $nb+l+m$ respectively, assuming all the parameters but $a$ are fixed.
\end{cor}
\begin{proof}
We can write $C_{n,n_0}(a,b,c,l,m)$ as
\begin{equation}\label{poly-C}
\CT_xx_0^{-l}\prod_{i=n-m+1}^n(1-q^{b+1}x_i/x_0)
\prod_{i=1}^n (x_0/x_i)_a(qx_i/x_0)_bL(x_1,\dots,x_n),
\end{equation}
where
\begin{equation}\label{Lx}
L(x_1,\dots,x_n)=h_{l}\Big[\sum_{i=1}^{n}\frac{1-q^{c-\chi(i\leq n_0)}}{1-q}x_i\Big]
\prod_{1\leq i<j\leq n}(x_i/x_j)_{c-\chi(i\leq n_0)}(qx_j/x_i)_{c-\chi(i\leq n_0)}.
\end{equation}
Expanding the first product in \eqref{poly-C}, we can see that $C_{n,n_0}(a,b,c,l,m)$ is a polynomial in $q^a$ of degree at most $nb+l+m$ by Lemma~\ref{lem1}.

By \eqref{e-Mac4} and \eqref{e-homo-sym}, we can write $B_{n,n_0}(a,b,c,l,\mu)$ as
\begin{multline}\label{poly-B}
\sum_{\nu\subset \mu}P_{\nu}\big(\big[\frac{q^{c-b-1}-q^a}{1-q^c}\big];q,q^c\big)
\CT_xx_0^{-l-|\mu|+|\nu|}
P_{\mu/\nu}\Big(\Big[\sum_{i=1}^n\frac{1-q^{c-\chi(i\leq n_0)}}{1-q^c}x_i\Big];q,q^c\Big) \\
\times\prod_{i=1}^n (x_0/x_i)_a(qx_i/x_0)_bL(x_1,\dots,x_n),
\end{multline}
where $L(x_1,\dots,x_n)$ is given by \eqref{Lx}.
We write
\begin{align*}
P_{\nu}\big(\big[\frac{q^{c-b-1}-q^a}{1-q^c}\big];q,q^c\big)
&=\sum_{\lambda\geq \nu}c_{\lambda}g_{\lambda}\Big(\Big[\frac{q^{c-b-1}-q^a}{1-q^c}\Big];q,q^c\Big)
\quad &\text{by \eqref{expand-P}}\\
&=\sum_{\lambda\geq \nu}c_{\lambda}h_{\lambda}\Big[\frac{q^{c-b-1}-q^a}{1-q}\Big]
\quad &\text{by \eqref{modi-complete}}\\
&=\sum_{\lambda\geq \nu}c_{\lambda}q^{(c-b-1)|\lambda|}h_{\lambda}\Big[\frac{1-q^{a+b+1-c}}{1-q}\Big]
\quad &\text{by \eqref{e-homo-sym}}\\
&=\sum_{\lambda\geq \nu}c_{\lambda}q^{(c-b-1)|\lambda|}
\prod_{i\geq 1}\frac{(q^{a+b+1-c})_{\lambda_i}}{(q)_{\lambda_i}}.
\end{align*}
The right-most equality holds by $h_r[(1-z)/(1-q)]=(z)_r/(q)_r$, see e.g., \cite[page 27]{Mac95}.
Then one can see that $P_{\nu}\big[\frac{q^{c-b-1}-q^a}{1-q^c}\big]$
is a polynomial in $q^a$ of degree at most $|\nu|$. Together with the fact that
the constant term in \eqref{poly-B} is a polynomial in $q^a$ of degree at most $nb+l+|\mu|-|\nu|$
by Lemma~\ref{lem1}, we conclude that $B_{n,n_0}(a,b,c,l,\mu)$ is a polynomial in $q^a$ of degree
at most $nb+l+|\mu|$.
\end{proof}

The following rationality result, which is implicitly due to
Stembridge \cite{stembridge1987}, as can be seen from the proof. One can also see this result in
\cite[Proposition~3.1]{XZ} and \cite[Lemma~7.5]{KNPV}.
The $q=1$ case of this result
is the equal parameter case of \cite[Proposition 2.4]{Gessel-Lv-Xin-Zhou2008}.
\begin{prop}\label{prop-rationality}
Let $n$ be a positive integer and $\alpha=(\alpha_1,\dots,\alpha_n)\in \mathbb{Z}^{n}$ such that
$|\alpha|=0$. For $c$ a nonnegative integer
\begin{align}
\CT_x x_1^{\alpha_1}\cdots  x_n^{\alpha_n}\prod_{1\leq i<j\leq n}
\Big(\frac{x_{i}}{x_{j}}\Big)_{c}\Big(\frac{x_{j}}{x_{i}}q\Big)_{c}=\frac{(q)_{nc}}{(q)_{c}^{n}}\cdot
R_n(q^c;q,\alpha),
\end{align}
where $R_n(q^c;q,\alpha)$ is a rational function in $q^c$ and $q$.
\end{prop}

A straightforward consequence of Proposition~\ref{prop-rationality} is the next result.
\begin{cor}\label{cor-rationality}
Let $c,d$ be nonnegative integers, $n$ be a positive integer, and $H$ be a homogeneous Laurent polynomial in $x_0,x_1,\dots,x_n$ of degree 0. If an expression for
\[
\CT_{x} H\prod_{1\leq i<j\leq n}
\Big(\frac{x_{i}}{x_{j}}\Big)_{c}\Big(\frac{x_{j}}{x_{i}}q\Big)_{c}
\]
holds for $c\geq d$, then it holds for all nonnegative integers $c$.
Here $d$ is independent of $c$, but it may depend on $H$.
\end{cor}
\begin{proof}
By Proposition~\ref{prop-rationality}, it exists a rational function $R_n(q^c;q,H)$ such that
\[
\CT_{x} H\prod_{1\leq i<j\leq n}
\Big(\frac{x_{i}}{x_{j}}\Big)_{c}\Big(\frac{x_{j}}{x_{i}}q\Big)_{c}
=\CT_x H|_{x_0=1}\prod_{1\leq i<j\leq n}
\Big(\frac{x_{i}}{x_{j}}\Big)_{c}\Big(\frac{x_{j}}{x_{i}}q\Big)_{c}
=\frac{(q)_{nc}}{(q)_{c}^{n}}\cdot R_n(q^c;q,H).
\]
Here the first equality holds by the homogeneity of $H$.
Then
\begin{equation}\label{rational-p1}
\frac{(q)_c^n}{(q)_{nc}}\CT_{x} H\prod_{1\leq i<j\leq n}
\Big(\frac{x_{i}}{x_{j}}\Big)_{c}\Big(\frac{x_{j}}{x_{i}}q\Big)_{c}
=R_n(q^c;q,H).
\end{equation}
It follows that if \eqref{rational-p1} holds for $c\geq d$ for some nonnegative integer $d$,
then it holds for all nonnegative integers $c$ since the both sides are rational functions in $q^c$.
\end{proof}

\section{Preliminaries for determination of the roots}\label{sec-preliminaries}

In this section, we present several essential results to determine the roots
of $B_{n,n_0}(a,b,c,l,\mu)$ and $C_{n,n_0}(a,b,c,l,m)$.

\subsection{The cardinality of an alphabet}

The main observation in this subsection is Proposition~\ref{lem-subs}, which
concerns the cardinality of an alphabet in $B_{n,n_0}(a,b,c,l,\mu)$
under certain substitutions.

For a positive integer $s$, let $\mathfrak{S}_s$ be the set of all the permutations of $\{1,2,\dots,s\}$.
For $w\in \mathfrak{S}_s$ and $r$ an integer such that $0\leq r<s$,
define
\begin{equation}\label{Nwr}
N_{w,r}:=\sum_{j=1}^su_j(w)+\chi(w(j)>r)\chi(w(j-1)>r),
\end{equation}
where $w(0):=0$ and
\begin{equation}\label{def-uj}
u_j(w)=\begin{cases}
1 &\text{if $w(j-1)<w(j)$,}\\
0 &\text{otherwise.}
\end{cases}
\end{equation}
For example, if $s=7$, $r=3$ and $w=(3147526)$, then  $N_{w,r}=6$.
Let
\begin{equation}\label{e-weight}
w(0)\mathop{\longrightarrow}^{e_1}w(1)\mathop{\longrightarrow}^{e_2} \cdots\mathop{\longrightarrow}^{e_{s-1}} w(s-1)\mathop{\longrightarrow}^{e_s}w(s)
\end{equation}
be a weighted directed path on $w\in \mathfrak{S}_s$,
where
\begin{equation}\label{defi-ej}
e_j:=u_j(w)+\chi(w(j)>r)\chi(w(j-1)>r).
\end{equation}
Then we can view $N_{w,r}$ as the sum of the weights on this directed path.
Note that $e_j$ is determined by $w(j), w(j-1)$ and the integer $r$,
and $e_1=1$.
For $r\in \{0,1,\dots,s-1\}$, denote $I_1:=\{1,2,\dots,r\}$ and $I_2:=\{r+1,r+2,\dots,s\}$.
Let $\mathfrak{S}_{s,r}$ be the subset of $\mathfrak{S}_s$ that includes all the permutations
that all the elements of $I_1$ and $I_2$ are in a decreasing order respectively.
We define $\varepsilon_1: \mathfrak{S}_s\mapsto \mathfrak{S}_{s,r}$
to be the map of permuting the elements of $I_1$ and $I_2$ in $w\in \mathfrak{S}_s$ respectively
such that all the elements of $I_1$ and $I_2$ in $w$ are decreasing.
For example, if $s=7$, $r=3$ and $w=(3147526)$, then $I_1=\{1,2,3\}, I_2=\{4,5,6,7\}$,
$\varepsilon_1\circ w=(3276514)$.
Having the above notation, we obtain the following property for
the number $N_{w,r}$.
\begin{lem}\label{lem-varepsilon1}
For a positive integer $s$, let $w\in \mathfrak{S}_s$ and $r$ be an integer such that $0\leq r<s$.
Let $\varepsilon_1$ be defined as above. Then
\begin{equation}\label{Neq0}
N_{w, r}\geq N_{\varepsilon_1\circ w, r}.
\end{equation}
Furthermore, if there is an $i\in \{2,\dots,s\}$ such that $w(i-1),w(i)\in I_1$ (resp. $I_2$) and $w(i-1)<w(i)$,
then
\begin{equation}\label{Neq}
N_{w,r}>N_{\varepsilon_1\circ w,r}.
\end{equation}
\end{lem}
\begin{proof}
Let $w_1=\varepsilon_1\circ w$.
We will show that for $j=2,\dots,s$
\begin{subequations}\label{w_1}
\begin{equation}\label{w_1-1}
\chi(w(j)>r)=\chi(w_1(j)>r),
\end{equation}
and
\begin{equation}\label{w_1-2}
u_j(w)\geq u_j(w_1).
\end{equation}
\end{subequations}
Then it is clear that $N_{w,r}\geq N_{\varepsilon_1\circ w,r}$ from the definition of $N_{w,r}$ in \eqref{Nwr}
and the fact that
\[
u_1(w)+\chi(1>r)\chi(0>r)=u_1(w_1)+\chi(1>r)\chi(0>r)=1.
\]

Since $\varepsilon_1$ permutes the elements of $I_1$ and $I_2$ in $w$ respectively,
for some $j$ if $w(j)\in I_1$ (resp. $I_2$) then $w_1(j)\in I_1$ (resp. $I_2$).
Thus \eqref{w_1-1} holds.

Since $u_j(w)$ is determined by $w(j)$, $w(j-1)$ and $r$, we prove \eqref{w_1-2} by discussing the following four cases:
\begin{enumerate}
 \item $w(j-1)\in I_1, w(j)\in I_2$;
 \item $w(j-1)\in I_2, w(j)\in I_1$;
 \item $w(j-1)\in I_1, w(j)\in I_1$;
 \item $w(j-1)\in I_2, w(j)\in I_2$.
\end{enumerate}
If $w(j-1)\in I_1$ and $w(j)\in I_2$ for some $j$,
then by the definition of $\varepsilon_1$, $w_1(j-1)\in I_1$ and $w_1(j)\in I_2$.
It follows that $w(j-1)<w(j)$ and $w_1(j-1)<w_1(j)$. Then $u_j(w)=u_j(w_1)=1$.
Similarly, if $w(j-1)\in I_2$ and $w(j)\in I_1$ for some $j$,
then $u_j(w)=u_j(w_1)=0$.
If $w(j-1), w(j)\in I_1$ for some $j$,
then $w_1(j-1), w_1(j)\in I_1$.
By the definition of $\varepsilon_1$,
all the elements of $I_1$ in $w_1$ are in a decreasing order.
Thus $w_1(j-1)>w_1(j)$.
Then by the definition of $u_j$ in \eqref{def-uj} we have
$u_j(w_1)=0$.
It follows that
\begin{equation}\label{w-case3}
u_j(w)\geq u_j(w_1).
\end{equation}
For the case $w(j-1),w(j)\in I_2$, similarly we can also get \eqref{w-case3}.
In conclusion, \eqref{w_1-2} holds for all the four cases.
Then $N_{w,r}\geq N_{\varepsilon_1\circ w,r}$ by \eqref{w_1}.

If for some $i\in \{2,\dots,s\}$ such that $w(i-1),w(i)\in I_1$ (or $I_2$) and $w(i-1)<w(i)$,
then $u_i(w)>u_i(w_1)$. Thus, in this case the equality of \eqref{Neq0} can not hold
and we obtain \eqref{Neq}.
\end{proof}

We define another map $\varepsilon_2: \mathfrak{S}_{s,r}\mapsto \mathfrak{S}_{s,r}$.
If $w=(s,s-1,\dots,1)$, then $\varepsilon_2\circ w=w$.
If $w\in \mathfrak{S}_{s,r}$ but $w\neq (s,s-1,\dots,1)$, then
it exists an integer $1\leq j<s$ such that $w(i)=s-i+1$ for $i=1,\dots,j-1$ but $w(j)\neq s-j+1$.
The map $\varepsilon_2$ acts on $w$ by putting $w(t)=s-j+1$ for some $t>j$ between $w(j-1)$ and $w(j)$ in $w$, and other entries of $w$ keep their original order in $w$.
That is
\begin{equation}\label{defi-varepsilon2}
\varepsilon_2\circ w=\big(w(1),\dots,w(j-1),w(t),w(j),\dots,w(t-1),w(t+1),\dots,w(s)\big).
\end{equation}
In particular, if $j=1$ then $w(t)=s$ and $\varepsilon_2\circ w=\big(s,w(1),\dots,w(t-1),w(t+1),\dots,w(s)\big)$.
Note that if $w\neq (s,s-1,\dots,1)$ then $w(t)>r$ since $w\in \mathfrak{S}_{s,r}$.
For example, if $s=7$, $r=3$ and $w=(3276514)$, then $w_1:=\varepsilon_2\circ (3276514)=(7326514)$.
In this case, $N_{w,r}=5$ and $N_{w_1,r}=4$.
We can decrease the number $N_{w,r}$ further by acting $\varepsilon_2$ on $w$.
This is the next lemma.
\begin{lem}\label{lem-varepsilon2}
For a positive integer $s$ and $r\in \{0,1,\dots,s-1\}$, let $w\in \mathfrak{S}_{s,r}$, $N_{w,r}$ be defined
in \eqref{Nwr} and $\varepsilon_2$ be defined as above.
Then
\begin{equation}\label{varep2}
N_{w,r}\geq N_{\varepsilon_2\circ w,r}.
\end{equation}
In particular, if the equality of \eqref{varep2} holds, then $w(1)=s$.
\end{lem}
\begin{proof}
If $w=(s,s-1,\dots,1)$ then $\varepsilon_2 \circ w=w$ and the lemma holds. Hence, we assume
$w\neq (s,s-1,\dots,1)$ in the following of the proof.
By the definition of $\varepsilon_2$, $N_{w,r}$ and $N_{\varepsilon_2\circ w,r}$ only differ in two parts (view them as weighted directed paths as in \eqref{e-weight}). That is
\begin{align*}
\big(w(j-1)\mathop{\longrightarrow}^{e_{j}} w(j)\big)&\mapsto
\big(w(j-1)\mathop{\longrightarrow}^{e'_{j}}w(t)\mathop{\longrightarrow}^{e'_{j+1}}w(j)\big),\\
\big(w(t-1)\mathop{\longrightarrow}^{e_{t}}w(t)\mathop{\longrightarrow}^{e_{t+1}}w(t+1)\big)&\mapsto \big(w(t-1)\mathop{\longrightarrow}^{e'_{t+1}}w(t+1)\big),
\end{align*}
where $j$ is the integer such that
$w(i)=s-i+1$ for $i=1,\dots,j-1$ but $w(j)\neq s-j+1$, and $w(t)=s-j+1$ for some $t>j$.
Recall that $N_{w,r}$ equals the sum of the weights on its corresponding directed path.
Thus, to prove $N_{w,r}\geq N_{\varepsilon_2\circ w,r}$ it is sufficient to prove
$e_{j}+e_{t}+e_{t+1}\geq e'_{j}+e'_{j+1}+e'_{t+1}$.
We will prove this by showing that
\begin{subequations}
\begin{equation}\label{claim2-1}
e_{j}+1\geq e'_{j}+e'_{j+1},
\end{equation}
and
\begin{equation}\label{claim2-2}
e_{t}+e_{t+1}=e'_{t+1}+1.
\end{equation}
\end{subequations}

If $j=1$ then $e_j=e'_j=1$, and $e'_{j+1}=1$ for $w(j)\in I_2$ and $e'_{j+1}=0$ for $w(j)\in I_1$. It follows that \eqref{claim2-1} holds for $j=1$.
By the definition of $\varepsilon_2$, for $j\geq 2$ we have $w(j-1), w(t)\in I_2$ and $w(j-1)=w(t)+1=s-j+2$.
We then prove \eqref{claim2-1} for $j\geq 2$ by discussing the two cases: $w(j)\in I_1$ and $w(j)\in I_2$.
The case $w(j)\in I_2$ can not occur. Otherwise, since $w(j-1)\in I_2$ and $w\in \mathfrak{S}_{s,r}$,
all the elements of $I_2$ in $w$ are in a decreasing order, $w(j)=w(j-1)-1=s-j+1$. This contradicts the assumption $w(j)\neq s-j+1$ in \eqref{defi-varepsilon2}.
If $w(j)\in I_1$ and $j\geq 2$ then $e_{j}=e'_{j+1}=0$. Together with the fact $e'_{j}=1$ by $w(j-1)>w(t)$ and $w(j-1),w(t)\in I_2$, we have $e_{j}+1=e'_{j}+e'_{j+1}=1$.
Therefore \eqref{claim2-1} holds for $j=1,2,\dots,s-1$.

By the definition of $\varepsilon_2$ in \eqref{defi-varepsilon2} again, we have $w(t)\in I_2, w(t)>w(t-1)$ and $w(t)>w(t+1)$. Since $w\in \mathfrak{S}_{s,r}$ and all the elements of $I_2$ in $w$ are in a decreasing order, $w(t-1)\in I_1$.
We then prove \eqref{claim2-2} by discussing the following two cases:
\begin{enumerate}
 \item $w(t-1)\in I_1, w(t+1)\in I_2$;
 \item $w(t-1)\in I_1, w(t+1)\in I_1$.
 \end{enumerate}
If $w(t-1)\in I_1, w(t+1)\in I_2$ then $e_{t}=e_{t+1}=e'_{t+1}=1$ and $e_{t}+e_{t+1}=e'_{t+1}+1=2$.
If $w(t-1), w(t+1)\in I_1$ then $e_{t}=1, e_{t+1}=0$. Since all the elements of $I_1$ in $w$ are in a decreasing order, $e'_{t+1}=0$. Then $e_{t}+e_{t+1}=e'_{t+1}+1=1$.
Therefore \eqref{claim2-2} holds.

We prove the particular case by showing that $N_{w,r}>N_{\varepsilon_2\circ w,r}$ for $w(1)\neq s$.
If $w(1)\neq s$ then $j=1$ and $w(1)\in I_1$.
Correspondingly, $e_1=e'_1=1$ and $e'_2=0$.
Hence, $e_1+1=2>e'_1+e'_2=1$. Together with \eqref{claim2-2} and the above discussion,
$N_{w,r}>N_{\varepsilon_2\circ w,r}$.
\end{proof}

By Lemma~\ref{lem-varepsilon1} and Lemma~\ref{lem-varepsilon2}, we obtain the next proposition.
\begin{prop}\label{prop-lowerbound}
For a positive integer $s$, let $w\in\mathfrak{S}_s$, $w(0):=0$ and $d_1,\dots,d_s$ be nonnegative integers such that $d_j\geq 1$ if $w(j)>w(j-1)$ for $j=1,\dots,s$.
For an integer $r$ such that $0\leq r\leq s$,
\begin{subequations}\label{lowerbounds}
\begin{equation}\label{e-lowerbound}
\sum_{j=1}^s\big(d_j+\chi(w(j)>r)\chi(w(j-1)>r)\big)\geq s-r,
\end{equation}
and for $i\in \{1,2,\dots,s\}$,
\begin{equation}\label{lowerbound1}
\sum_{\substack{j=1\\j\neq i}}^s\big(d_j+\chi(w(j)>r)\chi(w(j-1)>r)\big)\geq s-r-1.
\end{equation}
\end{subequations}
In particular, if the equality of \eqref{e-lowerbound} holds, then $w(1)>r$ and $d_1=1$.
\end{prop}
\begin{proof}
By the definition of $u_j(w)$ in \eqref{def-uj}, we have $d_j\geq u_j(w)$.
It follows that
\begin{subequations}\label{geq1}
\begin{equation}\label{geq1-1}
\sum_{j=1}^s\big(d_j+\chi(w(j)>r)\chi(w(j-1)>r)\big)
\geq \sum_{j=1}^s\big(u_j(w)+\chi(w(j)>r)\chi(w(j-1)>r)\big)=N_{w,r},
\end{equation}
and
\begin{equation}\label{geq2}
\sum_{\substack{j=1\\j\neq i}}^s\big(d_j+\chi(w(j)>r)\chi(w(j-1)>r)\big)\geq
\sum_{\substack{j=1\\j\neq i}}^su_j(w)+\chi(w(j)>r)\chi(w(j-1)>r).
\end{equation}
\end{subequations}
We then prove \eqref{e-lowerbound} by showing that $N_{w,r}\geq s-r$,
and prove \eqref{lowerbound1} by showing that the right-hand side of \eqref{geq2}
is no less than $s-r-1$.

For any permutation $w\in \mathfrak{S}_s$, let $\varepsilon_1$ act on $w$ and then let
$\varepsilon_2$ act on $\varepsilon_1\circ w$ recursively until the resulting permutation
is $w_1:=(s,s-1,\dots,1)$.
By the definition of $\varepsilon_1$ and $\varepsilon_2$, these operations are valid.
By Lemma~\ref{lem-varepsilon1} and Lemma~\ref{lem-varepsilon2},
\begin{equation}\label{N-chain}
N_{w,r}\geq N_{\varepsilon_1 \circ w,r}\geq N_{\varepsilon_2\circ(\varepsilon_1 \circ w),r}
\geq N_{\varepsilon_2\circ\cdots \circ\varepsilon_2\circ(\varepsilon_1 \circ w),r}
=N_{w_1,r}=s-r.
\end{equation}
Together with \eqref{geq1-1} we obtain \eqref{e-lowerbound}.

For a fixed $i\in \{1,2,\dots,s\}$,
it is clear that
$u_i(w)+\chi(w(i)>r)\chi(w(i-1)>r)\in \{0,1,2\}$.
If $u_i(w)+\chi(w(i)>r)\chi(w(i-1)>r)\leq 1$,
then
\begin{multline*}
\sum_{\substack{j=1\\j\neq i}}^su_j(w)+\chi(w(j)>r)\chi(w(j-1)>r)\\
=N_{w,r}-u_i(w)-\chi(w(i)>r)\chi(w(i-1)>r)
\geq s-r-1
\end{multline*}
using $N_{w,r}\geq s-r$.
Together with \eqref{geq2} gives \eqref{lowerbound1}.
In particular, \eqref{lowerbound1} holds for $i=1$ since in this case
$u_1(w)+\chi(w(1)>r)\chi(w(0)>r)=1$.
If $u_i(w)+\chi(w(i)>r)\chi(w(i-1)>r)=2$ for some $i\in \{2,\dots,s\}$,
then $w(i)>w(i-1)>r$, i.e., $w(i-1),w(i)\in I_2$ and $w(i-1)<w(i)$.
Let $\varepsilon_1$ act on $w$ and then let
$\varepsilon_2$ act on $\varepsilon_1\circ w$ recursively again until the resulting permutation
is $w_1:=(s,s-1,\dots,1)$. Using \eqref{Neq} we have
\[
N_{w,r}>N_{w_1,r}=s-r.
\]
It follows that
\begin{multline*}
\sum_{\substack{j=1\\j\neq i}}^su_j(w)+\chi(w(j)>r)\chi(w(j-1)>r)\\
>N_{w_1,r}-u_i(w)-\chi(w(i)>r)\chi(w(i-1)>r)
=s-r-2.
\end{multline*}
Hence, in all the cases the right-hand side of \eqref{geq2} is no less than $s-r-1$ and we obtain \eqref{lowerbound1}.

Let
\[
S:=\sum_{j=1}^s\big(d_j+\chi(w(j)>r)\chi(w(j-1)>r)\big).
\]
Finally,  we prove that if $S=s-r$ then $w(1)>r$ and $d_1=1$.
By the argument in the above, we know that $S\geq N_{w,r}\geq s-r$.
If $S=s-r$ then it forces $N_{w,r}=s-r$. It follows that $d_j=u_j(w)$ for $j=1,\dots,s$.
In particular, $d_1=u_1(w)=1$.
We then prove $w(1)>r$ by contradiction.
If $w(1)\leq r\neq s$, then
\[
N_{w,r}\geq N_{\varepsilon_1 \circ w,r}>N_{\varepsilon_2\circ(\varepsilon_1 \circ w),r}\geq N_{w_1,r}=s-r.
\]
Here $N_{\varepsilon_1 \circ w,r}\neq N_{\varepsilon_2\circ(\varepsilon_1 \circ w),r}$
by Lemma~\ref{lem-varepsilon2}.
Consequently,
$S\geq N_{w,r}>s-r$. This contradicts the fact $S=s-r$.
Hence, $w(1)>r$ for $S=s-r$.
\end{proof}

Now we can prove the key lemma in this paper, which plays an important role in proving the vanishing properties
of $B_{n,n_0}(a,b,c,l,\mu)$ and $C_{n,n_0}(a,b,c,l,m)$.
\begin{lem}\label{lem-key}
For $s,c$ positive integers, let $b,t$ and $k_1,\dots,k_s$ be nonnegative integers such that
$1\leq k_{i}\leq (s-1)(c-1)+b+t$ for $1\leq i\leq s$.
Then for an integer $r$ such that $0\leq r\leq s$, at least one of the following holds:
\begin{enumerate}
\item $1\leq k_i\leq b$ for some $i$ with $1\leq i\leq s$;
\item $-c+1\leq k_i-k_j\leq c-2$ for some $(i,j)$ such that $1\leq i<j\leq s$ and $i\leq r$;
\item $-c\leq k_{i}-k_{j}\leq c-1$ for some $(i,j)$ such that $r<i<j\leq s$;
\item there exists a permutation $w\in\mathfrak{S}_s$ and nonnegative integers $d_1,\dots,d_s$
such that
\begin{subequations}\label{e-k1}
\begin{equation}\label{e-k1-1}
k_{w(1)}=b+d_1,
\end{equation}
and
\begin{equation}\label{e-k1-2}
k_{w(j)}-k_{w(j-1)}=c-1+\chi(w(j)>r)\chi(w(j-1)>r)+d_j \quad \text{for $2\leq j\leq s$.}
\end{equation}
\end{subequations}
Here the $d_j$ satisfy
\begin{equation}\label{t}
s-r\leq \sum_{j=1}^{s}\big(d_j+\chi(w(j)>r)\chi(w(j-1)>r)\big)\leq t,
\end{equation}
$w(0):=0$, and $d_j>0$ if $w(j-1)<w(j)$ for $1\leq j\leq s$.
In particular, if $t=s-r$ and $r<s$, then
\begin{equation}\label{t-specialcase}
k_i=b+1 \quad  \text{for some $i>r$}.
\end{equation}
\end{enumerate}
\end{lem}

\begin{proof}
We prove the lemma by showing that if (1)--(3) fail then (4) must hold.

Assume that (1)--(3) are all false.
Then we construct a weighted tournament $T$ on a complete graph
on $s$ vertices, labelled $1,\dots,s$, as follows.
For the edge $(i,j)$ with $1\leq i<j\leq s$ and $i\leq r$, we draw an arrow from $j$ to $i$ and attach a weight $c-1$ if $k_i-k_j\ge c-1$.
If, on the other hand, $k_i-k_j\le -c$ then we draw an arrow
from $i$ to $j$ and attach the weight $c$.
Similarly, for the edge $(i,j)$ with $1\leq i<j\leq s$ and $i>r$,
we draw an arrow from $j$ to $i$ and attach a weight $c$ if $k_i-k_j\ge c$,
and we draw an arrow from $i$ to $j$ and attach the weight $c+1$
if $k_i-k_j\le -c-1$.
Note that the weight of each edge of a tournament is nonnegative since $c$ is a positive integer.

We call a directed edge from $i$ to $j$ ascending if $i<j$.
It is immediate from our construction that
(i) the weight of the edge $i\to j$ is less than or equal $k_j-k_i$, and
(ii) the weight of an ascending edge is positive.

We will use (i) and (ii) to show that any of the above-constructed
tournaments is acyclic and hence transitive.
As consequence of (i), the weight of a directed path from $i$ to $j$ in $T$,
defined as the sum of the weights of its edges, is at most $k_j-k_i$.
Proceeding by contradiction, assume that $T$ contains a cycle $C$.
By the above, the weight of $C$ must be non-positive, and hence $0$.
Since $C$ must have at least one ascending edge, which by (ii) has positive
weight, the weight of $C$ is positive, a contradiction.

Since each $T$ is transitive, there is exactly one directed Hamilton path $P$ in $T$,
corresponding to a total order of the vertices.
Assume $P$ is given by
\[
P=w(1)\rightarrow w(2)\rightarrow\cdots\rightarrow w(s-1)\rightarrow w(s),
\]
where we have suppressed the edge weights.
Then
\[
k_{w(s)}-k_{w(1)}\ge (s-1)(c-1)+\sum_{j=2}^s \chi(w(j)>r)\chi(w(j-1)>r),
\]
and thus
\begin{align}\label{e-contradiction}
k_{w(s)}&\ge k_{w(1)}+(s-1)(c-1)+\sum_{j=2}^s \chi(w(j)>r)\chi(w(j-1)>r)\nonumber\\
    & \ge b+1+(s-1)(c-1)+\sum_{j=2}^s \chi(w(j)>r)\chi(w(j-1)>r).
\end{align}
Together with the assumption that $k_{w(s)}\leq (s-1)(c-1)+b+t$
this implies that $P$ has at most $t-1$ ascending edges.
Let $d_1,\dots,d_s$ be nonnegative integers such that \eqref{e-k1} holds.
Since (1) does not hold and by \eqref{e-k1-1}, $k_{w(1)}=b+d_1\geq b+1$, so that $d_1\geq 1$.
For $2\leq j\leq s$, if $w(j-1)\to w(j)$ is an ascending edge, then $d_j$ is a positive integer.
That is, for $2\leq j\leq s$ if $w(j-1)<w(j)$ then $d_j\geq 1$.
Set $k_0:=0$.
Since
\begin{align*}
\sum_{j=1}^s (k_{w(j)}-k_{w(j-1)})=k_{w(s)}
&=b+(s-1)(c-1)+\sum_{j=1}^s \big(d_j+\chi(w(j)>r)\chi(w(j-1)>r)\big) \\
\leq b+(s-1)(c-1)+t,
\end{align*}
we have
\[
\sum_{j=1}^s \big(d_j+\chi(w(j)>r)\chi(w(j-1)>r)\big)\leq t.
\]
By Proposition~\ref{prop-lowerbound},
\[
\sum_{j=1}^s \big(d_j+\chi(w(j)>r)\chi(w(j-1)>r)\big)\geq s-r.
\]
This completes the proof of the assertion that (4) must hold if (1)--(3) fail.

If $t=s-r$ and $r<s$, then
\[
\sum_{j=1}^{s}\big(d_j+\chi(w(j)>r)\chi(w(j-1)>r)\big)=s-r.
\]
By Proposition~\ref{prop-lowerbound} we have $w(1)>r$, $d_1=1$ and $k_{w(1)}=b+1$.
That is the assertion \eqref{t-specialcase}.
\end{proof}

For $r=s$ and $t=1$, we have the next corollary of Lemma~\ref{lem-key}.
\begin{cor}\label{cor-key}
For $s,c$ positive integers, let $b$ and $k_1,\dots,k_s$ be nonnegative integers such that $1\leq k_{i}\leq (s-1)(c-1)+b+1$ for $1\leq i\leq s$.
Then, at least one of the following holds:
\begin{enumerate}
\item $1\leq k_i\leq b$ for some $i$ with $1\leq i\leq s$;
\item $-c+1\leq k_i-k_j\leq c-2$ for some $(i,j)$ with $1\leq i<j\leq s$;
\item For $i=1,\dots,s$, $k_i=(s-i)(c-1)+b+1$.
\end{enumerate}
\end{cor}
Note that this corollary corresponds to \cite[Lemma~4.2]{XZ} with minor modification.

By Lemma~\ref{lem-key}, we obtain another corollary.
\begin{cor}\label{prop-specialcase}
For positive integers $s,c$, let $n,n_0,b,m$ and $k_1,\dots,k_s$ be nonnegative integers such that
$1\leq k_{i}\leq (s-1)(c-1)+b+\chi(s>n_0+1)(s-n_0-1)+\chi(s>n-m)$ for $1\leq i\leq s\leq n$.
Let $m=0$ or $0\leq n-n_0\leq m\leq n$.
Then, for an integer $r$ such that $0\leq r\leq \min\{s,n_0\}$, at least one of the following four cases holds:
\begin{itemize}
\item[(i)] $1\leq k_i\leq b$ for some $i$ with $1\leq i\leq s$;
\item[(ii)] $-c+1\leq k_i-k_j\leq c-2$ for some $(i,j)$ such that $1\leq i<j\leq s$ and $i\leq r$;
\item[(iii)] $-c\leq k_i-k_j\leq c-1$ for some $(i,j)$ such that $r<i<j\leq s$;
\item[(iv)] if $m>0$ then $s>n-m$ and $k_i=b+1$ for some $i>n-m$.
\end{itemize}
\end{cor}
Note that only the first three cases can occur for $m=0$.
\begin{proof}
The cases (i)--(iii) of this proposition are exactly the same as (1)--(3) of Lemma~\ref{lem-key} respectively.
Hence, we only need to show the next two claims:
(a) For the $m=0$ case, take $t=\chi(s>n_0+1)(s-n_0-1)$ in Lemma~\ref{lem-key}, then the case (4) of Lemma~\ref{lem-key} can not occur;
(b) For the $0\leq n-n_0\leq m\leq n$ case, take $t=\chi(s>n_0+1)(s-n_0-1)+\chi(s>n-m)$ in Lemma~\ref{lem-key}, then (4) of Lemma~\ref{lem-key} leads to (iv).

Carrying out the substitution $t=\chi(s>n_0+1)(s-n_0-1)+\chi(s>n-m)$ in \eqref{t} gives
\begin{equation}\label{inequality}
s-r\leq \sum_{j=1}^{s}\big(d_j+\chi(w(j)>r)\chi(w(j-1)>r)\big)\leq
\chi(s>n_0+1)(s-n_0-1)+\chi(s>n-m).
\end{equation}
Here $w\in\mathfrak{S}_s$ and $d_1,\dots,d_s$ are nonnegative integers such that
$w(0):=0$, and $d_j>0$ if $w(j-1)<w(j)$ for $1\leq j\leq s$.
Notice that $d_1\geq 1$ since $0=w(0)<w(1)$.
It follows that
\begin{equation}\label{sc-1}
\sum_{j=1}^{s}\big(d_j+\chi(w(j)>r)\chi(w(j-1)>r)\big)\geq 1.
\end{equation}

We first prove the claim (a). If $s\leq n_0+1$ then $t=\chi(s>n_0+1)(s-n_0-1)=0$.
Substituting this and $m=0$ into \eqref{inequality} gives
\begin{equation}\label{e-contr2}
\sum_{j=1}^{s}\big(d_j+\chi(w(j)>r)\chi(w(j-1)>r)\big)\leq 0.
\end{equation}
This contradicts \eqref{sc-1}.
On the other hand, if $s>n_0+1$ and $m=0$ then
\begin{equation}\label{e-contr}
t=\chi(s>n_0+1)(s-n_0-1)=s-n_0-1.
\end{equation}
By \eqref{inequality}, we have $t\geq s-r$.
Together with $r\leq \min\{s,n_0\}$ gives $t\geq s-r\geq s-n_0$.
But this contradicts $t=s-n_0-1$ in \eqref{e-contr}.
Therefore, the claim (a) holds.

To prove the claim (b), we first show that it holds for the $r=s$ case.
For $r=s$, by $r\leq \min\{s,n_0\}$ we have $s\leq n_0$.
It follows that $t=\chi(s>n_0+1)(s-n_0-1)+\chi(s>n-m)=\chi(s>n-m)$.
Furthermore, if $s\leq n-m$ then $t=0$.
By \eqref{inequality} we also get \eqref{e-contr2}, which
contradicts \eqref{sc-1}.
Hence, $s>n-m$ and $t=\chi(s>n-m)=1$.
It follows that $d_1=1$, $d_i=0$ for $i=2,\dots,s$.
This implies that $w=(s,s-1,\dots,1)$ and $k_{w(1)}=k_s=b+1$
by \eqref{e-k1-1} with $d_1=1$.
Since $s>n-m$, we find an $i=s>n-m$ such that $k_i=b+1$.

Assume $m>0$ and $r<s$ in the following of the proof. We prove the claim (b).
We begin with the proof of $s>n-m$. Assume the contrary $s\leq n-m$.
Furthermore, if $s\leq n_0+1$, then $t=\chi(s>n_0+1)(s-n_0-1)+\chi(s>n-m)=0$.
By \eqref{inequality} we get \eqref{e-contr2} again and arrive at a contradiction to \eqref{sc-1}.
If $s>n_0+1$, using the assumption $s\leq n-m$ we get $m<n-n_0-1$. This contradicts the condition
$m\geq n-n_0$ in the claim (b).
Therefore, $s>n-m$.

For $s>n-m$ and $s>n_0+1$, we have
\[
t=\chi(s>n_0+1)(s-n_0-1)+\chi(s>n-m)=s-n_0.
\]
But by $r\leq \min\{s,n_0\}$ we have $s-r\geq s-n_0$.
Then, for $s>n-m$ and $s>n_0+1$, \eqref{inequality} holds only when $r=n_0$, and in this case
\[
\sum_{j=1}^{s}\big(d_j+\chi(w(j)>r)\chi(w(j-1)>r)\big)=s-n_0.
\]
By \eqref{t-specialcase} with $r=n_0$,
we have $k_i=b+1$ for some $i>r$. Here $r=n_0\geq n-m$. Thus, the index $i$ for $k_i=b+1$ satisfies
$i>n-m$.

On the other hand, for $s>n-m$ and $s\leq n_0$, we have
\[
t=\chi(s>n_0+1)(s-n_0-1)+\chi(s>n-m)=1.
\]
Together with \eqref{sc-1}, we find that \eqref{inequality} holds only when
\[
\sum_{j=1}^{s}\big(d_j+\chi(w(j)>r)\chi(w(j-1)>r)\big)=1.
\]
By the same discussion as above, we find an $i=s>n-m$ such that $k_i=k_s=b+1$.
\end{proof}

For positive integers $s,c$ and nonnegative integers $b,r$ such that $0\leq r<s$, let
\begin{equation}\label{N}
T_{s,r}:=(s-1)(c-1)+\chi(s-1>r)(s-r-1)+b+1.
\end{equation}
Note that we suppress the parameters $b$ and $c$ in $T_{s,r}$.
By Lemma~\ref{lem-key} we have the next consequence.
\begin{prop}\label{lem-subs}
For $s,c$ positive integers, let $k_1,\dots,k_s$, $b,r$ and $\mathfrak{t}$ be nonnegative integers such that
$1\leq k_i\leq T_{s,r}+\mathfrak{t}$ for $1\leq i\leq s$ and $0\leq r<s$.
If the $k_i$ are such that (4) of Lemma~\ref{lem-key} holds with
$t\mapsto \chi(s-1>r)(s-r-1)+1+\mathfrak{t}$, then
\begin{equation}\label{e-subs1}
-\frac{q^{c-b-1}-q^a}{1-q}x_0-\sum_{i=1}^{s}\frac{1-q^{c-\chi(i\leq r)}}{1-q}x_i
\bigg|_{\substack{-a=T_{s,r}+\mathfrak{t}, \\[1pt] x_i=q^{k_s-k_i},\,0\leq i\leq s}}
=q^{n_1}+\dots+q^{n_{\mathfrak{t}}},
\end{equation}
where $\{n_1,\dots,n_{\mathfrak{t}}\}$ is a set of integers determined by $s,r,b,c$ and the $k_i$,
and we set $k_0:=0$.
In particular, let $q^{n_1}+\dots+q^{n_{\mathfrak{t}}}|_{\mathfrak{t}=0}:=0$.
\end{prop}

We remark that the set $\{n_1,\dots,n_{\mathfrak{t}}\}$ ($n_i\neq n_j$ for $1\le i<j\le \mathfrak{t}$) can be explicitly determined.
However, since the precise values of the $n_i$ are irrelavent in the following, we
have omitted them from the above statement.
Indeed, the important fact about the right-hand side is that, viewed as an
alphabet, has cardinality $\mathfrak{t}$.

\begin{proof}
Denote the left-hand side of \eqref{e-subs1} by $L$.
Carrying out the substitutions
\[
a\mapsto -T_{s,r}-\mathfrak{t}\quad\text{and}\quad  x_i\mapsto q^{k_s-k_i} \text{ for $0\leq i\leq s$}
\]
in
\[
-\frac{q^{c-b-1}-q^a}{1-q}x_0-\sum_{i=1}^{s}\frac{1-q^{c-\chi(i\leq r)}}{1-q}x_i,
\]
we obtain
\begin{align}
L&=
-\frac{q^{k_s}}{1-q}\bigg(\big(q^{c-b-1}-q^{-T_{s,r}-\mathfrak{t}}\big)
+\sum_{i=1}^s\big(1-q^{c-\chi(i\leq r)}\big)q^{-k_i}\bigg)\nonumber \\
&=-\frac{q^{k_s}}{1-q}\bigg(\big(q^{c-b-1}-q^{-T_{s,r}-\mathfrak{t}}\big)
+\sum_{i=1}^s\big(1-q^{c-\chi\big(w(i)\leq r\big)}\big)q^{-k_{w(i)}}\bigg), \label{e-prop-subs1}
\end{align}
where $w\in \mathfrak{S}_s$ is any permutation such that \eqref{e-k1} holds.
By rearranging the terms in \eqref{e-prop-subs1}, it can be written as
\begin{multline*}
L=\frac{q^{k_s}}{1-q}\bigg(q^{c-\chi\big(w(1)\leq r\big)-k_{w(1)}}\big(1-q^{k_{w(1)}+\chi\big(w(1)\leq r\big)-b-1}\big)
+q^{-T_{s,r}-\mathfrak{t}}\big(1-q^{T_{s,r}+\mathfrak{t}-k_{w(s)}}\big)\\
+\sum_{i=2}^{s}q^{c-\chi\big(w(i)\leq r\big)-k_{w(i)}}\big(1-q^{-c+\chi\big(w(i)\leq r\big)+k_{w(i)}-k_{w(i-1)}}\big)
\bigg).
\end{multline*}

Next we will show that
\begin{subequations}\label{s}
\begin{equation}\label{s1}
k_{w(1)}+\chi\big(w(1)\leq r\big)-b-1\in\mathbb{N},
\end{equation}
\begin{equation}\label{s2}
T_{s,r}+\mathfrak{t}-k_{w(s)}\in\mathbb{N},
\end{equation}
and
\begin{equation}\label{s3}
\mathfrak{p}_i:=-c+\chi\big(w(i)\leq r\big)+k_{w(i)}-k_{w(i-1)}\in\mathbb{N} \quad\text{for $i=2,3,\dots,s$},
\end{equation}
\end{subequations}
where $\mathbb{N}=\{0,1,2,\dots\}$.
By \eqref{e-k1-1} we have $k_{w(1)}\geq b+1$. This implies \eqref{s1}.
By the assumption $k_i\leq T_{s,r}+\mathfrak{t}$ for all $i$,
we have \eqref{s2}.
By \eqref{e-k1-2},
\begin{equation}\label{s4}
\mathfrak{p}_i=\chi\big(w(i)\leq r\big)+\chi\big(w(i)>r\big)\chi\big(w(i-1)>r\big)+d_i-1
\end{equation}
for $i=2,\dots,s$.
If $w(i-1)<w(i)$, then $d_i\geq 1$. It follows that $\mathfrak{p}_i\geq 0$.
On the other hand, if $w(i-1)>w(i)$, then it is easy to see that
\begin{equation*}
\chi(w(i)>r)\chi(w(i-1)>r)=\chi(w(i)>r)=1-\chi(w(i)\leq r).
\end{equation*}
Substituting this into \eqref{s4} gives $\mathfrak{p}_i=d_i$.
Together with $d_i\geq 0$ yields $\mathfrak{p}_i\geq 0$.
Thus \eqref{s3} holds.

Since \eqref{s} holds,
and $(1-q^n)/(1-q)=1+\dots+q^{n-1}$ for $n$ a positive integer and is 0 for $n=0$, we conclude
that $L=q^{n_1}+\dots+q^{n_p}$,
where $p$ is given by
\[
k_{w(1)}+\chi\big(w(1)\leq r\big)-b-1+T_{s,r}+\mathfrak{t}-k_{w(s)}+\sum_{i=2}^s\Big(-c+\chi\big(w(i)\leq r\big)+k_{w(i)}-k_{w(i-1)}\Big).
\]
This can be written as
\[
-b-1+T_{s,r}+\mathfrak{t}-(s-1)c+\sum_{i=1}^s\chi\big(w(i)\leq r\big)
=-b-1+T_{s,r}+\mathfrak{t}-(s-1)c+r.
\]
Substituting the expression \eqref{N} for $T_{s,r}$ into above gives
\[
p=\chi(s-1>r)(s-r-1)-s+1+\mathfrak{t}+r.
\]
If $s-1>r$ then
\[
p=(s-r-1)-s+1+\mathfrak{t}+r=\mathfrak{t}.
\]
If $r=s-1$ then
\[
p=-s+1+\mathfrak{t}+s-1=\mathfrak{t}.
\]
Thus $p=\mathfrak{t}$ for $0\leq r<s$, completing the proof.
\end{proof}

\subsection{The rational function $Q_{n,n_0}(d,b,c;M_{\mathfrak{r}}\Mid u;k)$}

In this subsection, we are concerned with the rational function $Q_{n,n_0}(d,b,c;M_{\mathfrak{r}}\Mid u;k)$,
which is defined in \eqref{eq-Qrk} below.

For nonnegative integers $n,n_0,d,b,\mathfrak{r}$ and a positive integer $c$, define
\begin{equation}\label{def-Q3}
Q(d)=Q_{n,n_0}(d,b,c;M_{\mathfrak{r}})
:=x_0^{-\mathfrak{r}}M_{\mathfrak{r}}
\prod_{j=1}^{n}\frac{\big(qx_{j}/x_{0}\big)_b}
{\big(q^{-d}x_{0}/x_{j}\big)_{d}}
\prod_{1\leq i<j\leq n}
\big(x_{i}/x_{j}\big)_{c-\chi(i\leq n_0)}\big(qx_{j}/x_{i}\big)_{c-\chi(i\leq n_0)},
\end{equation}
where $M_{\mathfrak{r}}$ is a homogenous polynomial in $x_0,x_1,\dots,x_n$ of degree
$\mathfrak{r}$.
It is clear that
\[
\CT_x Q(d)=B_{n,n_0}(-d,b,c,l,\mu)
\]
by taking
\[
M_{\mathfrak{r}}=h_{l}\Big[\sum_{i=1}^{n}\frac{1-q^{c-\chi(i\leq n_0)}}{1-q}x_i\Big] \\
\times P_{\mu}\Big(\Big[\frac{q^{c-b-1}-q^{-d}}{1-q^c}x_0+\sum_{i=1}^{n}\frac{1-q^{c-\chi(i\leq n_0)}}{1-q^{c}}x_i\Big];q,q^c\Big)
\]
for $\mathfrak{r}=l+|\mu|$.
One can also see that
\[
\CT_x Q(d)=C_{n,n_0}(-d,b,c,l,m)
\]
by taking
\[
M_{\mathfrak{r}}=h_{l}\Big[\sum_{i=1}^{n}\frac{1-q^{c-\chi(i\leq n_0)}}{1-q}x_i\Big] \\
\times \prod_{j=n-m+1}^n(x_0-q^{b+1}x_j)
\]
for $\mathfrak{r}=l+m$.

For any rational function $F$ of $x_0, x_1, \dots, x_n$ and $s$ an integer such that $1\leq s\leq n$, and for
sequences of integers $k = (k_1, k_2, \dots, k_s)$ and $u = (u_1,
u_2, \dots, u_s)$ let $E_{u,k}F$ be the result of
replacing $x_{u_i}$ in $F$ with $x_{u_s}q^{k_s-k_i}$ for $i = 0,
1,\dots, s-1$, where we set $u_0 = k_0 = 0$. Then for $0 < u_1 <
u_2 <\dots < u_s \leq n$ and $1\leq k_i\leq d$, we define
\begin{equation}\label{eq-Qrk}
Q(d\Mid u;k)=Q_{n,n_0}(d,b,c;M_{\mathfrak{r}}\Mid u;k)
:=E_{u,k}\bigg(Q(d)\prod_{i=1}^{s}(1-\frac{x_{0}}{x_{u_{i}}q^{k_{i}}})\bigg).
\end{equation}
Set $Q(d\Mid u;k)|_{s=0}:=Q(d)$.
Note that the product in \eqref{eq-Qrk} cancels all the factors in the denominator of $Q(d)$ that would be
taken to zero by $E_{u,k}$.

The numerator of $Q(d)$ ---
\[
x_0^{-\mathfrak{r}}M_{\mathfrak{r}}
\prod_{j=1}^{n}\big(qx_{i}/x_{0}\big)_b
\prod_{1\leq i<j\leq n}
\big(x_{i}/x_{j}\big)_{c-\chi(i\leq n_0)}(qx_{j}/x_{i}\big)_{c-\chi(i\leq n_0)}
\]
is a Laurent polynomial in $x_0$ of non-positive degree.
The denominator of $Q(d)$ ---
\[
\prod_{j=1}^{n}\big(q^{-d}x_{0}/x_{j}\big)_{d}
\]
is of the form
\[
\prod_{r=1}^{nd} (1-c_r x_0/ x_{i_r}),
\]
where all the $i_r\neq 0$, and $c_r\neq c_v$ if $i_r=i_v$.
It is a polynomial in $x_0$ of degree $nd$.
Thus, for $d$ a positive integer $Q(d)$ is a rational function of the form
\eqref{e-defF} with respect to $x_0$. Then we can apply Lemma~\ref{lem-almostprop}
to $Q(d)$ with respect to $x_0$ and obtain
\begin{equation}\label{Q3-1}
\CT_{x_0}Q(d)=\sum_{\substack{1\leq k_1\leq d\\1\leq u_1\leq n}}
Q(d\Mid u_1;k_1),
\end{equation}
where $Q(d\Mid u_1;k_1)$ is defined in \eqref{eq-Qrk} with $s=1$.
We can further apply Lemma~\ref{lem-almostprop} to each $Q(d\Mid u_1;k_1)$ with respect to $x_{u_1}$ if applicable, and get a sum.
Continue this operation until Lemma~\ref{lem-almostprop} does not apply to every summand. In other words,  Lemma~\ref{lem-almostprop} is no longer valid to every summand. Finally we write
\begin{equation}\label{Q3-sum}
\CT_{x}Q(d)=\sum_{s\in T}\sum_{\substack{1\leq u_1<\cdots<u_s\leq n\\1\leq k_1,\dots,k_s\leq d}}
\CT_xQ(d\Mid u;k),
\end{equation}
where $T$ is a subset of $\{1,\dots,n\}$.
We call this operation the Gessel--Xin operation to the rational function $Q(d)$, since it first appeared in \cite{GX}.

Let $r$ be the integer such that
$u_1,\dots,u_r\leq n_0$ and $u_{r+1},\dots,u_s>n_0$.
That is
\begin{equation}\label{r}
r:=\big|\{u_i\Mid u_i\leq n_0, i=1,2,\dots,s\}\big|,
\end{equation}
where $|S|$ is the number of the elements in the set $S$.
Denote $U:=\{u_1,u_2,\dots,u_s\}$.
Then $Q(d\Mid u;k)$ can be written as
\begin{equation}\label{defi-Q}
Q(d\Mid u;k)=H\times V\times L \times
\prod_{i=1}^s(q^{k_i-d})^{-1}_{d-k_i}(q)^{-1}_{k_i-1}
\prod_{\substack{1\leq i<j\leq n\\i,j\notin U}}\big(x_i/x_j\big)_{c-\chi(i\leq n_0)}
(qx_{j}/x_{i}\big)_{c-\chi(i\leq n_0)},
\end{equation}
where
\[
H=E_{u,k}\big(x_{0}^{-\mathfrak{r}}\times M_{\mathfrak{r}}\big),
\]
\[
V=\prod_{i=1}^s(q^{1-k_i})_b\prod_{\substack{1\leq i<j\leq s\\i\leq r}}(q^{k_j-k_i})_{c-1}
(q^{k_i-k_j+1})_{c-1}
\prod_{r+1\leq i<j\leq s}(q^{k_j-k_i})_{c}(q^{k_i-k_j+1})_{c},
\]
and
\begin{multline*}
L=\prod_{\substack{i=1\\ i\notin U}}^n
\frac{(q^{1-k_s}x_i/x_{u_s})_{b}}{\big(q^{k_s-d}x_{u_s}/x_i\big)_d}
\prod_{\substack{i=1\\ i\notin U}}^n\prod_{j=1}^r
\big(q^{k_j-k_s+\chi(i>u_j)}x_i/x_{u_s}\big)_{c-1}
\big(q^{k_s-k_j+\chi(u_j>i)}x_{u_s}/x_i\big)_{c-1}\\
\times \prod_{\substack{i=1\\ i\notin U}}^n\prod_{j=r+1}^s
\big(q^{k_j-k_s+\chi(i>u_j)}x_i/x_{u_s}\big)_{c-\chi(i\leq n_0)}
\big(q^{k_s-k_j+\chi(u_j>i)}x_{u_s}/x_i\big)_{c-\chi(i\leq n_0)}.
\end{multline*}

The following properties of $Q(d\Mid u;k)$ is crucial in proving the main theorems.
\begin{lem}\label{lem-Q}
Let $Q(d\Mid u;k)$ and $r$ be defined in \eqref{eq-Qrk} and \eqref{r} respectively.
Assume $\prod_{j=n-m+1}^n(x_0-q^{b+1}x_j)|M_{\mathfrak{r}}$ for $m$ an integer such that $n-n_0\leq m\leq n$.
Then for $s$ an integer such that $1\leq s\leq n$, the rational function $Q(d\Mid u;k)$ has the following properties:
\begin{enumerate}
\item If $d\leq (s-1)(c-1)+b+\chi(s>n_0+1)(s-n_0-1)+\chi(s>n-m)$, then $Q(d\Mid u;k)=0$;

\item If $d>s(c-1)+(n-n_0-s+r)(s-r)/(n-s)$ and $s\neq n$, then
\begin{equation}\label{Q1}
\CT_{x_{u_s}}Q(d\Mid u;k)=
\begin{cases}\displaystyle
\sum_{\substack{u_s<u_{s+1}\leq n\\1\leq k_{s+1}\leq d}}
Q(d\Mid u_1,\dots,u_s,u_{s+1};k_1,\dots,k_s,k_{s+1}) \quad &\text{for $u_s<n$,}\\
0 \quad &\text{for $u_s=n$;}
\end{cases}
\end{equation}

\item If $s(c-1)+\chi(s>n_0)(s-n_0)+1\leq d\leq s(c-1)+(n-n_0-s+r)(s-r)/(n-s)$ and $s\neq n$, then
\[
\CT_x Q(d\Mid u;k)=0.
\]
\end{enumerate}
\end{lem}
Note that for $n_0=0$ and 1, Case (3) of Lemma~\ref{lem-Q} does not occur since the lower bound of $d$ exceeds its upper bound by routine calculations.
\begin{proof}
(1)
Since $1\leq k_i\leq d$ for $i=1,\dots,s$, if $d$ satisfies the condition in (1) then
\[
1\leq k_i\leq (s-1)(c-1)+b+\chi(s>n_0+1)(s-n_0-1)+\chi(s>n-m) \quad \text{for $i=1,\dots,s$}.
\]
By the definition of $r=\big|\{u_i\Mid u_i\leq n_0, i=1,2,\dots,s\}\big|$,
it is clear that $0\leq r\leq \min\{s,n_0\}$. Then, at least one of the cases (i)--(iv) of
Corollary~\ref{prop-specialcase} holds.

If (i) holds, then $Q(d\Mid u;k)$ has the factor
\[
E_{u,k}\Big[\big(qx_{u_i}/x_0\big)_b\Big]=\big(q^{1-k_i}\big)_b=0.
\]

If (ii) holds, then $Q(d\Mid u;k)$ has the factor
\[
E_{u,k}\Big[\big(x_{u_i}/x_{u_j}\big)_{c-1}\big(qx_{u_j}/x_{u_i}\big)_{c-1}\Big],
\]
which is equal to
\[
E_{u,k}\Big[q^{\binom{c}{2}}(-x_{u_j}/x_{u_i})^{c-1}\big(q^{1-c}x_{u_i}/x_{u_j}\big)_{2c-2}\Big]
=q^{\binom{c}{2}}(-q^{k_i-k_j})^{c-1}\big(q^{k_j-k_i-c+1}\big)_{2c-2}=0.
\]

If (iii) holds, then $Q(d\Mid u;k)$ has the factor
\[
E_{u,k}\Big[\big(x_{u_i}/x_{u_j}\big)_c\big(qx_{u_j}/x_{u_i}\big)_c\Big],
\]
which is equal to
\[
E_{u,k}\Big[q^{\binom{c+1}{2}}(-x_{u_j}/x_{u_i})^{c}\big(q^{-c}x_{u_i}/x_{u_j}\big)_{2c}\Big]
=q^{\binom{c+1}{2}}(-q^{k_i-k_j})^{c}\big(q^{k_j-k_i-c}\big)_{2c}=0.
\]

Since $\prod_{j=n-m+1}^n(x_0-q^{b+1}x_j)=x_0^{m}\prod_{j=n-m+1}^n(1-q^{b+1}x_j/x_0)$ is a factor of $M_{\mathfrak{r}}$, if (iv) holds then $Q(d\Mid u;k)$ has the factor
\[
E_{u,k}\Big[(qx_{u_{i}}/x_0)_{b+1}\Big]=\big(q^{1-k_{i}}\big)_{b+1}
=(q^{-b})_{b+1}=0.
\]
Here $u_i>n-m$ because $u_i\geq i>n-m$.

\vskip 0.2cm

(2)
We first show that $Q(d\Mid u;k)$ is of the form \eqref{e-defF} for $d>s(c-1)+(s-r)(n-n_0-s+r)/(n-s)$ and $s\neq n$.

Recall that $U=\{u_1,u_2,\dots,u_s\}$. By the expression for $Q(d\Mid u;k)$ in \eqref{defi-Q}, the parts contribute to the degree in $x_{u_s}$
of the numerator of  $Q(d\Mid u;k)$ is
\[
H\times \prod_{\substack{i=1\\ i\notin U}}^n\bigg(\prod_{j=1}^r
\big(q^{k_s-k_j+\chi(u_j>i)}x_{u_s}/x_i\big)_{c-1}
\prod_{j=r+1}^s\big(q^{k_s-k_j+\chi(u_j>i)}x_{u_s}/x_i\big)_{c-\chi(i\leq n_0)}\bigg),
\]
which has degree at most
\begin{align*}
s(n-s)(c-1)+\sum_{i\notin U}\sum_{j=r+1}^s\chi(i>n_0)
&=s(n-s)(c-1)+\sum_{\substack{i=n_0+1\\i\notin U}}^n\sum_{j=r+1}^s1 \\
&=s(n-s)(c-1)+(n-n_0-s+r)(s-r).
\end{align*}
The parts contribute to the degree in $x_{u_s}$ of the denominator of $Q(d\Mid u;k)$ is
\[
\prod_{\substack{i=1\\ i\notin U}}^n
\big(q^{k_s-d}x_{u_s}/x_i\big)_d,
\]
which has degree $(n-s)d$.
If $d>s(c-1)+(n-n_0-s+r)(s-r)/(n-s)$ then
\[
s(n-s)(c-1)+(n-n_0-s+r)(s-r)<(n-s)d.
\]
Thus $Q(d\Mid u;k)$ is of the form \eqref{e-defF}.
Applying Lemma~\ref{lem-almostprop} gives
\[
\CT_{x_{u_s}}Q(d\Mid u;k)=
\begin{cases}\displaystyle
\sum_{\substack{u_s<u_{s+1}\leq n\\1\leq k_{s+1}\leq d}}
Q(d\Mid u_1,\dots,u_s,u_{s+1};k_1,\dots,k_s,k_{s+1}) \quad &\text{for $u_s<n$,}\\
0 \quad &\text{for $u_s=n$.}
\end{cases}
\]

(3) The proof of the Case (3) is lenthy, we give the proof in Lemma~\ref{Case3} below.
\end{proof}

A direct consequence of Lemma~\ref{lem-Q} is the next result.
\begin{prop}\label{prop-roots-1}
Let $C_1$, $C_3$ and $Q(d)$ be defined in \eqref{Roots-B}, \eqref{Roots-C} and \eqref{def-Q3} respectively.
Assume $\prod_{j=n-m+1}^n(x_0-q^{b+1}x_j)|M_{\mathfrak{r}}$ for $m$ an integer such that
$n-n_0\leq m\leq n$.
If $-d\in C_1\cup C_3$, then $\CT\limits_x Q(d)=0$.
\end{prop}
Note that if $m=0$ then $C_3=\emptyset$ and $\prod_{j=n-m+1}^n(x_0-q^{b+1}x_j):=1$.
It is not hard to see that $B_{n,n_0}(a,b,c,l,\mu)$
vanishes for $a\in B_1$, and $C_{n,n_0}(a,b,c,l,m)$ vanishes for $a\in C_1\cup C_3$
by this proposition.
\begin{proof}
We prove by induction on $n-s$ that
\begin{equation}\label{prop-roots13-1}
\CT_{x} Q(d\Mid u;k)=0 \quad  \text{for $-d\in C_1\cup C_3$}.
\end{equation}
The proposition is the $s=0$ case of \eqref{prop-roots13-1}. Note that taking the constant
term with respect to a variable that does not appear has no effect.
We may assume that $0\leq s\leq n$ and $0<u_1<\cdots<u_s\leq n$, since
otherwise $Q(d\Mid u;k)$ is not defined. If $s=n$ then $u_i$ must equal $i$
for $i=1,\dots ,n$ and thus
\[
Q(d\Mid u;k)=Q(d\Mid 1,2,\dots,n;k_1,k_2,\dots,k_n).
\]
Since $d\leq(n-1)(c-1)+b+\chi(n>n_0+1)(n-n_0-1)+\chi(n>n-m)$,
by the property (1) of Lemma~\ref{lem-Q} with $s=n$ we have
\[
Q(d\Mid 1,2,\dots,n;k_1,k_2,\dots,k_n)=0.
\]

Now suppose that $0\leq s<n$.
If $d\leq (s-1)(c-1)+b+\chi(s>n_0+1)(s-n_0-1)+\chi(s>n-m)$, then
by the property (1) of Lemma~\ref{lem-Q} we have $Q(d\Mid u;k)=0$.
If $s(c-1)+\chi(s>n_0)(s-n_0)+1\leq d\leq s(c-1)+(s-r)(n-n_0-s+r)/(n-s)$, then
by the property (3) of Lemma~\ref{lem-Q}
\[
\CT_x Q(d\Mid u;k)=0.
\]
If $d>s(c-1)+(n-n_0-s+r)(s-r)/(n-s)$, then by the property (2) of Lemma~\ref{lem-Q}
\[
\CT_{x_{u_s}}Q(d\Mid u;k)
=
\begin{cases}\displaystyle
\sum_{\substack{u_s<u_{s+1}\leq n\\1\leq k_{s+1}\leq d}}
Q(d\Mid u_1,\dots,u_s,u_{s+1};k_1,\dots,k_s,k_{s+1}) \quad &\text{for $u_s<n$,}\\
0 \quad &\text{for $u_s=n$.}
\end{cases}
\]
For $u_s<n$, applying $\CT\limits_x$ to both sides of the above equation gives
\[
\CT_xQ(d\Mid u;k)=\sum_{\substack{u_s<u_{s+1}\leq n\\1\leq k_{s+1}\leq d}}
\CT_xQ(d\Mid u_1,\dots,u_s,u_{s+1};k_1,\dots,k_s,k_{s+1}).
\]
Each summand in the above is 0 by induction.

In conclusion, we show that $\CT\limits_x Q(d\Mid u;k)=0$ for all $-d\in C_1\cup C_3$,
completing the proof.
\end{proof}

\section{Proof of Theorem~\ref{thm-qF2}}\label{sec-proof1}

We prove Theorem~\ref{thm-qF2} along the outline in the introduction.
We show that $B_{n,n_0}(a,b,c,l,\mu)$ is a polynomial in $q^a$ of degree at most
$nb+l+|\mu|$ in Corollary~\ref{cor-poly-BC}. In Corollary~\ref{cor-rationality} we prove that
if \eqref{main-B} holds for sufficiently many integers $c$ (here for $c>b+\mu_1$), then it holds for all positive integers $c$. We can show that $B_{n,n_0}(a,b,c,l,\mu)=0$
for $a\in B_1$ and $a\in B_2$
by Proposition~\ref{prop-roots-1} and Lemma~\ref{lem-Roots-2} respectively.
To complete the proof of Theorem~\ref{thm-qF2},
it remains to prove that $B_{n,n_0}(a,b,c,l,\mu)=0$ for $a\in B_3$ and
to determine a closed-form expression for $B_{n,n_0}(a,b,c,l,\mu)$ at an additional point of $a$.
We show these in the next two subsections respectively.

\subsection{The roots $B_3$ of $B_{n,n_0}(a,b,c,l,\mu)$}

In this subsection, we show that the constant term $B_{n,n_0}(a,b,c,l,\mu)$ vanishes
for $a\in B_3$. The result is stated in the next lemma.
\begin{lem}\label{lem-roots-B3}
Let $B_{n,n_0}(a,b,c,l,\mu)$ and $B_3$ be defined in \eqref{eq-qF2} and \eqref{Roots-B} respectively.
If $\ell(\mu)<n-n_0$ and $c>b+\mu_1$, then $B_{n,n_0}(a,b,c,l,\mu)=0$ for $a\in B_3$.
\end{lem}
\begin{proof}
Assume that $B_3\neq \emptyset$, i.e., $\mu\neq 0$. Recall the definition of $Q(d)=Q_{n,n_0}(d,b,c;M_{\mathfrak{r}})$ in \eqref{def-Q3}. In the following of the proof we assume that the $M_{\mathfrak{r}}$ in $Q(d)$ is
\begin{equation}\label{Mr-for-B2}
h_{l}\Big[\sum_{i=1}^{n}\frac{1-q^{c-\chi(i\leq n_0)}}{1-q}x_i\Big]
P_{\mu}\Big(\Big[\frac{q^{c-b-1}-q^{-d}}{1-q^{c}}x_0
+\frac{1-q^{c-1}}{1-q^{c}}\sum_{i=1}^{n_0}x_i+\sum_{i=n_0+1}^nx_i\Big];q,q^c\Big).
\end{equation}
Then
\begin{equation}\label{Q_2}
\CT_x Q(d)=B_{n,n_0}(-d,b,c,l,\mu)
\end{equation}
for $d$ an integer.
Hence, we prove the lemma by showing that
\begin{equation}\label{e-B3-0}
\CT_x Q(d)=0 \quad \text{for $-d\in B_3$}.
\end{equation}
Notice that if $B_3$ is not an empty set, then all the elements of $B_3$ are distinct negative integers for
$c>b+\mu_1$.
For $d$ a positive integer, applying the Gessel--Xin operation to $Q(d)$ and using \eqref{Q3-sum}, we can write
\begin{equation}\label{e-B3-1}
\CT_{x}Q(d)=\sum_{s\in T\subseteq \{1,\dots,n\}}\sum_{\substack{1\leq u_1<\cdots<u_s\leq n\\1\leq k_1,\dots,k_s\leq d}}
\CT_xQ(d\Mid u;k),
\end{equation}
where $Q(d\Mid u;k)$ is defined in \eqref{eq-Qrk} with $M_{\mathfrak{r}}$ defined in \eqref{Mr-for-B2}.
Note that each $s$ in $Q(d\Mid u;k)$ is maximal by the Gessel--Xin operation.
Denote
\[
B_{3j}=\{-(n-j)c+n_0-b-1,\dots, -(n-j)c+n_0-b-\mu_j\}
\]
for $j=1,\dots,\ell(\mu)$.
Then $B_3=\cup_{j=1}^{\ell(\mu)}B_{3j}$.
We will show that every $\CT\limits_x Q(d\Mid u;k)$ in \eqref{e-B3-1} vanishes
for $-d\in B_{3j}$ for $j=1,\dots,\ell(\mu)$.

We first show that for a fixed $j\in \{1,2,\dots,\ell(\mu)\}$ and $-d\in B_{3j}$,
if $s\neq n-j+1$ then the constant term $\CT\limits_x Q(d\Mid u;k)$ in \eqref{e-B3-1} vanishes.

Since $-d\in B_{3j}$, we have
\[
(n-j)c-n_0+b+1\leq d \leq (n-j)c-n_0+b+\mu_j.
\]
If $s\leq n-j$ then
\[
d\geq (n-j)c-n_0+b+1=(n-j)(c-1)+n-j-n_0+b+1\geq s(c-1)+\chi(s>n_0)(s-n_0)+1.
\]
Here the last inequality holds by discussing the two cases $s>n_0$ and $s\leq n_0$
with the fact that $b\geq 0$, $n-j>n_0$ and $s\leq n-j$.
By the above inequality and using Lemma~\ref{lem-Q} with $m=0$ and $M_{\mathfrak{r}}$ equals the polynomial in \eqref{Mr-for-B2},
we conclude that only the cases (2) and (3) of Lemma~\ref{lem-Q} can occur.
If the case (3) holds, then $\CT\limits_x Q(d\Mid u;k)=0$.
Otherwise, the case (2) holds, and
$\CT\limits_{x_{u_s}} Q(d\Mid u;k)$ is either zero or a sum.
It can not be written as a sum as the form of the nonzero part of \eqref{Q1}
since the $s$ in $Q(d\Mid u;k)$ is maximal by the Gessel--Xin operation.
Hence, we have $\CT\limits_x Q(d\Mid u;k)=0$ for $s\leq n-j$.
On the other hand, if $s\geq n-j+2$ then $s>n_0+1$ since $j\leq \ell(\mu)<n-n_0$.
It follows that
\begin{multline*}
(s-1)(c-1)+b+\chi(s>n_0+1)(s-n_0-1)
=(s-1)(c-1)+b+(s-n_0-1)\geq (n-j+1)c+b-n_0\\
\geq (n-j)c-n_0+2b+1+\mu_1>d.
\end{multline*}
The second to the last inequality holds by the assumption $c>b+\mu_1$.
Using Lemma~\ref{lem-Q} with $m=0$, $M_{\mathfrak{r}}$ equals the polynomial in \eqref{Mr-for-B2}
again and by the fact that $d<(s-1)(c-1)+b+\chi(s>n_0+1)(s-n_0-1)$, we can conclude that the case (1) of Lemma~\ref{lem-Q} holds and $\CT\limits_x Q(d\Mid u;k)=0$ for $s\geq n-j+2$.

In conclusion, the $s$ in \eqref{e-B3-1} can only be $n-j+1$ if $-d\in B_{3j}$ for a fixed integer $j\in \{1,2,\dots,\ell(\mu)\}$. Therefore, \eqref{e-B3-1} reduces to
\begin{equation}\label{e-B3-3}
\CT_{x}Q(d)=\sum_{\substack{1\leq u_1<\cdots<u_s\leq n\\1\leq k_1,\dots,k_s\leq d}}
\CT_xQ(d\Mid u;k),
\end{equation}
where $s=n-j+1$.
We continue the proof by showing that
every $\CT\limits_xQ(d\Mid u;k)$ in \eqref{e-B3-3} vanishes for $-d\in B_{3j}$ and $s=n-j+1$.

Now Lemma~\ref{lem-Q} does not apply. Notice that $s=n-j+1>n_0+1$ since $j\leq \ell(\mu)<n-n_0$.
For $-d\in B_{3j}$,
we can write
\[
d=(n-j)(c-1)+b+t=(s-1)(c-1)+b+t
\]
for
\begin{equation}\label{range-t}
t\in \{n-j-n_0+1,\dots,n-j-n_0+\mu_j\}.
\end{equation}
Since the $k_i$ in \eqref{e-B3-3} satisfy $1\leq k_i\leq d$ for $i=1,\dots,s$,
we can get $1\leq k_i\leq (s-1)(c-1)+b+t$ for all $i$.
Using Lemma~\ref{lem-key} with
\begin{equation}\label{defi-r-B}
r=\big|\{u_i\Mid u_i\leq n_0, i=1,2,\dots,s\}\big|,
\end{equation}
we can conclude that if one of the cases (1)--(3) of Lemma~\ref{lem-key} holds,
then $Q(d\Mid u;k)=0$ by the same argument as that in the part (1) of the proof of Lemma~\ref{lem-Q}.
We then proceed the proof by discussing the case (4) of Lemma~\ref{lem-key} with $r$ defined in \eqref{defi-r-B}.
If $t-s+r<0$ then \eqref{t} can not hold and this case does not occur.
Hence we can assume $t-s+r\geq 0$ in the following.

If (4) of Lemma~\ref{lem-key} holds (this implies that $t-s+r\geq 0$), then by Proposition~\ref{lem-subs} with $r$ defined in \eqref{defi-r-B}, $T_{s,r}=(s-1)(c-1)+s-r+b$ and $\mathfrak{t}=t-s+r$, we have
\begin{multline}\label{e-B3-2}
\frac{q^{c-b-1}-q^{-d}}{1-q^c}x_0
+\frac{1-q^{c-1}}{1-q^c}\sum_{i=1}^{n_0}x_i+\sum_{i=n_0+1}^nx_i\bigg|_{\substack{d=(s-1)(c-1)+b+t, \\[1pt] x_{u_i}=x_{u_s}q^{k_s-k_i},\,0\leq i\leq s}}\\
=\frac{q-1}{1-q^c}(q^{n_1}+\cdots+q^{n_{\mathfrak{t}}})x_{u_s}
+\frac{1-q^{c-1}}{1-q^c}\sum_{\substack{i=1\\ i \notin U}}^{n_0}x_i
+\sum_{\substack{i=n_0+1\\ i\notin U}}^nx_i,
\end{multline}
where $\{n_1,\dots,n_{\mathfrak{t}}\}$ is a set of integers and $U=\{u_1,\dots,u_s\}$.
Since
\[
\frac{1-q^{c-1}}{1-q^c}=1+q^{c-1}\frac{q-1}{1-q^c},
\]
we can further write the right-hand side of \eqref{e-B3-2} as
\[
\frac{q-1}{1-q^c}\Big((q^{n_1}+\cdots+q^{n_{\mathfrak{t}}})x_{u_s}
+q^{c-1}\sum_{\substack{i=1\\ i \notin U}}^{n_0}x_i\Big)
+\sum_{\substack{i=1\\ i\notin U}}^nx_i.
\]
This is of the form
\[
\frac{q-1}{1-q^{c+1}}(\alpha_1+\cdots+\alpha_{i_1})+(\beta_1+\cdots+\beta_{i_2}),
\]
where
\begin{equation}\label{i_1}
i_1=\mathfrak{t}+n_0-r=t-s+r+n_0-r=t-(n-j+1)+n_0\leq \mu_j-1
\end{equation}
and $i_2=n-s=j-1$.
The last inequality in \eqref{i_1} holds by \eqref{range-t}.
Therefore,
\begin{equation}\label{e-factor-B3}
P_{\mu}\Big(\Big[\frac{q^{c-b-1}-q^{-d}}{1-q^c}x_0
+\frac{1-q^{c-1}}{1-q^c}\sum_{i=1}^{n_0}x_i+\sum_{i=n_0+1}^nx_i\Big];q,q^c\Big)
\bigg|_{\substack{d=(s-1)(c-1)+b+t, \\[1pt] x_{u_i}=x_{u_s}q^{k_s-k_i},\,0\leq i\leq s}}
\end{equation}
is of the form
\[
P_{\mu}\Big(\Big[\frac{q-1}{1-q^c}(\alpha_1+\cdots+\alpha_{i_1})+(\beta_1+\cdots+\beta_{i_2})\Big];q,q^c\Big),
\]
where $0\leq i_1\leq \mu_j-1$ and $i_2=j-1$.
This is zero by Proposition~\ref{prop-Mac-vanish}.
Since $Q(d\Mid u;k)$ has \eqref{e-factor-B3} as a factor,
we can obtain that $Q(d\Mid u;k)=0$ for $-d\in B_{3j}$.
It follows that $\CT\limits_xQ(d)=0$ for $-d\in B_{3j}$ by \eqref{e-B3-3}.
Since $j$ is over $\{1,\dots,\ell(\mu)\}$,
\eqref{e-B3-0} holds and we complete the proof.
\end{proof}

\subsection{An expression for $B_{n,n_0}(a,b,c,l,\mu)$ at an additional point}

In this subsection, we obtain an explicit closed-form expression for $B_{n,n_0}(a,b,c,l,\mu)$
at $-a=n_0(c-1)+b+1$ for $\ell(\mu)<n-n_0$.

In the following of this subsection, we assume that the $M_\mathfrak{r}$ in $Q(d)=Q_{n,n_0}(d,b,c;M_{\mathfrak{r}})$ is the polynomial in \eqref{Mr-for-B2} and denote $d=n_0(c-1)+b+1$.
Then $\CT\limits_x Q(d)=B_{n,n_0}(-d,b,c,l,\mu)$.
Since $d=n_0(c-1)+b+1>0$, we can use \eqref{e-B3-1} again and get
\begin{equation}\label{extra-B-1}
\CT_{x}Q(d)=\sum_{s\in T\subseteq \{1,\dots,n\}}\sum_{\substack{1\leq u_1<\cdots<u_s\leq n\\1\leq k_1,\dots,k_s\leq d}}
\CT_xQ(d\Mid u;k),
\end{equation}
where $Q(d\Mid u;k)$ is defined in \eqref{eq-Qrk} and the $s$ in
each $\CT\limits_xQ(d\Mid u;k)$ is maximal by the Gessel--Xin operation.
We will show that the constant term $\CT\limits_xQ(d\Mid u;k)$ in \eqref{extra-B-1} does not vanish only when $s=n_0+1$ and $u_i=i$ for $i=1,\dots,n_0$. (Note that in this case $u_{n_0+1}$ can be any element of $\{n_0+1,n_0+2,\dots,n\}$.)

We first show that if $s\neq n_0+1$ then $\CT\limits_xQ(d\Mid u;k)=0$.
If $s\leq n_0$, then
\[
d=n_0(c-1)+b+1\geq s(c-1)+b+1\geq s(c-1)+1.
\]
By Lemma~\ref{lem-Q} with $m=0$, $Q(d\Mid u;k)$ must satisfy either the case (2) or the case (3) of Lemma~\ref{lem-Q}.
If the case (3) holds, then $\CT\limits_x Q(d\Mid u;k)=0$.
Otherwise, the case (2) holds.
But the nonzero part of \eqref{Q1} can not hold since the $s$ in $Q(d\Mid u;k)$ is maximal.
Therefore, $\CT\limits_xQ(d\Mid u;k)=0$ for $s\leq n_0$.
On the other hand, if $s\geq n_0+2$ then
\[
d=n_0(c-1)+b+1\leq
(s-1)(c-1)+b+(s-n_0-1).
\]
Hence, the case (1) of Lemma~\ref{lem-Q} (with $m=0$ again) holds and $Q(d\Mid u;k)=0$.
In conclusion, the $s$ in \eqref{extra-B-1} can only be $n_0+1$ and the equation reduces to
\begin{equation}\label{extra-B-2}
\CT_{x}Q(d)=\sum_{\substack{1\leq u_1<\cdots<u_{n_0+1}\leq n\\1\leq k_1,\dots,k_{n_0+1}\leq d}}
\CT_xQ(d\Mid u;k).
\end{equation}

By Corollary~\ref{cor-key} with $s=n_0+1$, if $1\leq k_1,\dots,k_{n_0+1}\leq d=n_0(c-1)+b+1$ then one of the following three cases holds:
\begin{enumerate}
\item $1\leq k_i\leq b$ for some $i$ with $1\leq i\leq n_0+1$;
\item $-c+1\leq k_i-k_j\leq c-2$ for some $(i,j)$ with $1\leq i<j\leq n_0+1$;
\item $k_i=(n_0+1-i)(c-1)+b+1$ for $i=1,\dots,n_0+1$.
\end{enumerate}
If (1) or (2) holds, then $Q\big(d\Mid u;k\big)=0$
by the same argument as that in the first part of the proof of Lemma~\ref{lem-Q}.
Hence, \eqref{extra-B-2} further reduces to
\begin{equation}\label{extra-B-4}
\CT_{x}Q(d)=\sum_{1\leq u_1<\cdots<u_{n_0+1}\leq n}
\CT_xQ\big(d\Mid u;k\big),
\end{equation}
where
\begin{equation}\label{k}
k=\big(n_0(c-1)+b+1,\dots,c+b,b+1\big).
\end{equation}

We proceed by showing that
if there exist an integer $i$ such that $1\leq i\leq n_0$ and $n_0+1\leq u_i<u_{i+1}\leq n$,
then $Q(d\Mid u;k)=0$ ($k$ is fixed by \eqref{k}). This can be easily seen since
$Q(d\Mid u;k)$ has the factor
\[
E_{u,k}\Big[\big(x_{u_i}/x_{u_{i+1}}\big)_c\big(qx_{u_{i+1}}/x_{u_i}\big)_c\Big],
\]
which is equal to
\begin{multline*}
E_{u,k}\Big[q^{\binom{c+1}{2}}(-x_{u_{i+1}}/x_{u_i})^c\big(q^{-c}x_{u_i}/x_{u_{i+1}}\big)_{2c}\Big]\\
=q^{\binom{c+1}{2}}(-q^{k_i-k_{i+1}})^c\big(q^{k_{i+1}-k_i-c}\big)_{2c}
=(-1)^{c}q^{\binom{c+1}{2}+c(c-1)}\big(q^{-2c+1}\big)_{2c}=0.
\end{multline*}
Therefore, we can fix $u_i=i$ for $i=1,\dots,n_0$. Then, we can write \eqref{extra-B-4} as
\begin{equation}\label{extra-B-3}
\CT_{x}Q(d)=\sum_{u_{n_0+1}=n_0+1}^n
\CT_xQ(d\Mid u;k),
\end{equation}
where $k$ is fixed by \eqref{k}, and
\begin{equation}\label{u}
u=(1,\dots,n_0,u_{n_0+1}).
\end{equation}

Let
\begin{align*}
Q'(d)&=x_0^{-l-|\mu|}P_{(l)}\Big(\Big[\sum_{i=1}^{n}\frac{1-q^{c-\chi(i\leq n_0)}}{1-q^c}x_i\Big];q,q^c\Big)\\
&\quad \times P_{\mu}\Big(\Big[\frac{q^{c-b-1}-q^{-d}}{1-q^c}x_0
+\sum_{i=1}^{n}\frac{1-q^{c-\chi(i\leq n_0)}}{1-q^c}x_i\Big];q,q^c\Big)
\prod_{j=1}^{n}\frac{\big(qx_{j}/x_{0}\big)_b}
{\big(q^{-d}x_{0}/x_{j}\big)_{d}}\\
&\quad \times
\prod_{\substack{1\leq i<j\leq n\\ 1\leq i\leq n_0}}
\big(x_{i}/x_{j}\big)_{c-1}\big(qx_j/x_i\big)_{c-1}
\prod_{n_0+1\leq i\neq j\leq n}
\big(x_{i}/x_{j}\big)_c,
\end{align*}
and define
\[
Q'\big(d\Mid u;k\big):=E_{u,k}\Big(Q'(d)\prod_{i=1}^{s}\big(1-\frac{x_{0}}{x_{u_{i}}q^{k_{i}}}\big)\Big).
\]

We can write \eqref{extra-B-3}
using $\CT\limits_x Q'\big(d\Mid u;k\big)$. We state this relationship in the next lemma.
\begin{lem}\label{lem-B-add1}
For $u_{n_0+1}$ an integer such that $n_0+1\leq u_{n_0+1}\leq n$
let $u=(1,\dots,n_0,u_{n_0+1})$,
and let $k=(k_1,\dots,k_{n_0+1})$ such that $1\leq k_i\leq d$ for $d$ a nonnegative integer. Then,
\begin{equation}\label{e-extra-B-equiv}
\sum_{u_{n_0+1}=n_0+1}^n\CT_x Q(d\Mid u;k)
=\frac{1}{(n-n_0-1)!}\prod_{i=1}^{n-n_0-1}
\frac{1-q^{(i+1)c}}{1-q^c}
\CT_x Q'\big(d\Mid (1,2,\dots,n_0+1);k\big).
\end{equation}
\end{lem}
\begin{proof}
We can write the left-hand side of \eqref{e-extra-B-equiv} as
\[
\CT_x \prod_{n_0+1\leq i<j\leq n}\big(x_{i}/x_{j}\big)_c\big(qx_{j}/x_{i}\big)_c\times R(x),
\]
where
\begin{align*}
R(x)&=\sum_{u_{n_0+1}=n_0+1}^nE_{u,k}\bigg( x_0^{-l-|\mu|}
P_{\mu}\Big(\Big[\frac{q^{c-b-1}-q^{-d}}{1-q^c}x_0
+\sum_{i=1}^{n}\frac{1-q^{c-\chi(i\leq n_0)}}{1-q^c}x_i\Big];q,q^c\Big)\\
&\quad \times P_{(l)}\Big(\Big[\sum_{i=1}^{n}\frac{1-q^{c-\chi(i\leq n_0)}}{1-q^c}x_i\Big];q,q^c\Big)
\prod_{j=1}^{n}\frac{\big(qx_{j}/x_{0}\big)_b}
{\big(q^{-d}x_{0}/x_{j}\big)_{d}}\prod_{i=1}^{n_0+1}\Big(1-\frac{x_0}{x_{u_i}q^{k_i}}\Big)\\
&\quad \times\prod_{\substack{1\leq i<j\leq n\\1\leq i\leq n_0}}
\big(x_{i}/x_{j}\big)_{c-1}\big(qx_j/x_i\big)_{c-1}\bigg).
\end{align*}
It is not hard to see that $R(x)$ is a rational function symmetric in $x_{n_0+1},\dots,x_n$.
By Proposition~\ref{prop-equiv}, we have
\[
\sum_{u_{n_0+1}=n_0+1}^n\CT_x Q(d\Mid u;k)
=\frac{1}{(n-n_0)!}\prod_{i=1}^{n-n_0-1}
\frac{1-q^{(i+1)c}}{1-q^c}\times
\sum_{u_{n_0+1}=n_0+1}^n\CT_x Q'(d\Mid u;k).
\]
Since $\CT\limits_x Q'(d\Mid u;k)$ is symmetric in $x_{n_0+1},\dots,x_n$, we obtain \eqref{e-extra-B-equiv}.
\end{proof}

Our next objective is to express $\CT\limits_x Q'\big(n_0(c-1)+b+1\Mid (1,2,\dots,n_0+1);k\big)$ using $A_n(a,b,c,\lambda,\mu)$. This is the content of the next lemma.
\begin{lem}\label{lem-B-add2}
Let $u=(1,2,\dots,n_0+1)$ and $k=\big(n_0(c-1)+b+1,\dots,c+b,b+1\big)$. Then, for $c>b$
\begin{equation}
\CT_x Q'\big(n_0(c-1)+b+1\Mid u;k\big)
=C\times A_{n-n_0-1}\big(c-b-1,(n_0+1)(c-1)+b+1,c,\mu,l\big),
\end{equation}
where
\[
C=(-1)^{(n_0+1)b}q^{-\binom{n_0+1}{2}b(c-1)-(n_0+1)\binom{b+1}{2}}
(n-n_0-1)!\frac{(q)_{(n_0+1)(c-1)}}{(q)_{c-1}^{n_0+1}}
\prod_{i=1}^{n-n_0-2}\frac{1-q^c}{1-q^{(i+1)c}}.
\]
\end{lem}
\begin{proof}
Denote $d=n_0(c-1)+b+1$.
We can write $Q'(d\Mid u;k)$ as
\begin{equation}\label{Q2'}
Q'\big(d\Mid u;k\big)=E_{u,k}\Big(Q'(d)\prod_{i=1}^{n_0+1}\big(1-\frac{x_{0}}{x_{u_{i}}q^{k_{i}}}\big)\Big)
=H\times M\times L,
\end{equation}
where
\begin{multline}\label{H}
H=(x_{n_0+1}q^{b+1})^{-(l+|\mu|)}\times
P_{(l)}\Big(\Big[\frac{1-q^{c-1}}{1-q^{c}}\sum_{i=1}^{n_0}x_{n_0+1}q^{-(n_0+1-i)(c-1)}
+\sum_{i=n_0+1}^{n}x_i\Big];q,q^c\Big)\\
\times P_{\mu}\Big(\Big[\frac{q^{c-b-1}-q^{-d}}{1-q^{c}}x_{n_0+1}q^{b+1}
+\frac{1-q^{c-1}}{1-q^c}\sum_{i=1}^{n_0}x_{n_0+1}q^{-(n_0+1-i)(c-1)}+\sum_{i=n_0+1}^nx_i\Big];q,q^c\Big),
\end{multline}
\begin{align}
M&=E_{u,k}\Big(\prod_{i=1}^{n_0+1}\frac{\big(qx_{i}/x_{0}\big)_{b}\big(1-q^{k_{i}}x_0/x_{u_i}\big)}
{\big(q^{-d}x_{0}/x_{i}\big)_{d}}
\prod_{1\leq i<j\leq n_0+1}
\big(x_{i}/x_{j}\big)_{c-1}\big(qx_{j}/x_{i}\big)_{c-1}\Big)\nonumber \\
&=\prod_{i=1}^{n_0+1}\frac{(q^{1-k_i})_{b}}{(q^{k_i-d})_{d-k_i}(q)_{k_i-1}}
\prod_{1\leq i<j\leq n_0+1}(q^{k_j-k_i})_{c-1}(q^{k_i-k_j+1})_{c-1}
\bigg|_{\substack{k_i=(n_0+1-i)(c-1)+b+1\\1\leq i\leq n_0+1}}\nonumber \\
&=\prod_{i=0}^{n_0}\frac{(q)_{(i+1)(c-1)}(q^{-i(c-1)-b})_{b}}{(q)_{b+i(c-1)}(q)_{c-1}}
=(-1)^{(n_0+1)b}q^{-\binom{n_0+1}{2}b(c-1)- (n_0+1)\binom{b+1}{2}}\frac{(q)_{(n_0+1)(c-1)}}{(q)_{c-1}^{n_0+1}},\label{EukM}
\end{align}
and
\begin{align*}
L&=E_{u,k}\Big(\prod_{j=n_0+2}^{n}\frac{\big(qx_{j}/x_{0}\big)_{b}}
{\big(q^{-d}x_{0}/x_{j}\big)_{d}}
\prod_{i=1}^{n_0}\prod_{j=n_0+2}^{n}
\big(x_{i}/x_{j}\big)_{c-1}\big(qx_{j}/x_{i}\big)_{c-1}\prod_{n_0+1\leq i\neq j\leq n}^{n}
\big(x_{i}/x_{j}\big)_{c}\Big)\\
&=\prod_{i=1}^{n_0}\prod_{j=n_0+2}^{n}
\big(q^{-(n_0-i+1)(c-1)}x_{n_0+1}/x_{j}\big)_{c-1}\big(q^{(n_0-i+1)(c-1)+1}x_{j}/x_{n_0+1}\big)_{c-1}\\
&\quad \times \prod_{j=n_0+2}^{n}\frac{\big(q^{-b}x_{j}/x_{n_0+1}\big)_{b}\big(x_{n_0+1}/x_{j}\big)_{c}\big(x_{j}/x_{n_0+1}\big)_{c}}
{\big(q^{-n_0(c-1)}x_{n_0+1}/x_{j}\big)_{n_0(c-1)+b+1}}
\prod_{n_0+2\leq i\neq j\leq n}\big(x_{i}/x_{j}\big)_{c}.
\end{align*}
For $c>b$, we find that $L$ can be written as
\begin{equation}\label{EukL}
L=\prod_{j=n_0+2}^{n}\big(q^{-b}x_{j}/x_{n_0+1}\big)_{(n_0+1)(c-1)+b+1}
\big(q^{b+1}x_{n_0+1}/x_i\big)_{c-b-1}
\prod_{n_0+2\leq i\neq j\leq n}\big(x_{i}/x_{j}\big)_{c}.
\end{equation}
By a direct calculation or using Proposition~\ref{lem-subs} with $s=n_0+1, r=n_0, \mathfrak{t}=0$, we have
\begin{subequations}\label{for-H}
\begin{multline}
\frac{q^{c-b-1}-q^{-d}}{1-q^{c}}x_{n_0+1}q^{b+1}
+\frac{1-q^{c-1}}{1-q^{c}}\sum_{i=1}^{n_0}x_{n_0+1}q^{-(n_0+1-i)(c-1)}+x_{n_0+1}\\
=\frac{1-q}{1-q^{c}}x_{n_0+1}\bigg(
\frac{q^{c-b-1}-q^{-d}}{1-q}q^{b+1}
+\frac{1-q^{c-1}}{1-q}\sum_{i=1}^{n_0}q^{-(n_0+1-i)(c-1)}+\frac{1-q^{c}}{1-q}
\bigg)=0.
\end{multline}
It is straightforward to get
\begin{equation}
\frac{1-q^{c-1}}{1-q^{c}}\sum_{i=1}^{n_0}x_{n_0+1}q^{-(n_0+1-i)(c-1)}
+x_{n_0+1}=\frac{q^{-n_0(c-1)}-q^{c}}{1-q^{c}}x_{n_0+1}.
\end{equation}
\end{subequations}
Substituting \eqref{for-H} into \eqref{H} gives
\begin{multline}\label{EukH}
H=(x_{n_0+1}q^{b+1})^{-(l+|\mu|)}
\times P_{\mu}\Big(\Big[\sum_{i=n_0+2}^nx_i\Big];q,q^{c}\Big)\\
\times P_{(l)}\Big(\Big[\frac{q^{-n_0(c-1)}-q^{c}}{1-q^{c}}x_{n_0+1}+\sum_{i=n_0+2}^{n}x_i\Big];q,q^{c}\Big).
\end{multline}
Substituting \eqref{EukL} and \eqref{EukH} into \eqref{Q2'}, we can write $Q'\big(d\Mid u;k\big)$ as
\begin{multline*}
M\times (x_{n_0+1}q^{b+1})^{-(l+|\mu|)}
P_{(l)}\Big(\Big[\frac{q^{-n_0(c-1)}-q^{c}}{1-q^{c}}x_{n_0+1}+\sum_{i=n_0+2}^{n}x_i\Big];q,q^{c}\Big)
P_{\mu}\Big(\Big[\sum_{i=n_0+2}^nx_i\Big];q,q^{c}\Big)\\
\times
\prod_{j=n_0+2}^{n}\big(q^{-b}x_{j}/x_{n_0+1}\big)_{(n_0+1)(c-1)+b+1}
\big(q^{b+1}x_{n_0+1}/x_i\big)_{c-b-1}
\prod_{n_0+2\leq i\neq j\leq n}\big(x_{i}/x_{j}\big)_{c}.
\end{multline*}
The constant term of $Q'\big(d\Mid u;k\big)$ remains the same by taking $x_{n_0+1}q^{b+1}\mapsto x_0$ and $x_{n_0+i+1}\mapsto x_{i}$ for $i=1,2,\dots,n-n_0-1$.
This yields
\begin{multline}\label{Q2'-2}
\CT_x Q'\big(d\Mid u;k\big)
=M\times \CT_{x'} x_{0}^{-(l+|\mu|)}P_{(l)}\Big(\Big[\frac{q^{-d}-q^{c-b-1}}{1-q^{c}}x_{0}
+\sum_{i=1}^{n-n_0-1}x_i\Big];q,q^{c}\Big)\\
\times P_{\mu}\Big(\Big[\sum_{i=1}^{n-n_0-1}x_i\Big];q,q^{c}\Big)
\prod_{j=1}^{n-n_0-1}\big(x_{0}/x_j\big)_{c-b-1}\big(qx_{j}/x_{0}\big)_{(n_0+1)(c-1)+b+1}
\prod_{1\leq i\neq j\leq n-n_0-1}\big(x_{i}/x_{j}\big)_{c},
\end{multline}
where $x'=(x_0,\dots,x_{n-n_0-1})$.
By Proposition~\ref{prop-equiv}, the right-hand side of \eqref{Q2'-2} equals
\begin{align*}
&M\times (n-n_0-1)!\prod_{i=1}^{n-n_0-2}\frac{1-q^{c}}{1-q^{(i+1)c}}\times
\CT_{x'} x_{0}^{-(l+|\mu|)} P_{\mu}\Big(\Big[\sum_{i=1}^{n-n_0-1}x_i\Big];q,q^{c}\Big)\\
&\times P_{(l)}\Big(\Big[\frac{q^{-d}-q^{c-b-1}}{1-q^{c}}x_{0}
+\sum_{i=1}^{n-n_0-1}x_i\Big];q,q^{c}\Big)
\prod_{j=1}^{n-n_0-1}\big(x_{0}/x_j\big)_{c-b-1}\big(qx_{j}/x_{0}\big)_{(n_0+1)(c-1)+b+1}\\
&\prod_{1\leq i<j\leq n-n_0-1}\big(x_{i}/x_{j}\big)_{c}\big(qx_{j}/x_{i}\big)_{c}
=C\times A_{n-n_0-1}\big(c-b-1,(n_0+1)(c-1)+b+1,c,\mu,l\big),
\end{align*}
where
\[
C=M\times (n-n_0-1)!\prod_{i=1}^{n-n_0-2}\frac{1-q^{c}}{1-q^{(i+1)c}}.
\]
Substituting the expression \eqref{EukM} for $M$ into above, we obtain the formula for $C$ in the lemma.
\end{proof}

By \eqref{extra-B-3}, Lemma~\ref{lem-B-add1}, Lemma~\ref{lem-B-add2} and the fact that
\[
B_{n,n_0}(-n_0(c-1)-b-1,b,c,l,\mu)=\CT_{x}Q(n_0(c-1)+b+1),
\]
it is straightforward to obtain the next result.
\begin{lem}\label{lem-extra-B}
For $\ell(\mu)<n-n_0$ and $c>b$,
\begin{multline}\label{e-extr2}
B_{n,n_0}\big(-n_0(c-1)-b-1,b,c,l,\mu\big)
=(-1)^{(n_0+1)b}q^{-\binom{n_0+1}{2}b(c-1)-(n_0+1)\binom{b+1}{2}}\\
\times \frac{1-q^{(n-n_0)c}}{1-q^{c}}\cdot\frac{(q)_{(n_0+1)(c-1)}}{(q)_{c-1}^{n_0+1}} A_{n-n_0-1}(c-b-1,(n_0+1)(c-1)+b+1,c,\mu,l).
\end{multline}
\end{lem}

Now we have obtained all the ingredient to determine $B_{n,n_0}\big(a,b,c,l,\mu\big)$.
Denote the right-hand side of \eqref{main-B} by $B'_{n,n_0}\big(a,b,c,l,\mu\big)$.
By Corollary~\ref{cor-rationality}, to prove Theorem~\ref{thm-qF2} it suffices to show that
\[
B_{n,n_0}\big(a,b,c,l,\mu\big)=B'_{n,n_0}\big(a,b,c,l,\mu\big)
\]
for $c>b+\mu_1$.
To prove this, we need to verify that for $c>b+\mu_1$ the constant term $B'_{n,n_0}\big(a,b,c,l,\mu\big)$ satisfies
\begin{enumerate}
\item $B'_{n,n_0}\big(a,b,c,l,\mu\big)$ is a polynomial in $q^a$ of degree $nb+l+|\mu|$.

\item $B'_{n,n_0}\big(a,b,c,l,\mu\big)=0$ for $a\in B_1\cup B_2\cup B_3$.

\item $B'_{n,n_0}\big(-n_0(c-1)-b-1,b,c,l,\mu\big)=B_{n,n_0}\big(-n_0(c-1)-b-1,b,c,l,\mu\big)$.
\end{enumerate}
The above verifications are routine, we omit the details. Therefore, we complete the proof of Theorem~\ref{thm-qF2}.

\section{Proof of Theorem~\ref{thm-qForrester}}\label{sec-proof2}

In this section, we give a proof of Theorem~\ref{thm-qForrester} along the outline in the introduction.
By Corollary~\ref{cor-poly-BC}, $C_{n,n_0}(a,b,c,l,m)$ is a polynomial in $q^a$ of degree at most $nb+l+m$.
By Corollary~\ref{cor-rationality}, if we can obtain an expression for $C_{n,n_0}(a,b,c,l,m)$
for sufficiently many integers $c$ (here $c>b+1$), we can extend the result to all positive integers $c$.
Recall the definition of the sets $C_1,C_2,C_3$ in \eqref{Roots-B} and \eqref{Roots-C}.
By Lemma~\ref{lem-Roots-2}, we have $C_{n,n_0}(a,b,c,l,m)=0$ for $a\in C_2$.
By Proposition~\ref{prop-roots-1} with
\begin{equation}\label{M-poly}
M_{\mathfrak{r}}=h_l\Big[\sum_{i=1}^{n}\frac{1-q^{c-\chi(i\leq n_0)}}{1-q}x_i\Big]
\prod_{i=n-m+1}^n(x_0-q^{b+1}x_i),
\end{equation}
we obtain $C_{n,n_0}(a,b,c,l,m)=\CT\limits_{x}Q(-a)=0$ for $a\in C_1\cup C_3$.
Hence, we have determined all the roots for $C_{n,n_0}(a,b,c,l,m)$ since all the elements of $C_1\cup C_2\cup C_3$ are distinct for $c>b+1$.
Then, to give a closed-form expression for $C_{n,n_0}(a,b,c,l,m)$, the last step is to obtain
an explicit (non-zero) expression for $C_{n,n_0}(a,b,c,l,m)$ for $a$ at an additional point.
This is the content of the next proposition.

\begin{prop}\label{extra-C}
Let $C_{n,n_0}(a,b,c,l,m)$ be defined in \eqref{eq-qForrester}.
Then for $n-n_0\leq m<n$ and $c>b+1$,
\begin{align}\label{e-extra-C}
&C_{n,n_0}(a,b,c,l,m)|_{-a=(n-m-1)(c-1)+b+1}\\
& =
(-1)^lh_l\Big[\frac{1-q^{nc-n_0}}{1-q}\Big]
\frac{(q^{-(n-m-1)(c-1)-b-l})_l}{(q^{(n-1)c+b-n_0+2})_l}
\prod_{i=0}^{n-m-1}\frac{(q)_{(i+1)(c-1)}(q^{-i(c-1)-b})_{b}}{(q)_{b+i(c-1)}(q)_{c-1}}
\prod_{j=1}^{n-n_0-1}(1-q^{(j+1)c})\nonumber \\
& \times q^{\binom{l}{2}}\prod_{j=n-m}^{n-1}\frac{(q^{(j-n+m+1)(c-1)-b+\chi(j>n_0)(j-n_0)})_{(n-m)(c-1)+b+1}
(q)_{(j-n+m+1)(c-1)+\chi(j>n_0)(j-n_0)}}
{(q)_{j(c-1)+b+1+\chi(j>n_0)(j-n_0)}(q)_{c-\chi(j\leq n_0)}}.\nonumber
\end{align}
\end{prop}

Recall the definition of $Q(d)=Q_{n,n_0}(d,b,c;M_{\mathfrak{r}})$ in \eqref{def-Q3}.
For $d$ a nonnegative integer
\begin{equation}\label{eq-Quk}
C_{n,n_0}(-d,b,c,l,m)=\CT_x Q_{n,n_0}(d,b,c;M_{\mathfrak{r}})=\CT_x Q(d)
\end{equation}
by taking $M_{\mathfrak{r}}$ as in \eqref{M-poly}
and $\mathfrak{r}=l+m$. In this section, we always assume that the $M_{\mathfrak{r}}$ in $Q(d)$ and $Q(d\Mid u;k)$ is the polynomial in \eqref{M-poly}.
Thus, to prove Proposition~\ref{extra-C}, it suffices to compute
$\CT\limits_x Q\big((n-m-1)(c-1)+b+1\big)$.
Since $(n-m-1)(c-1)+b+1$ is a positive integer by the assumption $m<n$, we can apply the Gessel--Xin operation
to $\CT\limits_x Q\big((n-m-1)(c-1)+b+1\big)$ and by \eqref{Q3-sum}, we have
\begin{multline}\label{e-C-add}
\CT_x Q\big((n-m-1)(c-1)+b+1\big)\\
=\sum_{s\in T\subseteq \{1,\dots,n\}}\sum_{\substack{1\leq u_1<\cdots<u_s\leq n\\1\leq k_1,\dots,k_s\leq d}}
\CT_xQ\big((n-m-1)c+b+1\Mid u_1,\dots,u_s;k_1,\dots,k_s\big).
\end{multline}
We find that only one term in the sum of \eqref{e-C-add} does not vanish.
We present this result in the next lemma.
\begin{lem}\label{lem-extra-C1}
Let $Q(d)$ be defined as in \eqref{eq-Quk}.
For $n-n_0\leq m<n$,
\begin{multline}\label{e-Q3}
\CT_x Q\big((n-m-1)(c-1)+b+1\big)\\
=\CT_xQ\big((n-m-1)(c-1)+b+1\Mid 1,\dots,n-m;(n-m-1)(c-1)+b+1,
\dots,c+b,b+1\big).
\end{multline}
\end{lem}
\begin{proof}
Denote by $e:=(n-m-1)(c-1)+b+1$ for short.
We will show that
only the term $\CT\limits_xQ(e\Mid 1,\dots,n-m;e,\dots,c+b,b+1)$ in the sum of \eqref{e-C-add}
does not vanish.
We first show that the $s$ in \eqref{e-C-add} can only be $n-m$.

If $s>n-m$ then $e<(s-1)(c-1)+b+1$. Then by the property (1) of Lemma~\ref{lem-Q} we have
$Q(e\Mid u;k)=0$, where $u=(u_1,\dots,u_s)$ and $k=(k_1,\dots,k_s)$ are the integer sequences in \eqref{e-C-add}.

If $1\leq s<n-m$ then $e>s(c-1)$.
It follows that only the cases (2) and (3) of Lemma~\ref{lem-Q} can occur.
If (3) of Lemma~\ref{lem-Q} applies, or (2) of Lemma~\ref{lem-Q} holds and $u_s=n$, then $\CT\limits_xQ(e\Mid u;k)=0$.
If (2) of Lemma~\ref{lem-Q} applies and $u_s\neq n$, then
\[
\CT_{x_{u_s}}Q(e\Mid u;k)=\sum_{\substack{u_s<u_{s+1}\leq n\\1\leq k_{s+1}\leq e}}
Q(e\Mid u_1,\dots,u_s,u_{s+1};k_1,\dots,k_s,k_{s+1}).
\]
But this contradicts the Gessel--Xin operation.

By the above argument we can conclude that the $s$ in \eqref{e-C-add} can only be $n-m$, otherwise the summands
in the right-hand side of \eqref{e-C-add} vanish.
Then \eqref{e-C-add} reduces to
\begin{equation}\label{Q3-4}
\CT_x Q(e)
=\sum_{\substack{1\leq u_1<\cdots<u_{n-m}\leq n\\1\leq k_i\leq e}}
\CT_xQ(e\Mid u';k'),
\end{equation}
where $u'=(u_1,\dots,u_{n-m})$ and $k'=(k_1,\dots,k_{n-m})$.

We complete the proof by showing that all the terms in the sum of \eqref{Q3-4}
vanish except only when $k_i=(n-m-i)(c-1)+b+1$ and $u_i=i$ for $i=1,\dots,n-m$.

Since $1\leq k_i\leq e$ for $i=1,\dots,n-m$,
by Corollary~\ref{cor-key} with $s=n-m$,
at least one of the following cases holds:
\begin{enumerate}
\item $1\leq k_i\leq b$ for some $i$ with $1\leq i\leq n-m$;
\item $-c+1\leq k_i-k_j\leq c-2$ for some $(i,j)$ with $1\leq i<j\leq n-m$;
\item $k_i=(n-m-i)(c-1)+b+1$ for $i=1,\dots,n-m$.
\end{enumerate}
We then carry out similar argument as that in the proof of Lemma~\ref{lem-Q}.

For $k = (k_1, k_2, \dots, k_{s})$, $u = (u_1,
u_2, \dots, u_{s})$ and a rational function $F$,
recall that $E_{u,k}F$ is the result of
replacing $x_{u_i}$ in $F$ with $x_{u_{s}}q^{k_{s}-k_i}$ for
$i = 0,1,\dots, s-1$. Set $u_0=k_0:=0$.
If (1) holds, then $Q(e\Mid u';k')$ has the factor
\[
E_{u',k'}\Big[\big(qx_{u_i}/x_0\big)_{b}\Big]=\big(q^{1-k_i}\big)_{b}=0.
\]
If (2) holds, then $Q(e\Mid u';k')$ has the factor
\[
E_{u',k'}\Big[\big(x_{u_i}/x_{u_j}\big)_{c-1}\big(qx_{u_j}/x_{u_i}\big)_{c-1}\Big],
\]
which is equal to
\[
E_{u',k'}\Big[q^{\binom{c}{2}}(-x_{u_j}/x_{u_i})^{c-1}\big(q^{-c+1}x_{u_i}/x_{u_j}\big)_{2c-2}\Big]
=q^{\binom{c}{2}}(-q^{k_i-k_j})^{c-1}\big(q^{k_j-k_i-c+1}\big)_{2c-2}=0.
\]
If (3) holds and $u_{n-m}>n-m$, then $Q(e\Mid u';k')$ has the factor
\[
E_{u',k'}\Big[\big(qx_{u_{n-m}}/x_0\big)_{b+1}\Big]
=\big(q^{1-k_{n-m}}\big)_{b+1}=\big(q^{-b}\big)_{b+1}=0.
\]

In conclusion, $Q(e\Mid u';k')$ does not vanish only when
$k_i=(n-m-i)(c-1)+b+1$ and $u_i=i$ for $i=1,\dots,n-m$.
Together with \eqref{Q3-4} gives \eqref{e-Q3}.
\end{proof}

We find that the right-hand side of \eqref{e-Q3}
can be expressed by $B_{n,n_0}(a,b,c,l,\mu)$.
We state this result in the next lemma.
\begin{lem}\label{lem-extra-C2}
For $n-n_0\leq m<n$ and $c>b+1$,
\begin{multline}\label{Q-e}
\CT_xQ\big((n-m-1)(c-1)+b+1\Mid 1,\dots,n-m;(n-m-1)(c-1)+b+1,
\dots,c+b,b+1\big)\\
=\frac{(q^c)_l}{(q)_l}\prod_{i=0}^{n-m-1}\frac{(q)_{(i+1)(c-1)}(q^{-i(c-1)-b})_{b}}{(q)_{b+i(c-1)}(q)_{c-1}}
\times B_{m,n_0-n+m}\big(c-b-2,(n-m)(c-1)+b+1,c,0,l\big),
\end{multline}
where $B_{n,n_0}(a,b,c,l,\mu)$ is defined in \eqref{eq-qF2}.
\end{lem}
\begin{proof}
By the definition of $Q(d\Mid u,k)$ in \eqref{eq-Qrk} with $M_{\mathfrak{r}}$ defined in \eqref{M-poly}, we can write the left-hand side of \eqref{Q-e} as
\begin{equation}\label{HCL}
\CT_xG\times C\times L,
\end{equation}
where
\begin{align*}
G&=x_{n-m}^{-l}q^{-(b+1)l}h_l\Big[\frac{x_{n-m}(1-q^{c-1})}{1-q}
\sum_{i=1}^{n-m}q^{-(n-m-i)(c-1)}
+\sum_{i=n-m+1}^n\frac{1-q^{c-\chi(i\leq n_0)}}{1-q}x_i\Big]\\
&=x_{n-m}^{-l}q^{-(b+1)l}h_l\Big[\frac{x_{n-m}(1-q^{(n-m)(c-1)})}{1-q}q^{-(n-m-1)(c-1)}
+\sum_{i=n-m+1}^n\frac{1-q^{c-\chi(i\leq n_0)}}{1-q}x_i\Big],
\end{align*}
\begin{equation}
C=\prod_{i=1}^{n-m}\frac{(q^{-(n-m-i)(c-1)-b})_b}{(q^{-(i-1)(c-1)})_{(i-1)(c-1)}(q)_{(n-m-i)(c-1)+b}}
\prod_{1\leq i<j\leq n-m}(q^{(i-j)(c-1)})_{c-1}(q^{(j-i)(c-1)+1})_{c-1},
\end{equation}
and
\begin{align*}
L&=\prod_{j=n-m+1}^n \frac{(q^{-b}x_j/x_{n-m})_{b+1}}{(q^{-(n-m-1)(c-1)}x_{n-m}/x_j)_{(n-m-1)(c-1)+b+1}}\\
&\quad \times\prod_{n-m+1\leq i<j\leq n}(x_i/x_j)_{c-\chi(i\leq n_0)}(qx_j/x_i)_{c-\chi(i\leq n_0)}\\
&\quad \times \prod_{i=1}^{n-m}\prod_{j=n-m+1}^n(q^{-(n-m-i)(c-1)}x_{n-m}/x_j)_{c-1}
(q^{(n-m-i)(c-1)+1}x_j/x_{n-m})_{c-1}.
\end{align*}
It is routine to check that
\begin{equation}\label{C}
C=\prod_{i=0}^{n-m-1}\frac{(q)_{(i+1)(c-1)}(q^{-i(c-1)-b})_{b}}{(q)_{b+i(c-1)}(q)_{c-1}},
\end{equation}
and
\begin{align*}
L&=\prod_{j=n-m+1}^n \frac{(q^{-b}x_j/x_{n-m})_{b+1}}{(q^{-(n-m-1)(c-1)}x_{n-m}/x_j)_{(n-m-1)(c-1)+b+1}}\\
&\quad \times \prod_{n-m+1\leq i<j\leq n}(x_i/x_j)_{c-\chi(i\leq n_0)}(qx_j/x_i)_{c-\chi(i\leq n_0)}\\
&\quad \times \prod_{j=n-m+1}^n(q^{-(n-m-1)(c-1)}x_{n-m}/x_j)_{(n-m)(c-1)}(qx_j/x_{n-m})_{(n-m)(c-1)}.
\end{align*}
For $c>b+1$, we can simplify $L$ as
\begin{multline*}
L=\prod_{j=n-m+1}^n (q^{b+1}x_{n-m}/x_j)_{c-b-2}(q^{-b}x_j/x_{n-m})_{(n-m)(c-1)+b+1}\\
\times \prod_{n-m+1\leq i<j\leq n}(x_i/x_j)_{c-\chi(i\leq n_0)}(qx_j/x_i)_{c-\chi(i\leq n_0)}.
\end{multline*}
Letting $x_{n-m}q^{b+1}\mapsto x_0$, $n_0-n+m\mapsto n_0'$, $c-b-2\mapsto a'$,
$(n-m)(c-1)+b+1\mapsto b'$ and
$(x_{n-m+1},x_{n-m+2},\dots,x_n)\mapsto (x_1,x_2,\dots,x_m)$, we find that
\begin{multline}\label{HL}
\CT_x G\times L
=\CT_{x_0,\dots,x_m}x_0^{-l}h_l\Big[\frac{q^{c-b'-1}-q^{a'}}{1-q}x_0
+\sum_{i=1}^m\frac{1-q^{c-\chi(i\leq n_0')}}{1-q}x_i\Big]
\\
\times \prod_{i=1}^m (x_0/x_i)_{a'}(qx_i/x_0)_{b'}
\prod_{1\leq i<j\leq m}(x_i/x_j)_{c-\chi(i\leq n_0')}(qx_j/x_i)_{c-\chi(i\leq n_0')}.
\end{multline}
By \eqref{modi-complete} and \eqref{relation-P-g} we can write \eqref{HL} as
\begin{multline*}
\CT_x G\times L
=\frac{(q^c)_l}{(q)_l}\CT_{x_0,\dots,x_m}x_0^{-l}P_{(l)}\Big(\Big[\frac{q^{c-b'-1}-q^{a'}}{1-q^c}x_0
+\sum_{i=1}^m\frac{1-q^{c-\chi(i\leq n_0')}}{1-q^c}x_i\Big];q,q^c\Big)
\\
\times \prod_{i=1}^m (x_0/x_i)_{a'}(qx_i/x_0)_{b'}
\prod_{1\leq i<j\leq m}(x_i/x_j)_{c-\chi(i\leq n_0')}(qx_j/x_i)_{c-\chi(i\leq n_0')}.
\end{multline*}
By the definition of $B_{n,n_0}(a,b,c,l,\mu)$ in \eqref{eq-qF2},
the constant term $\CT\limits_x G\times L$ is in fact
\[
\frac{(q^c)_l}{(q)_l}B_{m,n_0'}(a',b',c,0,l)
=\frac{(q^c)_l}{(q)_l}B_{m,n_0-n+m}\big(c-b-2,(n-m)(c-1)+b+1,c,0,l\big).
\]
Together with \eqref{C} gives
\begin{multline*}
\CT_x C\times G\times L
=\frac{(q^c)_l}{(q)_l}\prod_{i=0}^{n-m-1}\frac{(q)_{(i+1)(c-1)}(q^{-i(c-1)-b})_{b}}{(q)_{b+i(c-1)}(q)_{c-1}}\\
\times B_{m,n_0-n+m}\big(c-b-2,(n-m)(c-1)+b+1,c,0,l\big).
\end{multline*}
Then \eqref{Q-e} follows.
\end{proof}
Now by Lemma~\ref{lem-extra-C1} and Lemma~\ref{lem-extra-C2}, we can give a proof of Proposition~\ref{extra-C}.
\begin{proof}[Proof of Proposition~\ref{extra-C}]
By \eqref{eq-Quk}, \eqref{e-Q3} and \eqref{Q-e}, we have
\begin{multline}\label{sub-B}
C_{n,n_0}(a,b,c,l,m)|_{-a=(n-m-1)c+b+1}\\
=\frac{(q^c)_l}{(q)_l}\prod_{i=0}^{n-m-1}\frac{(q)_{(i+1)(c-1)}(q^{-i(c-1)-b})_{b}}{(q)_{b+i(c-1)}(q)_{c-1}}
\times B_{m,n_0-n+m}\big(c-b-2,(n-m)(c-1)+b+1,c,0,l\big).
\end{multline}
By the expression \eqref{main-B} for $B_{n,n_0}(a,b,c,l,\mu)$ with
\[
(n,n_0,a,b,l,\mu)\mapsto (m,n_0-n+m,c-b-2,(n-m)(c-1)+b+1,0,l),
\]
it is straightforward to get
\begin{align*}\label{e-extra-c2-1}
&B_{m,n_0-n+m}\big(c-b-2,(n-m)(c-1)+b+1,c,0,l\big)\\
&\quad=(-1)^lq^{\binom{l}{2}}\frac{(q)_l}{(q^c)_l}h_l\Big[\frac{1-q^{nc-n_0}}{1-q}\Big]
\frac{(q^{-(n-m-1)(c-1)-b-l})_l}{(q^{(n-1)c+b-n_0+2})_l}\prod_{j=1}^{n-n_0-1}(1-q^{(j+1)c})\\
&\quad \prod_{j=n-m}^{n-1}\frac{(q^{(j-n+m+1)(c-1)-b+\chi(j>n_0)(j-n_0)})_{(n-m)(c-1)+b+1}
(q)_{(j-n+m+1)(c-1)+\chi(j>n_0)(j-n_0)}}
{(q)_{j(c-1)+b+1+\chi(j>n_0)(j-n_0)}(q)_{c-\chi(j\leq n_0)}}.
\end{align*}
Then we obtain the expression for $C_{n,n_0}(a,b,c,l,m)|_{-a=(n-m-1)c+b+1}$ in \eqref{e-extra-C}.
\end{proof}

\section{Case (3) of Lemma~\ref{lem-Q}}\label{sec-case3}

In this section, we give a proof of Case (3) of Lemma~\ref{lem-Q}.
For this case, we have to deal with the constant term of a certain rational function, where the degree of an indeterminate (say $x$) in the numerator is larger than the degree of $x$ in the denominator.
The method used in this section extends the Laurent series proof of constant term identities, such as the Gessel--Xin proof of the $q$-Dyson constant term identity in \cite{GX}, and the Laurent series proof of the first-layer coefficients of the $q$-Dyson product by Lv et al. in \cite{LXZ}.

Suppose $Q_{n,n_0}(d,b,c;M_{\mathfrak{r}})=Q_{n,0}(d,b,c-1;M_{\mathfrak{r}})$ for $n_0\geq n$.
(When we specify $M_{\mathfrak{r}}$ to be \eqref{Mr-for-B2} and \eqref{M-poly} to prove Theorem~\ref{thm-qF2} and Theorem~\ref{thm-qForrester} respectively, it is the case.)
Moreover, since the Case (3) of Lemma~\ref{lem-Q} does not occur for $n_0=0,1$,
we assume $2\leq n_0\leq n-1$ in this section.

\subsection{The Laurent polynomiality of $Q(d\Mid u;k)$}

In this subsection, we show that the rational function $Q(d\Mid u;k)$ (defined in \eqref{eq-Qrk}) is either 0, or can be written as a Laurent polynomial. This is the content of the next lemma.
\begin{lem}\label{lem-QLaurentpoly}
Let $(u_1,u_2,\dots,u_s)$ be an integer sequence such that $1\leq u_1<u_2<\cdots<u_s\leq n$ and set $U=\{u_1,u_2,\dots,u_s\}$. Denote $r=\big|\{u_i\Mid u_i\leq n_0, i=1,\dots,s\}\big|$.
For $0\leq r\leq n_0-2$ and $1\leq d\leq s(c-1)+b+\chi(s>n_0)(s-n_0)$,
$Q(d\Mid u;k)$ is either 0, or can be written as a Laurent polynomial of the form
\begin{equation}\label{Q-Laurent}
 x_{u_s}^l\prod_{\substack{i=1\\ i\notin U}}^n
\Big(p_i(x_i/x_{u_s})x_i^{d-s(c-1)-\chi(i>n_0)(s-r)}\Big)
\prod_{\substack{1\leq i<j\leq n\\i,j\notin U}}\big(x_i/x_j\big)_{c-\chi(i\leq n_0)}
\big(qx_j/x_i\big)_{c-\chi(i\leq n_0)},
\end{equation}
where $p_i(z)$ is a polynomial in $z$ and
\[
l=(n-s)\big(s(c-1)-d\big)+(s-r)(n-n_0-s+r).
\]
\end{lem}

Recall that $Q(d\Mid u;k)$ can be written as \eqref{defi-Q}.
We can further write it as
\begin{equation}\label{Q-further}
Q(d\Mid u;k)=
\prod_{\substack{1\leq i<j\leq n\\i,j\notin U}}\big(x_i/x_j\big)_{c-\chi(i\leq n_0)}
\big(qx_j/x_i\big)_{c-\chi(i\leq n_0)}\times H\times C\times V' \times L',
\end{equation}
where
\[
H=E_{u,k}\big(x_{0}^{-\mathfrak{r}}\times M_{\mathfrak{r}}\big),
\]
\begin{align*}
C=&\prod_{i=1}^s(q^{k_i-d})^{-1}_{d-k_i}(q)^{-1}_{k_i-1}
\prod_{\substack{1\leq i<j\leq s\\i\leq r}}(-1)^{c-1}q^{\binom{c}{2}+(c-1)(k_i-k_j)}
\prod_{r+1\leq i<j\leq s}(-1)^{c}q^{\binom{c+1}{2}+c(k_i-k_j)}\\
&\  \times \prod_{\substack{i=1\\ i\notin U}}^n (-1)^dq^{\binom{d+1}{2}-dk_s}
\prod_{\substack{i=1\\ i\notin U}}^n
\prod_{j=1}^r (-1)^{c-1}q^{(c-1)\big(k_s-k_j+\chi(u_j>i)\big)+\binom{c-1}{2}} \\
&\  \times
\prod_{\substack{i=1\\ i\notin U}}^n\prod_{j=r+1}^s
(-1)^{c-\chi(i\leq n_0)}q^{(c-\chi(i\leq n_0))\big(k_s-k_j+\chi(u_j>i)\big)+\binom{c-\chi(i\leq n_0)}{2}},
\end{align*}

\begin{equation}\label{V}
V'=\prod_{i=1}^s(q^{1-k_i})_b\prod_{\substack{1\leq i<j\leq s\\i\leq r}}(q^{k_j-k_i-c+1})_{2c-2}
\prod_{r+1\leq i<j\leq s}(q^{k_j-k_i-c})_{2c},
\end{equation}
and
\begin{align}\label{L}
L'&=\prod_{\substack{i=1\\ i\notin U}}^n
\frac{(q^{1-k_s}x_i/x_{u_s})_b}{(x_{u_s}/x_i)^d\big(q^{1-k_s}x_i/x_{u_s}\big)_d}
\prod_{\substack{i=1\\ i\notin U}}^n\prod_{j=1}^r
(x_{u_s}/x_i)^{c-1}
\big(q^{k_j-k_s-c+2-\chi(u_j>i)}x_i/x_{u_s}\big)_{2c-2}\\
&\ \ \times \prod_{\substack{i=1\\ i\notin U}}^n\prod_{j=r+1}^s
(x_{u_s}/x_i)^{c-\chi(i\leq n_0)}
\big(q^{k_j-k_s-c+1+\chi(i\leq n_0)-\chi(u_j>i)}x_i/x_{u_s}\big)_{2c-2\chi(i\leq n_0)}.
\nonumber
\end{align}

For certain conditions, $L'$ can be written as a Laurent polynomial. That is, the product $\prod_{\substack{i=1\\ i\notin U}}^n(q^{1-k_s}x_i/x_{u_s})_d$ in the denominator of $L'$ is in fact a factor of  the numerator of $L'$.
\begin{lem}\label{lem-Laurent}
For $s\in \{1,2,\dots,n\}$, let $0\leq r\leq \min\{n_0-2,s\}$, $1\leq d\leq s(c-1)+b+\chi(s>n_0)(s-n_0)$.
For $i=1,\dots,s$ the $k_i$ satisfy $1\leq k_i\leq d$, and
there exists a permutation $w\in\mathfrak{S}_s$ and nonnegative integers $t_1,\dots,t_s$
such that
\begin{equation}\label{kw1}
k_{w(1)}=b+t_1,
\end{equation}
and
\begin{equation}\label{kwj}
k_{w(j)}-k_{w(j-1)}=c-1+\chi(w(j)>r)\chi(w(j-1)>r)+t_j \quad \text{for $2\leq j\leq s$.}
\end{equation}
Here the $t_j$ satisfy
\begin{equation}\label{sumbound}
s-r\leq \sum_{j=1}^{s}\big(t_j+\chi(w(j)>r)\chi(w(j-1)>r)\big)\leq c-1+\chi(s>n_0)(s-n_0),
\end{equation}
$w(0):=0$, and $t_j>0$ if $w(j-1)<w(j)$ for $1\leq j\leq s$.
Then the rational function $L'$ defined in \eqref{L} can be written as the form
\begin{equation}\label{L-Laurent}
x_{u_s}^l\prod_{\substack{i=1\\ i\notin U}}^n
\Big(p_i(x_i/x_{u_s})x_i^{d-s(c-1)-\chi(i>n_0)(s-r)}\Big),
\end{equation}
where
\[
l=(n-s)\big(s(c-1)-d\big)+(s-r)(n-n_0-s+r),
\]
and the $p_i(z)$ are polynomials in $z$.
\end{lem}
\begin{proof}
Let
\begin{align*}
S_0:=&\{1-k_s,2-k_s,\dots,d-k_s\},
\quad B:=\{1-k_s,2-k_2,\dots,b-k_s\}, \\
S_j:=&\{k_j-k_s-c+2-\chi(u_j>i),\dots,k_j-k_s+c-1-\chi(u_j>i)\}
\end{align*}
for $j=1,\dots,r$ and $i\in \{1,\dots,n\}\setminus U$,
and
\begin{equation*}
S_j:=\{k_j-k_s-c+1+\chi(i\leq n_0)-\chi(u_j>i),\dots,k_j-k_s+c-\chi(i\leq n_0)-\chi(u_j>i)\}
\end{equation*}
for $j=r+1,\dots,s$ and $i\in \{1,\dots,n\}\setminus U$.
Assuming the conditions of the lemma, we show that the product
\[
\prod_{\substack{i=1\\ i\notin U}}^n \big(q^{1-k_s}x_i/x_{u_s}\big)_d
\]
in the denominator of $L'$ is in fact a factor of its numerator. Then $L'$ can be written as a Laurent polynomial, rather than a rational function.
To achieve this, it is equivalent to showing that
\begin{equation}\label{set}
S_0\subseteq \bigcup_{j=1}^sS_j\bigcup B.
\end{equation}
Since $S'_j:=\{k_j-k_s-c+2,\dots,k_j-k_s+c-2\}\subseteq S_j$ for $j=1,\dots,s$,
\eqref{set} holds if we show that
\begin{equation}\label{set2}
S_0\subseteq \bigcup_{j=1}^sS'_j\bigcup B.
\end{equation}
To obtain \eqref{set2}, it suffices to prove
\begin{enumerate}
\item $b\geq k_{w(1)}-c+1$;
\item $k_{w(s)}+c-2\geq d$;
\item $k_{w(j-1)}+c-2\geq k_{w(j)}-c+1$ for $j=2,\dots,s$.
\end{enumerate}

To prove (1), we need to find the upper bound for $k_{w(1)}$.
By \eqref{kwj} we have
\begin{equation}\label{sum}
\sum_{j=2}^s\big(k_{w(j)}-k_{w(j-1)}\big)=k_{w(s)}-k_{w(1)}
=(s-1)(c-1)+\sum_{j=2}^s\big(t_j+\chi(w(j)>r)\chi(w(j-1)>r)\big).
\end{equation}
Then
\begin{equation}\label{e-l92-2}
k_{w(1)}=k_{w(s)}-(s-1)(c-1)-\sum_{j=2}^s\big(t_j+\chi(w(j)>r)\chi(w(j-1)>r)\big).
\end{equation}
Since the $t_j$ satisfy Proposition~\ref{prop-lowerbound}, by \eqref{lowerbound1} we have
\begin{equation}\label{e-l92-1}
\sum_{j=2}^s\big(t_j+\chi(w(j)>r)\chi(w(j-1)>r)\big)\geq s-r-1
\end{equation}
for $r<s$. It is clear that the inequality \eqref{e-l92-1} also holds for $r=s$.
Hence, substituting \eqref{e-l92-1} into \eqref{e-l92-2} gives
\begin{equation}\label{e-192-3}
k_{w(1)}\leq k_{w(s)}-(s-1)(c-1)-(s-r-1).
\end{equation}
Since $r\leq n_0-2$ and $k_{w(s)}\leq d\leq s(c-1)+b+\chi(s>n_0)(s-n_0)$,
\[
k_{w(1)}\leq c-1+b+(s-n_0)-\big(s-(n_0-2)-1\big)=c+b-2
\]
for $s>n_0$. Thus (1) holds for $s>n_0$.
If $s\leq n_0$, then $k_{w(s)}\leq d\leq s(c-1)+b$. Substituting this into \eqref{e-l92-2} yields
\[
k_{w(1)}\leq  c-1+b-\sum_{j=2}^s\big(t_j+\chi(w(j)>r)\chi(w(j-1)>r)\big)\leq c-1+b.
\]
Consequently, (1) holds if $s\leq n_0$.
In conclusion, (1) holds for all the $s$.

To prove (2), we need to find the lower bound for $k_{w(s)}$.
By \eqref{sum}
\[
k_{w(s)}
=k_{w(1)}+(s-1)(c-1)+\sum_{j=2}^s\big(t_j+\chi(w(j)>r)\chi(w(j-1)>r)\big).
\]
Substituting \eqref{kw1} into the above equation gives
\[
k_{w(s)}
=b+(s-1)(c-1)+\sum_{j=1}^s\big(t_j+\chi(w(j)>r)\chi(w(j-1)>r)\big).
\]
If $s>n_0$ then by \eqref{sumbound} and $r\leq n_0-2$ we have
\[
k_{w(s)}
\geq b+(s-1)(c-1)+s-r\geq b+(s-1)(c-1)+s-n_0+2
=b+s(c-1)+s-n_0-c+3\geq d-c+3.
\]
Hence, (2) holds for $s>n_0$.
If $s\leq n_0$, then by the fact that $t_1\geq 1$ we have
\[
k_{w(s)}
\geq b+(s-1)(c-1)+1
=s(c-1)+b-c+2\geq d-c+2.
\]
Thus (2) also holds for $s\leq n_0$. In conclusion, (2) holds for all $s$.

If (3) fails for some $i\in \{2,\dots,s\}$, then
\[
k_{w(i)}-k_{w(i-1)}\geq 2(c-1).
\]
Together with \eqref{kwj} gives
\begin{equation}\label{kw-need}
k_{w(s)}-k_{w(1)}=\sum_{j=2}^s\big(k_{w(j)}-k_{w(j-1)}\big)
\geq s(c-1)+\sum_{\substack{j=2\\j\neq i}}^s\big(t_j+\chi(w(j)>r)\chi(w(j-1)>r)\big).
\end{equation}
Substituting \eqref{kw1} into the above equation, we have
\[
k_{w(s)}\geq s(c-1)+b+\sum_{\substack{j=1\\j\neq i}}^s\big(t_j+\chi(w(j)>r)\chi(w(j-1)>r)\big).
\]
We conclude contradictions by discussing the following cases:
\begin{itemize}
\item For $r<s$ and $s>n_0$, using \eqref{lowerbound1} again and $r\leq n_0-2$ yields
\[
k_{w(s)}\geq s(c-1)+b+s-r-1\geq s(c-1)+b+s-(n_0-2)-1=s(c-1)+b+s-n_0+1>d.
\]
This contradicts the assumption $k_i\leq d$ for all $i$.
\item If $r=s>n_0$ then it contradicts to the condition $r\leq \min\{n_0-2,s\}$ in the lemma.
\item If $s\leq n_0$ then $d\leq s(c-1)+b$. By \eqref{kw-need} and $k_{w(1)}=b+t_1\geq b+1$, we have
\[
k_{w(s)}\geq s(c-1)+k_{w(1)}+\sum_{\substack{j=2\\j\neq i}}^s\big(t_j+\chi(w(j)>r)\chi(w(j-1)>r)\big)
\geq s(c-1)+b+1>d.
\]
This also contradicts the assumption $k_i\leq d$ for all $i$.
\end{itemize}
Therefore, (3) holds.

Since (1)-(3) hold,  we obtain \eqref{set}.
It follows that $L'$ is a Laurent polynomial of the form
\begin{align*}
&\prod_{\substack{i=1\\ i\notin U}}^n
\frac{p_i(x_i/x_{u_s})}{(x_{u_s}/x_i)^d}
\prod_{\substack{i=1\\ i\notin U}}^n
\bigg(\prod_{j=1}^r(x_{u_s}/x_i)^{c-1}
\prod_{j=r+1}^s
(x_{u_s}/x_i)^{c-\chi(i\leq n_0)}\bigg) \\
&\ \ =\prod_{\substack{i=1\\ i\notin U}}^n
\bigg(\frac{p_i(x_i/x_{u_s})}{(x_{u_s}/x_i)^d}
\prod_{j=1}^s(x_{u_s}/x_i)^{c-1}
\prod_{j=r+1}^s
(x_{u_s}/x_i)^{\chi(i>n_0)}\bigg) \\
&\ \ =x_{u_s}^l\prod_{\substack{i=1\\ i\notin U}}^n
\Big(p_i(x_i/x_{u_s})x_i^{d-s(c-1)-\chi(i>n_0)(s-r)}\Big),
\end{align*}
where the $p_i(z)$ are polynomials in $z$, and
\[
l=(n-s)\big(s(c-1)-d\big)+(s-r)(n-n_0-s+r).\qedhere
\]
\end{proof}

By Lemma~\ref{lem-Laurent} and Lemma~\ref{lem-key}, we can give a proof of Lemma~\ref{lem-QLaurentpoly}.
\begin{proof}[Proof of Lemma~\ref{lem-QLaurentpoly}]
Recall that $Q(d\Mid u;k)$ can be written as the form in \eqref{Q-further}.
By Lemma~\ref{lem-key} with $t=c-1+\chi(s>n_0)(s-n_0)$, at least one of the following cases holds:
\begin{enumerate}
\item $1\leq k_i\leq b$ for some $i$ with $1\leq i\leq s$;
\item $-c+1\leq k_i-k_j\leq c-2$ for some $(i,j)$ such that $1\leq i<j\leq s$ and $i\leq r$;
\item $-c\leq k_{i}-k_{j}\leq c-1$ for some $(i,j)$ such that $r<i<j\leq s$;
\item there exists a permutation $w\in\mathfrak{S}_s$ and nonnegative integers $t_1,\dots,t_s$
such that
\begin{subequations}
\begin{equation*}
k_{w(1)}=b+t_1,
\end{equation*}
and
\begin{equation*}
k_{w(j)}-k_{w(j-1)}=c-1+\chi(w(j)>r)\chi(w(j-1)>r)+t_j \quad \text{for $2\leq j\leq s$.}
\end{equation*}
\end{subequations}
Here the $t_j$ satisfy
\begin{equation*}
s-r\leq \sum_{j=1}^{s}\big(t_j+\chi(w(j)>r)\chi(w(j-1)>r)\big)\leq c-1+\chi(s>n_0)(s-n_0),
\end{equation*}
$w(0):=0$, and $t_j>0$ if $w(j-1)<w(j)$ for $1\leq j\leq s$.
\end{enumerate}

If one of the cases (1)-(3) holds, then $V'$ defined in \eqref{V} is zero by a similar argument as that in the first part of the proof of Lemma~\ref{lem-Q}.
Since $V'$ is a factor of $Q(d\Mid u;k)$, we can conclude that $Q(d\Mid u;k)=0$.

If Case (4) holds and $0\leq r\leq \min\{n_0-2,s\}$,
then $Q(d\Mid u;k)$ can be written as the form of \eqref{Q-Laurent} by Lemma~\ref{lem-Laurent}.
Here $C,V'\in K$ can be put into the coefficients of some polynomial $p_i$,
and $H=E_{u,k}\big(x_{0}^{-\mathfrak{r}}\times M_{\mathfrak{r}}\big)$
is also of the form
$\prod_{\substack{i=1\\ i\notin U}}^np_i(x_i/x_{u_s})$.
\end{proof}

\subsection{Proof of Case (3) of Lemma~\ref{lem-Q}}\label{subsec-case3}

In this subsection, we give a proof of the Case (3) of Lemma~\ref{lem-Q}.
We present it in the next lemma.
\begin{lem}\label{Case3}
Assume the conditions of Lemma~\ref{lem-Q}.
If $s\neq n$ and
\begin{equation}\label{e-d}
s(c-1)+\chi(s>n_0)(s-n_0)+1\leq d\leq s(c-1)+(s-r)(n-n_0-s+r)/(n-s),
\end{equation}
then
\[
\CT_x Q(d\Mid u;k)=0.
\]
\end{lem}
\begin{proof}
Recall that $r$ is defined as
\[
r=\big|\{u_i\Mid u_i\leq n_0, i=1,2,\dots,s\}\big|.
\]
Then $0\leq r\leq \min\{n_0,s\}$.
We first show that $r\neq n_0-1,n_0$ in this lemma.

Since
\[
\frac{(s-r)(n-n_0-s+r)}{n-s}\bigg|_{r=n_0-1}
=\frac{(s-n_0+1)(n-s-1)}{n-s}<s-n_0+1\leq \chi(s>n_0)(s-n_0)+1,
\]
the lower bound of $d$ exceeds its upper bound in \eqref{e-d} for $r=n_0-1$.
Hence, $r\neq n_0-1$.
Since
\[
\frac{(s-r)(n-n_0-s+r)}{n-s}\bigg|_{r=n_0}
=s-n_0<\chi(s>n_0)(s-n_0)+1,
\]
the lower bound of $d$ also exceeds its upper bound in \eqref{e-d} for $r=n_0$.
Hence, $r\neq n_0$.
Therefore, we conclude that \eqref{e-d} implies $0\leq r\leq n_0-2$. Then $0\leq r\leq \min\{n_0-2,s\}$.

Since $0\leq r\leq \min\{n_0-2,s\}$, we can apply  Lemma~\ref{lem-QLaurentpoly} to obtain that
$Q(d\Mid u;k)$ is either 0, or can be written as a Laurent polynomial of the form
\begin{equation*}
L:=x_{u_s}^l\prod_{\substack{i=1\\ i\notin U}}^n
\Big(p_i(x_i/x_{u_s})x_i^{d-s(c-1)-\chi(i>n_0)(s-r)}\Big)
\prod_{\substack{1\leq i<j\leq n\\i,j\notin U}}\big(x_i/x_j\big)_{c-\chi(i\leq n_0)}
\big(qx_j/x_i\big)_{c-\chi(i\leq n_0)},
\end{equation*}
where $p_i(z)$ is a polynomial in $z$, $l=(n-s)\big(s(c-1)-d\big)+(s-r)(n-n_0-s+r)$ and $U=\{u_1,\dots,u_s\}$.
Notice that by taking the constant term of $L$ with respect to $x_{u_s}$,
we can write $\CT\limits_{x_{u_s}}L$ as a finite sum of the form
\begin{equation*}
P:=c\times\frac{\prod_{\substack{i=1\\ i\notin U}}^{n_0}
x_i^{d-s(c-1)}}
{\prod_{\substack{i=n_0+1\\ i\notin U}}^{n}
x_i^{s(c-1)-d+s-r}}
\prod_{\substack{i=1\\i\notin U}}^nx_i^{t_i}
\prod_{\substack{1\leq i<j\leq n\\i,j\notin U}}\big(x_i/x_j\big)_{c-\chi(i\leq n_0)}
\big(qx_j/x_i\big)_{c-\chi(i\leq n_0)},
\end{equation*}
where $c\in K$ and the $t_i$ are nonnegative integers such that
$\sum_{\substack{i=1\\i\notin U}}^nt_i=l$. We will show that $\CT\limits_{x}P=0$.

Let
\begin{align*}
(x_1,\dots,\hat{x}_{u_1},\dots,\hat{x}_{u_s},x_n)&\mapsto
z:=(z_1,\dots,z_{n-s}) \quad \text{and} \\
(t_1,\dots,\hat{t}_{u_1},\dots,\hat{t}_{u_s},t_n)&\mapsto
(t'_1,\dots,t'_{n-s}).
\end{align*}
Then
\begin{equation}\label{L2}
\CT_x P=c\times\CT_{z}\frac{\prod_{i=1}^{n_0-r}
z_i^{d-s(c-1)}}
{\prod_{i=n_0-r+1}^{n-s}
z_i^{s(c-1)-d+s-r}}
\prod_{i=1}^{n-s}z_i^{t'_i}
\prod_{1\leq i<j\leq n-s}\big(z_i/z_j\big)_{c-\chi(i\leq n_0-r)}
\big(qz_j/z_i\big)_{c-\chi(i\leq n_0-r)},
\end{equation}
where $\sum_{i=1}^{n-s}t'_i=l$.
If we further take
$s(c-1)-d+s-r\mapsto h$, $n-s\mapsto n'$ and $n_0-r\mapsto n'_0$ in \eqref{L2},
then
\begin{equation}\label{P}
\CT_x P
=c\times\CT_z\frac{\prod_{i=1}^{n'_0}z_i}
{\prod_{i=n'_0+1}^{n'}z_i^{h}}
\prod_{i=1}^{n'}z_i^{t'_i+d-s(c-1)-1}
\prod_{1\leq i<j\leq n'}\big(z_i/z_j\big)_{c-\chi(i\leq n'_0)}
\big(qz_j/z_i\big)_{c-\chi(i\leq n'_0)}.
\end{equation}
Note that these substitutions for $P$ do not change its constant term.
By the condition $s(c-1)+\chi(s>n_0)(s-n_0)+1\leq d$ in the lemma,
we have
\[
h=s(c-1)-d+s-r\leq -\chi(s>n_0)(s-n_0)-1+s-r\leq -(s-n_0)-1+s-r
=n_0-r-1=n'_0-1.
\]
For $d\leq s(c-1)+(s-r)(n-n_0-s+r)/(n-s)$ and $r\leq n_0-2$ we have
\[
h=s(c-1)-d+s-r\geq s-r-(s-r)(n-n_0-s+r)/(n-s)\geq s-r-(s-r)(n-s-2)/(n-s)>0.
\]
Then $1\leq h\leq n'_0-1$.
By $0\leq r\leq \min\{n_0-2,s\}$ and the assumption $2\leq n_0\leq n-1$,
we have $n_0'=n_0-r\geq 2$ and $n_0'=n_0-r\leq n_0-s\leq n-1-s=n'-1$.
It is clear that $d-s(c-1)>0$ by the condition $d\geq s(c-1)+\chi(s>n_0)(s-n_0)+1$ in the lemma.
Then $t'_i+d-s(c-1)-1\geq 0$ for all the $i$.
Hence, we can apply Lemma~\ref{lem-vaniconst2} to the right-hand side of \eqref{P}
and obtain $\CT\limits_xP=0$.
It follows that $\CT\limits_x Q(d\Mid u,k)=0$.
\end{proof}

\subsection*{Acknowledgements}

This work was supported by the National Natural Science Foundation of China (No. 12071311, 12171487).


\begin{thebibliography}{10}

\bibitem{AFLT}
V. A. Alba, V. A. Fateev, A. V. Litvinov and G. M. Tarnopolskiy, \emph{On combinatorial expansion of the conformal blocks arising from AGT conjecture}, Lett. Math. Phys. 98 (2011), 33--64.

\bibitem{ARW}
S. P. Albion, E. M. Rains and S. O. Warnaar, \emph{AFLT-type Selberg integrals}, Commun. Math. Phys. 388 (2021), 735--791.



\bibitem{Askey}
R. Askey, \emph{Some basic hypergeometric extensions of integrals of Selberg and Andrews}, SIAM J. Math. Anal. 11 (1980), 938--951.

\bibitem{BF}
T. H. Baker and P. J. Forrester, \emph{Generalizations of the $q$-Morris constant term identity}, J. Combin. Theory, Ser. A 81 (1998), 69--87.

\bibitem{cai}
T. W. Cai, \emph{Macdonald symmetric functions of rectangular shapes}, J. Combin. Theory Ser. A 128 (2014), 162--179.

\bibitem{FW}
P. J. Forrester and S. O. Warnaar, \emph{The importance of the Selberg integral}, Bull. Amer. Math. Soc. (N.S.) 45 (2008), 489--534.


\bibitem{GR}
G. Gasper and M. Rahman, \emph{Basic Hypergeometric Series}, Encyclopedia of Mathematics and its Applications, Vol. 35, second edition, Cambridge University Press, Cambridge, 2004.

\bibitem{Gessel-Lv-Xin-Zhou2008}
I. M. Gessel, L. Lv, G. Xin and Y. Zhou, \emph{A unified elementary approach to the Dyson, Morris, Aomoto,
and Forrester constant terms}, J. Combin. Theory Ser. A 115 (2008), 1417--1435.

\bibitem{GX}
I. M. Gessel and G. Xin, \emph{A short proof of the
Zeilberger--Bressoud $q$-Dyson theorem}, Proc. Amer. Math. Soc.
134 (2006), 2179--2187.


\bibitem{Habsieger}
L. Habsieger, \emph{Une $q$-int\'{e}grale de Selberg et Askey}, SIAM J. Math. Anal. 19 (1988) 1475--1489.

\bibitem{haglund}
J. Haglund, \emph{The $q,t$-Catalan Numbers and the Space of Diagonal Harmonics:
With an appendix on the combinatorics of Macdonald polynomials},
Univ. Lecture Ser., vol. 41, Amer. Math. Soc., Providence, RI, (2008).

\bibitem{Kad}
K. W. J. Kadell, \emph{A proof of Askey's conjectured $q$-analogue of Selberg's integral and a conjecture of Morris}, SIAM J. Math. Anal. 19 (1988), 969--986.

\bibitem{KNPV}
G. K\'{a}rolyi, Z. L. Nagy, F. V. Petrov, and V. Volkov, \emph{A new approach to constant term identities and Selberg-type integrals}, Adv. Math. 277 (2015), 252--282.

\bibitem{Lascoux}
A. Lascoux, \emph{Symmetric Functions and Combinatorial Operators on Polynomials},
CBMS Reg. Conf. Ser. Math., Vol. 99, Amer. Math. Soc., Providence, RI, 2003.

\bibitem{Lassalle}
M. Lassalle, \emph{A short proof of generalized Jacobi-Trudi expansions for Macdonald polynomials}, Contemp. Math. 417 (2006), 271--280.

\bibitem{LXZ}
L. Lv, G. Xin and Y. Zhou, \emph{A family of $q$-Dyson style
constant term identities}, J. Combin. Theory Ser. A 116
(2009), 12--29.

\bibitem{MacSMC}
I. G. Macdonald, \emph{A new class of symmetric functions}, Actes du 20e S\'eminaire Lotharingien, vol. 372/S-20, Publications I.R.M.A., Strasbourg, 1988, pp. 131--171.

\bibitem{Mac95}
I. G. Macdonald, \emph{Symmetric Functions and Hall Polynomials},
2nd edition, The Clarendon Press, Oxford University Press, 1995.

\bibitem{Morris1982}
W. G. Morris, \emph{Constant Term Identities for Finite and Affine Root System: Conjectures and Theorems}, Ph.D. thesis, Univ. Wisconsin--Madison, 1982.

\bibitem{RW}
E. M. Rains and S. O. Warnaar,
\emph{Bounded Littlewood identities},
Mem. Amer. Math. Soc., 270 (2021), No 1317, vii+115pp.

\bibitem{Selberg}
A. Selberg, \emph{Bemerkninger om et multipelt integral}, Norsk Mat. Tidsskr. 26 (1944), 71--78.

\bibitem{stembridge1987}
J.~R. Stembridge, \emph{First layer formulas for characters of
$SL(n,\mathbb{C})$}, Trans. Amer. Math. Soc. 299 (1987),
319--350.

\bibitem{Warnaar05}
 S. O. Warnaar, \emph{$q$-Selberg integrals and Macdonald polynomials}, Ramanujan J. 10 (2005), 237--268.

\bibitem{xinresidue}
G.~Xin, \emph{The Ring of {M}alcev--{N}eumann Series and the Residue Theorem},
Ph.D Thesis, Brandeis University, 2004.

\bibitem{xiniterate}
G.~Xin, \emph{A fast algorithm for {M}ac{M}ahon's partition
analysis},
  Electron. J. Combin. 11 (2004), R58, 20 pp.

\bibitem{XZ}
G. Xin and Y. Zhou, \emph{A Laurent series proof of the Habsieger--Kadell $q$-Morris identity},
Electron. J. Combin. 21 (2014), \#P3.38.

\bibitem{zeil-bres1985}
D. Zeilberger and D. M. Bressoud, \emph{A proof of Andrews'
$q$-Dyson conjecture}, Discrete Math. 54 (1985),
201--224.

\bibitem{Zhou}
Y. Zhou, \emph{On the $q$-Dyson orthogonality problem}, Adv. Appl. Math. 130 (2021), 102224.

\end{thebibliography}
\end{document}